\documentclass[a4paper,11pt]{amsart}
\usepackage{amsmath,amsthm,amssymb}
\usepackage[latin1]{inputenc}
\usepackage[T1]{fontenc}
\usepackage[all]{xy}
\usepackage{mathrsfs}
\usepackage[colorlinks=true,linkcolor=blue,pagebackref,breaklinks]{hyperref}
\usepackage{mathabx}

\numberwithin{equation}{section}
\newtheorem{thm}[equation]{Theorem}
\newtheorem*{thmA*}{Theorem A}
\newtheorem*{thmB*}{Theorem B}
\newtheorem*{thmC*}{Theorem C}
\newtheorem*{thmD*}{Theorem D}
\newtheorem*{thmE*}{Theorem E}
\newtheorem{conj}[equation]{Conjecture}
\newtheorem{cor}[equation]{Corollary}
\newtheorem{lem}[equation]{Lemma}
\newtheorem{prop}[equation]{Proposition}
\newtheorem{hyp}[equation]{Hypothesis}
\newtheorem*{cons}{Consequence}
\newtheorem{df-lem}[subsection]{Definition-Lemma}

\theoremstyle{definition}
\newtheorem{aspt}[equation]{Convention}

\newtheorem{rem}[equation]{Remark}

\newtheorem{defn}[equation]{Definition}

\allowdisplaybreaks

\DeclareMathOperator{\h}{H}

\newcommand{\triv}{{\mathbf{1}}}
\def\R{\mathbb R}

\def\Z{\mathbb Z}
\def\A{\mathbb A}
\def\Q{\mathbb Q}
\def\C{\mathbb C}
\def\F{\mathbb F}
\def\E{\mathbb E}
\def\EE{\mathcal E}
\def\cF{\mathcal F}

\def\ira{\stackrel{\sim}{\longrightarrow}}
\def\hra{\hookrightarrow}
\def\ra{\rightarrow}

\def\g{\mathfrak g}
\def\gl{\mathfrak{gl}}

\def\a{\mathfrak a}

\def\G{\mathscr{G}}

\def\O{\mathcal O}
\def\h{\mathfrak h}

\def\k{\mathfrak k}

\def\<{\langle}
\def\>{\rangle}
\def\V{\mathcal{V}}
\def\W{\mathcal{W}}

\def\GL{{\rm GL}}

\def\cm{F}
\def\tr{F^{+}}
\def\Acm{\mathbb{A}_{F}}

\def\Atr{\mathbb{A}_{F^{+}}}

\begin{document}

\title[Special $L$-values and the refined GGP-conjecture]{Special values of $L$-functions and the refined Gan-Gross-Prasad conjecture}
\author{Harald Grobner \& Jie Lin}
\thanks{H.G. is supported by START-prize Y-966 and the Stand-alone research project P-32333, both sponsored by the Austrian Science Fund (FWF). J.L. was supported by the European Research Council under the European Community's Seventh Framework Programme (FP7/2007-2013) / ERC Grant agreement no. 290766 (AAMOT)}
\subjclass[2010]{11F67 (Primary) 11F70, 11G18, 11R39, 22E55 (Secondary). }

\maketitle

\begin{abstract}
We prove explicit rationality-results for Asai- $L$-functions, $L^S(s,\Pi',{\rm As}^\pm)$, and Rankin-Selberg $L$-functions, $L^S(s,\Pi\times\Pi')$, over arbitrary CM-fields $F$, relating critical values to explicit powers of $(2\pi i)$. Besides determining the contribution of archimedean zeta-integrals to our formulas as concrete powers of $(2\pi i)$, it is one of the advantages of our approach, that it applies to very general non-cuspidal isobaric automorphic representations $\Pi'$ of $\GL_n(\A_F)$. As an application, this enables us to establish a certain algebraic version of the Gan--Gross--Prasad conjecture, as refined by N.\ Harris, for totally definite unitary groups. As another application we obtain a generalization of a result of Harder--Raghuram on quotients of consecutive critical values, proved by them for totally real fields, and achieved here for arbitrary CM-fields $F$ and pairs $(\Pi,\Pi')$ of relative rank one.
\end{abstract}

\setcounter{tocdepth}{1}
\tableofcontents

\section*{Introduction}\label{intro}

\subsection*{Rationality for critical values}
In the algebraic theory of special values of $L$-functions, Deligne's conjecture for critical $L$-values of motives is still one of the driving forces. Cut down to one line, it asserts that the critical values at $s=m\in\Z$ of the $L$-function $L(s,\mathbb M)$ of a motive $\mathbb M$ can be described, up to multiplication by elements in a concrete number-field $E(\mathbb M)$, in terms of certain {\it geometric} period-invariants $c^\pm(\mathbb M)$ and certain explicit powers of $(2\pi i)$, \cite[Conj.\ 2.8]{deligne79}:
$$L(m,\mathbb M)\sim_{E( \mathbb M)} (2\pi i)^{d(m)}c^{(-1)^m}(\mathbb M).$$
In this generality, Deligne's conjecture is still far open. The deeper reason for this, though, seems almost like a paradox: It is tempting to believe that it is exactly the rigidity of the world of motives, which allows one to express critical values $L(m,\mathbb M)$ by such clear and basal invariants (namely $c^{(-1)^m}(\mathbb M)$, $E( \mathbb M)$ and $(2\pi i)^{d(m)}$), on the one hand, while it seems to be exactly the same rigidity of the world of motives, which does not leave enough argumentative room to attack Deligne's conjecture directly, on the other hand.\\\\
Yielding to this belief, it is hence not surprising that it has been the (much less rigid) automorphic side -- invoking the (conjectural) dictionary, hinging motives $\mathbb M$ over a number field $F$ and automorphic representations $\Pi$ of $\GL_n(\A_F)$, by a comparison of their $L$-functions -- where most progress on understanding the algebraic nature of special values of $L$-functions has been achieved.\\\\
Indeed, there is a growing series of results, relating critical values $s =\alpha+m$ (due to a basic shift of the argument $s$ now in $\alpha+\Z$) of an automorphic $L$-function $L(s,\Pi,r)$, up to multiplication by elements in a number field $E(\Pi)$ depending on $\Pi$, to certain {\it representation-theoretical} period invariants $p(\Pi)$ and a purely archimedean factor $p(m,\Pi_\infty,r)$. Obviously, interpreting Deligne's conjecture automorphically, here the period-invariant $p(\Pi)$ takes the role of $c^\pm(\mathbb M)$, the number field $E(\Pi)$ the role of $E(\mathbb M)$ and finally the archimedean factor $p(m,\Pi_\infty,r)$ the place of $(2\pi i)^{d(m)}$.\\\\
In many regards it is the latter archimedean factor $p(m,\Pi_\infty,r)$ (essentially the inverse of a weighted sum of archimedean zeta-integrals), which turns out to be the most mysterious ingredient: In fact, over several decades it has even been unknown if it is eventually zero (which would obviously have made all automorphic rationality-theorems meaningless) until -- after various important but partial results -- B.\ Sun established the non-vanishing of $p(m,\Pi_\infty,r)$ in great generality in breakthrough work. \\\\
However, apart from particular cases, an {\it explicit expression} for $p(m,\Pi_\infty,r)$, putting it in a precise relationship with its motivic counterpart $(2\pi i)^{d(m)}$, predicted by Deligne's conjecture, is yet to be found.\\\\
In this paper, we solve this problem, for Rankin-Selberg $L$-functions, $L^S(s,\Pi\times\Pi')$, and Asai- $L$-functions, $L^S(s,\Pi',{\rm As}^\pm)$, over arbitrary CM-fields $F$: We establish precise rationality-theorems, whose archimedean factors arise in a very natural way and are indeed explicit powers of $(2\pi i)$. As a general rule, these powers may be made fit with the powers predicted by Deligne, see Rem.\ \ref{rem:Deligne}.

\subsection*{Main results I: Rationality for Rankin-Selberg $L$-functions with explicit archimedean factors}
Our rationality-results apply to a large class of automorphic representations $\Pi$ and $\Pi'$. More precisely, we let $F$ be any CM-field and $\Pi$ a cuspidal automorphic representation of $\GL_n(\A_F)$, whereas $\Pi'=\Pi_1\boxplus...\boxplus\Pi_k$ may even be an isobaric sum on $\GL_{n-1}(\A_F)$, fully induced from an arbitrary number $k\geq 1$ of distinct, but again arbitrary, unitary cuspidal automorphic representations $\Pi_i$. \\\\
Let $s=\tfrac12+m$ be a critical point of $L^S(s,\Pi\times\Pi')$. Clearly, in order to explicitly determine the archimedean factor (i.e., the contribution of the archimedean zeta integrals to our formulas), we have to specify our possible choices of $\Pi_\infty$ and $\Pi'_\infty$: If $m\neq 0$, the only condition they have to satisfy is to be unitary with non-vanishing relative Lie algebra cohomology with respect to an irreducible algebraic coefficient module $\EE_\mu$, respectively $\EE_{\mu'}$, allowing a non-trivial $\GL_{n-1}(F\otimes_\Q\R)$-intertwining $\EE_\mu\otimes\EE_{\mu'}\ra\C$.\\\\
The special case $m=0$, i.e., to obtain a rationality-result with explicit powers of $(2\pi i)$ for the central critical value $L^S(\tfrac12,\Pi\times\Pi')$, is more complicated by nature and needs an additional non-vanishing assumption on the central critical value of some auxiliary representations, constructed from suitable Hecke characters, see Hyp.\ \ref{hyp a 0}, \ref{hyp a 1} \& \ref{cond} for our precise assumptions in this case. We indeed expect our hypotheses to hold in complete generality: Evidence for this expectation is provided by \cite{jiang-zhang}, \cite{eischen} and \cite{ginz-jiang-rallis}, but also and more originally by \cite{rohrlich}: We refer to \S \ref{sect:nonvan} for a more detailed discussion and explanations on our hypotheses. Here we only remark that Hyp.\ \ref{hyp a 0} \& \ref{hyp a 1}, can be dropped, for instance, if $\EE_\mu$ and $\EE_{\mu'}$ are sufficiently regular, i.e., the successive coordinates of $\mu$ and $\mu'$ differ at least by $2$.\\\\
Here is our first main theorem, relating critical values of $L^S(s,\Pi\times\Pi')$ with explicit powers of $(2\pi i)$:

\begin{thmA*}
Let $\Pi$ be a cuspidal automorphic representation of $\GL_n(\A_F)$, and let $\Pi'=\Pi_1\boxplus...\boxplus\Pi_k$ be an isobaric automorphic representation of $\GL_{n-1}(\A_F)$, fully induced from an arbitrary number $k\geq 1$ of distinct unitary cuspidal automorphic representations $\Pi_i$ of $\GL_{n_i}(\A_F)$ and write $\mathcal G(\omega_{\Pi'_f})$ for the Gau\ss{}-sum of its central character. Assume that $\Pi_\infty$ and $\Pi'_\infty$ are conjugate self-dual, cohomological with respect to an irreducible algebraic coefficient module $\EE_\mu$, respectively $\EE_{\mu'}$, allowing a non-trivial $\GL_{n-1}(F\otimes_\Q\C)$-intertwining $\EE_\mu\otimes\EE_{\mu'}\ra\C$. Let $s=\tfrac12+m$ be a critical point of $L^S(s,\Pi\times\Pi')$, where, if $m=0$, we assume the auxiliary non-vanishing hypotheses Hyp.\ \ref{hyp a 0}, \ref{hyp a 1} \& \ref{cond} mentioned above.

Then there are non-zero Whittaker periods $p(\Pi)\in\C^\times$ and $p(\Pi')\in\C^\times$, defined by a comparison of a fixed rational structure on the Whittaker model of $\Pi_f$, resp.\ $\Pi'_f$, with a fixed rational structure on the cohomology of $\Pi$, resp.\ $\Pi'$, and we obtain
\begin{equation}\label{thma}
L^{S}(\tfrac 12+m,\Pi \times \Pi') \ \sim_{E(\Pi)E(\Pi')}(2\pi i)^{mdn(n-1)-\frac{1}{2}d(n-1)(n-2)} p(\Pi)\ p(\Pi')\ \mathcal G(\omega_{\Pi'_f})
\end{equation}
which is equivariant under the natural action of $\textrm{\emph{Aut}}(\C/F^{Gal})$. Here, ``$\sim_{E(\Pi)E(\Pi')}$'' means up to multiplication by an element in the number field ${E(\Pi)E(\Pi')}$ obtained by composing the Galois closure $F^{Gal}$ of $F/\Q$ in $\bar\Q$ with the fields of rationality of $\Pi$ resp.\ $\Pi'$.
\end{thmA*}
Being able to determine the contribution of the archimedean zeta-integrals for the first time as an explicit power of $(2\pi i)$, our Thm.\ A may be regarded as a joint refinement of \cite[Thm. 1.9]{grob-app}, \cite[Thm.\ 3.9]{grob-harr} and \cite[Thm.\ 1.1]{ragh16} over general CM-fields $F$. We remark that the presence of $F^{Gal}$ in our formula(s) is indispensable due to the use of our ``Minimizing--Lemma'', cf.\ Lem.\ \ref{minimize}: This is a useful tool, which allows to reduce relations of algebraicity to fields of minimal size, as long as they contain $F^{Gal}$.\\\\
The Whittaker periods $p(\Pi)$ and $p(\Pi')$ mentioned in Thm.\ A are constructed in Prop.\ \ref{prop:Whittakerperiods} in one go: Invoking several theorems on the nature of {\it non-cuspidal} automorphic cohomology, we are able to transfer the general principle of how to construct Whittaker periods, developed in \cite{hardermodsym}, \cite{mahnkopf} and in particular in \cite{raghuram-shahidi-imrn}, from cuspidal representations to general {\it Eisenstein representations}, i.e., (a slight generalization of) our general isobaric sums $\Pi'=\Pi_1\boxplus...\boxplus\Pi_k$. This is a non-trivial step in the construction of our Whittaker periods, which is established in \S\ref{mindeg} (see in particular Cor.\ \ref{cor:Whittakerperiods}) for which the fine analysis of the space of Eisenstein cohomology as achieved in \cite{grobner-EisRes} turns out to be indispensable. As a result, we obtain a uniform generalization of the construction in \cite{raghuram-shahidi-imrn}: Our generalization applies to arbitrary Eisenstein representations (which fully cover the case of a cuspidal representation by specifying $k=1$ in $\Pi'=\Pi_1\boxplus...\boxplus\Pi_k$). \\\\
{\it Important remark:} The careful reader, experienced with the construction of Whittaker periods $p(\Pi)$, will have noticed that this construction, as it is carried out in Prop.\ \ref{prop:Whittakerperiods}, involves the choice of a comparison isomorphism $\Upsilon_\Pi$, or, equivalently, the choice of a generator of the one-dimensional, archimedean cohomology space $H^{dn(n-1)/2}(\g_\infty,K_\infty,W(\Pi)_\infty\otimes \EE_\mu)$, cf.\ \S \ref{mindeg}. On the other hand, the number $L^{S}(\tfrac 12+m,\Pi \times \Pi')$ clearly does not involve any such archimedean choices. In this regard, the assertion of Thm.\ A may be reformulated in that there exists a natural choice of $\Upsilon_\Pi$ (and independently (!) $\Upsilon_{\Pi'}$) such that \eqref{thma} holds. This choice is made as follows (while we refer to our two Conventions \ref{conv:2} and \ref{diag}, to be found in \S\ref{sect:gens}, for all details): Firstly, by the uniqueness of archimedean smooth Whittaker models, one may assume that we fix the same generators of the one-dimensional, archimedean cohomology spaces $H^{dn(n-1)/2}(\g_\infty,K_\infty,W(\Pi)_\infty\otimes \EE_\mu)$ for all automorphic representations $\Pi$, which share isomorphic archimedean components, cf.\ Rem.\ \ref{rem:same}. Moreover, we recall that by construction, for all algebraic Hecke characters $\chi$ one may chose a canonical generator $[\chi_\infty]$ of $H^{0}(\g_{1,\infty},K_{1,\infty},W(\chi)_\infty\otimes \chi^{-1}_\infty)\cong\C$, such that $p(\chi)\in\Q(\chi)^\times$, see Lem.\ \ref{lem:chiq}. This assumption, made as Conv.\ \ref{conv:2} throughout the paper, fixes a canonical choice of $\Upsilon_\chi$ or all algebraic Hecke characters $\chi$. Finally, turning our attention back to cohomological Eisenstein representations $\Pi'=\Pi_1\boxplus...\boxplus\Pi_k$, one may assume that the choice of $\Upsilon_{\Pi'}$ is compatible with parabolic induction, i.e., that $\Upsilon_{\Pi'}$ is determined by the tensor product of the generators $[\chi_{j,\infty}]$, which define the Langlands datum of $\Pi'_\infty$: This condition, which is only sketched here, is made precise in \S\ref{sect:gens}. We refer to Conv.\ \ref{diag} for further explanations. It is then shown that in combination these assumptions in fact fix the embedding $\Upsilon_{\Pi'}$ for all twisted Eisenstein representations, i.e., in particular for all representations $\Pi$ and $\Pi'$, which satisfy the conditions of Thm.\ A. There is hence no hidden ambiguity in the statement of Thm.\ A.

\subsection*{Main results II: Rationality for Asai $L$-functions with explicit archimedean factors}
In order to explain our second main theorem, let $\Pi'= \Pi_1\boxplus...\boxplus\Pi_k$ be an Eisenstein representation of $\GL_n(\A_F)$ as above, i.e., a cohomological isobaric sum, fully induced from an arbitrary number $k\geq 1$ of distinct unitary cuspidal automorphic representations $\Pi_i$ of $\GL_{n_i}(\A_F)$. A short moment of thought shows that $\Pi_i$ is cohomological itself, if and only if $n\equiv n_i \mod 2$, cf.\ \S\ref{sect:Eisen}. Putting $e\in\{0,1\}$ equal to the residue class of $n-n_i\mod 2$ and $\eta:F^\times\backslash\A^\times_F\ra\C^\times$ equal to the extension of the quadratic character $\varepsilon$ attached to $F$ and its maximal totally real subfield $F^+$ by class field theory, cf.\ \S\ref{sect:char}, we end up with a unitary cuspidal representation $\Pi^{\sf alg}_i:=\Pi_i\otimes\eta^{e}$, which is cohomological with respect to an algebraic coefficient system $\EE_{\mu_i}$ in any case.\\\\
If $\Pi'$ is moreover conjugate self-dual, then one can show (cf.\ Cor.\ \ref{cor:asai-nonvan}) that the Asai $L$-function $L^S(s,\Pi',{\rm As}^{(-1)^{n}})$ of sign $(-1)^n$ is holomorphic and non-vanishing at $s=1$. Moreover, $s=1$ is critical for $L(s,\Pi',{\rm As}^{(-1)^{n}})$. Our second main theorem relates this critical value $L^S(1,\Pi',{\rm As}^{(-1)^{n}})$ with an explicit power of $(2\pi i)$:

\begin{thmB*}
Let $F$ be any CM-field and let $\Pi'=\Pi_1\boxplus...\boxplus\Pi_k$ be a cohomological isobaric automorphic representation of $\GL_{n}(\Acm)$, fully induced from an arbitrary number $k\geq 1$ of distinct conjugate self-dual cuspidal automorphic representations $\Pi_i$ of $\GL_{n_i}(\A_F)$. If $\EE_{\mu_i}$ is not sufficiently regular, we assume the auxiliary non-vanishing hypotheses Hyp.\ \ref{hyp a 0} \& \ref{hyp a 1} for $\Pi^{\sf alg}_i$. Then we have
$$
L^{S}(1,\Pi',{\rm As}^{(-1)^{n}})\ \sim_{E(\Pi')} \ (2\pi i)^{dn} p(\Pi')
$$
which is equivariant under the natural action of $\emph{\textrm{Aut}}(\C/F^{Gal})$.
\end{thmB*}
{\it Important remark:} As for Thm.\ A above, the Whittaker period $p(\Pi')$ depends on archimedean choices, whereas $L^{S}(1,\Pi',{\rm As}^{(-1)^{n}})$ obviously does not. However, the same remark as at the end of the last section applies also in this situation: Our standing assumptions on how we restrict ourselves in choosing generators of the one-dimensional cohomology space $H^{dn(n-1)/2}(\g_\infty,K_\infty,W(\Pi')_\infty\otimes \EE_\mu)$ turn out to be sufficiently restrictive in order for Thm.\ B to hold. Otherwise put, these assumptions on choices of generators already fix our possible choices for $\Upsilon_{\Pi'}$ for isobaric sums $\Pi'$ as in the statement of Thm.\ B (and hence $p(\Pi')$ up to multiplication by $\Q(\Pi')^\times$). This makes Thm.\ B into a meaningful statement without hidden ambiguities.

\subsection*{Main applications I: The refined conjecture of Gan--Gross-Prasad for unitary groups}

Combining Thm.\ A with Thm.\ B yields the following result, which is both, a generalization as well as a subtle refinement of \cite{grob-harr}, Cor.\ 6.25, with the additional asset that it avoids any reference to our global Whittaker periods $p(\Pi)$ and $p(\Pi')$:

\begin{thmC*}
Let $F$ be any CM-field and let $\Pi$ and $\Pi'$ be two cohomological conjugate self-dual automorphic representations of $\GL_n(\A_F)$, resp.\ $\GL_{n-1}(\A_F)$, which satisfy the conditions of Thm.\ A and Thm.\ B. Then, for every critical point $\tfrac12+m$ of $L(s,\Pi\times\Pi')$, we obtain
$$\frac{L^S(\tfrac12+m, \Pi\times\Pi')}{L^S(1,\Pi,{\rm As}^{(-1)^n}) \ L^S(1,\Pi',{\rm As}^{(-1)^{n-1}})} \sim_{E(\Pi)E(\Pi')} (2\pi i)^{mdn(n-1)-dn(n+1)/2}.$$
and this relation is equivariant under the natural action of ${\rm Aut}(\C/F^{Gal})$.
\end{thmC*}

We believe that this result, which holds for all critical points $\tfrac12+m$ of $L(s,\Pi\times\Pi')$, is interesting in its own right. Specifying $m=0$, however, we immediately obtain the relation
\begin{equation}\label{eq:ggp}
\frac{L^S(\tfrac12, \Pi\times\Pi')}{L^S(1,\Pi,{\rm As}^{(-1)^n}) \ L^S(1,\Pi',{\rm As}^{(-1)^{n-1}})} \sim_{E(\Pi)E(\Pi')} (2\pi i)^{-dn(n+1)/2},
\end{equation}
which leads us to the heart of the refined global Gan--Gross-Prasad conjecture, \cite{neil_harris, liu}, for unitary groups: Recall our arbitrary CM-field $F$ with maximal totally real subfield $F^+$ and the quadratic Hecke character $\varepsilon$ attached to the extension $F/F^+$. For unitary groups $\G(\V)/F^+$ and $\G(\W)/F^+$, attached to a pair of Hermitian spaces $\W\subset \V$ of dimension $\dim_F(\V)=n> \dim_F(\W)=m$, the global GGP-conjecture, as most recently refined by Liu, \cite{liu}, predicts a precise relationship of a quotient of $L$-functions, which is of the type of the left-hand-side of \eqref{eq:ggp}, and a global period integral $\mathcal P(\varphi,\varphi')$ of two tempered cusp forms $\varphi\in\pi$ and $\varphi'\in\pi'$ on $\G(\V)(\A_{F^+})$, reps.\ $\G(\W)(\A_{F^+})$:
\begin{equation}\label{eq:introggp}
|\mathcal P(\varphi,\varphi')|^2 = \frac{\Delta_{\G(\V)}}{2^a} \ \frac{L^S(\tfrac12, \pi\boxtimes\pi')}{L^S(1,\pi,{\rm Ad}) \ L^S(1,\pi',{\rm Ad})} \ \prod_{v\in S} \alpha_v(\varphi_v,\varphi'_v).
\end{equation}
Here, $\alpha_v(\varphi_v,\varphi'_v)$ are local integrals -- stabilized and suitably normalized -- over certain matrix coefficients, whereas $ \Delta_{\G(\V)}/2^a$ is a rather elementary constant, attached to the (expected) Vogan-Arthur packets of $\pi$ and $\pi'$ and the Gross-motives: $\Delta_{\G(\V)}=\prod_{i=1}^n L(i,\varepsilon_f^i)$. The careful reader, interested in precise definitions and assertions, is referred to \S\ref{sect:ggp}--\S\ref{sect:ref_ggp} for a detailed account.\\\\
Considering the classical case -- we refer to the original paper \cite{gp1}, where it all started; and \cite{neil_harris}, which anticipates the conjectures in \cite{liu}) -- when $\W$ has codimension $1$ in $\V$, i.e., $m=n-1$, provides the link of our explicit formula for quotients of critical $L$-values, \eqref{eq:ggp}, and the refined GGP-conjecture, \eqref{eq:introggp}, just pronounced: If $\pi$ and $\pi'$ are tempered cohomological cuspidal representations of $\G(\V)(\A_{F^+})$, reps.\ $\G(\W)(\A_{F^+})$, then we may apply quadratic base change $BC$ (unconditionally, as established in \cite{shin},\cite{KMSW}) and obtain two cohomological isobaric automorphic representations $BC(\pi)$ of $\GL_n(\A_F)$ and $BC(\pi')$ of $\GL_{n-1}(\A_F)$, respectively. If these new representations $BC(\pi)$ and $BC(\pi')$ satisfy the conditions of our Thm.\ A with Thm.\ B above, they may take the role of $\Pi$ and $\Pi'$ in \eqref{eq:ggp}, and so we may replace the quotient of $L$-functions in \eqref{eq:introggp} by the respective quotient of $L$-functions in \eqref{eq:ggp}; as moreover $\Delta_{\G(\V)}\sim_{F^{Gal}} (2\pi i)^{dn(n+1)/2},$ i.e., up to some algebraic number in $F^{Gal}$, the Gross-motives' factor $\Delta_{\G(\V)}$ in \eqref{eq:introggp} equals the inverse of the right-hand-side $(2\pi i)^{-dn(n+1)/2}$ of \eqref{eq:ggp}, we obtain the fundamental relation
\small
$$\frac{\Delta_{\G(\V)}}{2^a} \ \frac{L^S(\tfrac12, \pi\boxtimes\pi')}{L^S(1,\pi,{\rm Ad}) \ L^S(1,\pi',{\rm Ad})} \sim_{F^{Gal}}  \frac{(2\pi i)^{\tfrac{dn(n+1)}{2}} \ L^S(\tfrac12, \Pi\times\Pi')}{L^S(1,\Pi,{\rm As}^{(-1)^n}) \ L^S(1,\Pi',{\rm As}^{(-1)^{n-1}})}\sim_{E(\pi) E(\pi')} 1,$$
\normalsize
which finally enables us to show

\begin{thmD*}
Let $F$ be any CM-field with maximal totally real subfield $F^+$ and let $\G(V)$ and $\G(\W)$ be two arbitrary unitary groups over $F^+$ of codimension one. Let $\pi$ (resp.\ $\pi'$) be a tempered cohomological cuspidal automorphic representation of $\G(V)$ (resp.\ $\G(\W)$). Assume that the quadratic base change $BC(\pi)=\Pi$ is a cohomological cuspidal automorphic representation $\Pi$ of $\GL_n(\A_F)$ and that the quadratic base change $BC(\pi')=\Pi'$ is a cohomological isobaric automorphic representation of $\GL_{n-1}(\A_F)$ fully induced from an arbitrary number $k\geq 1$ of distinct cuspidal representations.
\begin{enumerate}
\item[(i)] If the pair $(\Pi,\Pi')$ satisfies the conditions of Thm.\ A and Thm.\ B, and if $\G(V)$ and $\G(\W)$ are moreover totally definite, then for all decomposable smooth $E(\pi)$-rational (resp.\ $E(\pi')$-rational) functions $\varphi=\otimes'_v\varphi_{v}\in\pi$ (resp.\ $\varphi'=\otimes'_v\varphi'_{v}\in\pi'$),
\begin{equation}\label{i}
|\mathcal P(\varphi,\varphi')|^2 \sim_{E(\pi) E(\pi')} \frac{\Delta_{\G(V)} \ L^S(\tfrac12, \pi\boxtimes\pi')}{L^S(1,\pi,{\rm Ad}) \ L^S(1,\pi',{\rm Ad})} \ \prod_{v\in S} \alpha_v(\varphi_v,\varphi'_v)
\end{equation}
where ``$\sim_{E(\pi) E(\pi')}$'' means up to multiplication by an element $q$ in the number field $E(\pi) E(\pi')$, depending only on $\pi$ and $\pi'$. This $q$ is in fact independent of $\varphi$ and $\varphi'$.
\item[(ii)]
If $\G(V)$ and $\G(\W)$ are not totally definite, but the respective coefficient modules in cohomology $\EE_\mu$ and $\EE_{\mu'}$ allow a non-trivial $\GL_{n-1}(F\otimes_\Q\C)$-intertwining $\EE_\mu\otimes\EE_{\mu'}\ra\C$, then the same conclusion as in (i) holds trivially for all decomposable cusp forms $\varphi=\otimes'_v\varphi_{v}\in\pi$ and $\varphi'=\otimes'_v\varphi'_{v}\in\pi'$.
\end{enumerate}
\end{thmD*}
We refer to \S\ref{final section} for a proof of Thm.\ D. Let us emphasize its two main advantages in the context of the recent iterature of the GGP-conjecture:
\begin{enumerate}
\item We do not assume any condition of local supercuspidality of $\pi\otimes\pi'$. In all preceding important work on the refined GGP-conjecture for unitary groups, which built on the trace formula, this assumption of supercuspidality has been indispensable (see \cite{zhang}, \cite{beuzard-plessis16}, \cite{beuzard-plessis18}, \cite{chau_zydor}, \cite{xue}). Here we completely avoid this condition, as well as any problems connected to the use of the fundamental lemma for the Jacquet-Rallis relative trace formulae.
\item We allow general isobaric sums for the base change of $\pi'$, i.e., we do not restrict ourselves to representations lifting to cuspidal representations. This restriction has been made in \cite{zhang}, Thm.\ 1.2.(2), for instance. 
\end{enumerate}
It is intrinsic to our approach via relations of algebraicity that our result cannot detect the non-vanishing of the left- and right-hand-side in \eqref{i}. If the quantities in \eqref{i} are non-zero, however, then our theorem asserts that both sides of \eqref{eq:introggp} are inside the same number field $E(\pi) E(\pi')$. 

\subsection*{Main applications II: A result of Harder--Raghuram}
Our main result on period relation of critical values of Rankin--Selberg $L$-functions with explicit powers of $(2\pi i)$, Thm.\ A, also provides a direct generalization of a result of Harder--Raghuram. Indeed, recapitulating their result very shortly, in \cite{harder-raghuram} a period, denoted $\Omega^{\varepsilon'}({}^\iota\sigma'_f)$, has been constructed and related to the ratio of consecutive critical values of Ranking--Selberg $L$-functions of cohomological cuspidal automorphic representations $\sigma$ and $\sigma'$ of $\GL_n(\A_{F^+})\times\GL_{n'}(\A_{F^+})$. Here, $n$ is assumed to be even while $n'$ is assumed to be odd. Here we prove

\begin{thmE*}
Let $F$ be any CM-field and let $\Pi$ be a cuspidal automorphic representation of $\GL_n(\A_F)$, and $\Pi'=\Pi_1\boxplus...\boxplus\Pi_k$ an isobaric automorphic representation of $\GL_{n-1}(\A_F)$, fully induced from an arbitrary number $k\geq 1$ of distinct unitary cuspidal automorphic representations $\Pi_i$. Assume that $\Pi_\infty$ and $\Pi'_\infty$ are cohomological with respect to an irreducible algebraic coefficient module $\EE_\mu$, respectively $\EE_{\mu'}$, allowing a non-trivial $\GL_{n-1}(F\otimes_\Q\C)$-intertwining $\EE_\mu\otimes\EE_{\mu'}\ra\C$. Let $\tfrac12+m,\tfrac12+\ell$ be two critical points of $L^S(s,\Pi\times\Pi')$, where, if $m\ell=0$, we assume the auxiliary non-vanishing hypotheses Hyp.\ \ref{hyp a 0}, \ref{hyp a 1} \& \ref{cond}. Whenever $L^S(\tfrac12+\ell,\Pi\times\Pi')$ is non-zero (e.g., if $\ell\neq 0$), we obtain
$$\frac{L^S(\tfrac12+m, \Pi\times\Pi')}{L^S(\tfrac12+\ell, \Pi\times\Pi')} \sim_{E(\Pi)E(\Pi')} (2\pi i)^{d(m-\ell)n(n-1)}.$$
and this relation is equivariant under the action of ${\rm Aut}(\C/F^{Gal})$. 
\end{thmE*}
This theorem also generalizes \cite{janus}, Thm.\ A, where an analogously explicit result has been proved (under different assumptions) for a pair of cuspidal automorphic representations. 

\medskip
\noindent{\it Acknowledgements: } \small
We would like to thank Michael Harris for many very valuable discussions about several results of this paper and his constant advice. We also thank Hang Xue and Rapha\"el Beuzart-Plessis for their helpful comments and for answering our questions on the Gan--Gross--Prasad conjecture. Finally, we would like to thank the anonymous referee for pointing out an intricate inaccuracy in a pervious version of our paper.
\normalsize

\section{Preliminaries}\label{sect:1}
\subsection{Number fields and Hecke characters} \label{sect:fields}
\subsubsection{Number fields}
Generally, if $\F\subset\C$ is any number field, then we denote by $J_\F$ the finite set of its field embeddings $\iota:\F\hra\C$ and by $\F^{Gal}$ the Galois closure of $\F/\Q$ in $\bar\Q$. More concretely, we let $\cm$ be any CM-field of dimension $2d = \dim_\Q \cm$ and set of archimedean places $S_\infty=S(\cm)_\infty$. Each place $v\in S_\infty$ hence refers to a fixed pair of conjugate complex embeddings $(\iota_v,\bar\iota_v)\in J_{\cm}^2$ of $\cm$, where we will drop the subscript ``$v$'' if it is clear from the context. This fixes a choice of a CM-type $\Sigma=\{\iota_v, v\in S_\infty\}$. We write $F^+$ for the maximal totally real subfield of $\cm$. Its set of real places will be identified with $S_\infty$, identifying a place $v$ with its first component embedding $\iota_v\in\Sigma$. Again, we may drop the subscript ``$v$'' if possible. We let Gal$(\cm/\tr)=\{1,c\}$. The ring of adeles over $\cm$ (resp.\ over $\tr$) is denoted $\A_{\cm}$ (resp.\ $\Atr$), their respective rings of integers $\O_{\cm}$ (resp.\ $\O_{\tr}$).  \\\\ Whenever we write $L^S$ for an object $L=\prod_{v}L_v$ admitting an Euler product factorization, then we mean the partial object $L^S:=\prod_{v \notin S}L_v$ for some choice of finite set of places $S$ of $F$, containing $S_\infty$. As a general rule, if $L$ depends on further data for which the notion of ramification is defined, we assume that $S$ contains all such ramified places.
\subsubsection{Characters and Gau\ss ~sums}\label{sect:char}
Let $\chi$ be any Hecke character of a CM-field. Following \cite{harris97} p.\ 82, we denote its contragredient conjugate dual by $\check{\chi}:=\chi^{-1,c}=\bar{\chi}^{\vee}$. The normalized absolute value on $\A_{\cm}$ is denoted $\|\cdot\|$. We extend the quadratic Hecke character $\varepsilon: (\tr)^\times\backslash \A^\times_{\tr}\ra\C^\times$, associated to $\cm/\tr$ via class field theory, to a conjugate self-dual unitary Hecke character $\eta: \cm^\times\backslash\A_{\cm}^\times\ra\C^\times$. At $v\in S_\infty$ we have $\eta_v(z)=z^t \bar z^{-t}$, for $z\in F_v$, where $t=t_v\in\tfrac12+\Z$. For our results there will be no loss of generality, if we assume from now on that $t=\tfrac12$, i.e., $\eta_v(z)=z^{1/2} \bar z^{-1/2}$. We may define a non-unitary algebraic Hecke character
$\phi: \cm^\times\backslash\A_{\cm}^\times\ra\C^\times$, by $\phi:=\eta\ \|\cdot\|^{1/2}$. We then have that $\phi\phi^c=\|\cdot\|$ and $\phi_v(z)=z^1 \bar z^0$ for all $v\in S_\infty$ and $z\in F_v$. Once and for all we fix a non-trivial additive character $\psi:{\cm}\backslash\A_{\cm}\ra\C^\times$ as in Tate's thesis, see, e.g., \cite{raghuram-shahidi-imrn}.\\
Let $\chi$ be an algebraic Hecke character. We define the Gau\ss ~sum of its finite part $\chi_f$, following Weil
\cite[VII, Sect.\ 7]{weil}: Let $\mathfrak{c}_\chi$ stand for the conductor ideal of
$\chi_f$ and let $y = (y_v)_{v \notin S_\infty} \in {\mathbb A}_f^{\times}$ be chosen such
that ${\rm ord}_v(y_v) = -{\rm ord}_v(\mathfrak{c}_\chi) -{\rm ord}_v(\mathfrak{D}_F)$. Here, $\mathfrak{D}_F$ stands
for the absolute different of $F$, that is, $\mathfrak{D}_F^{-1} = \{x \in F : Tr_{F/\Q}(x \O_F) \subset \Z\}$.\\\\
The Gau\ss ~sum of $\chi_f$ with respect to $y$ and $\psi$ is now defined as $\mathcal{G}(\chi_f,\psi_f,y) := \prod_{v \notin S_\infty} \mathcal{G}(\chi_v,\psi_v,y_v)$,
where the local Gau\ss ~sum $\mathcal{G}(\chi_v,\psi_v,y_v)$ is defined as
$$
\mathcal{G}(\chi_v,\psi_v,y_v) := \int_{\mathcal{O}_{F_v}^{\times}} \chi_v(u_v)^{-1}\psi_v(y_vu_v)\, du_v.
$$
For almost all $v$, we have $\mathcal{G}(\chi_v,\psi_v,y_v)=1$, and for all $v$ we have
$\mathcal{G}(\chi_v,\psi_v,y_v) \neq 0$. (See, for example, Godement \cite[Eq.\ 1.22]{godement}.)
Note that, unlike in \cite{weil}, we do not normalize the Gau\ss ~sum to make it have absolute
value one. For the sake of easing notation and readability we suppress its dependence on $\psi$ and $y$, and denote $\mathcal{G}(\chi_f,\psi_f,y)$ simply by $\mathcal{G}(\chi_f)$.

\subsection{Algebraic groups and real Lie groups} \label{sect:alggrp}
We let $G_n$, or simply $G$, be $G:=G_n:=\GL_n/{\cm}$. \\
Let $V_n$ be an $n$-dimensional, non-degenerate $c$-Hermitian space over $\cm$, $n\geq 2$, with corresponding unitary group $H:=H_n:=U(V_n)$ over $F^+$. At $v\in S_\infty$ (identified now with its first entry $v=\iota_v$), we have $H_n(\tr_v)\cong U(r_v,s_v)$ for some signature $0\leq r_v,s_v\leq n$. If $V_k$ is some non-degenerate $F$-subspace of $V_n$, we view $U(V_k)$ as a natural $\tr$-subgroup of $U(V_n)$.\\

If $\mathscr{G}$ is any reductive algebraic group over a number field $\F$, we write $\G_\infty=R_{\F/\Q}(\G)(\R)$. At $v\in S_\infty$ we denote by $K_v$ the product of the center $Z_G(\cm_v)$ of $G(\cm_v)$ and a fixed maximal compact subgroup of $G(\cm_v)$ (isomorphic to the compact real unitary group $U(n)$) and we let $K_\infty:=\prod_{v\in S_\infty}K_v\subset G_\infty$. If we want to emphasize the rank of $G=G_n$, we also write $K_{n,v}$ and $K_{n,\infty}$. Similarly, if $H$ is any given unitary group, we let $C_v$ be a fixed maximal compact subgroup of $H(\tr_v)$ (isomorphic to $U(r_v)\times U(s_v)$ and automatically containing the center $Z_H(\tr_v)\cong U(1)$ of $H(\tr_v)$) and we let $C_\infty:=\prod_{v\in S_\infty}C_v\subset H_\infty$.\\ Lower case gothic letters denote the Lie algebra of the corresponding real Lie group (e.g., $\g_v:=Lie(G(\cm_v))$, $\k_v:=Lie(K_v)$, $\h_v:=Lie(H(\tr_v))$, etc. ...).

\subsection{Highest weight modules and cohomological representations}\label{sect:finitereps}
We let $\EE_\mu$ be an irreducible finite-dimensional representation of the {\it real} Lie group $G_\infty=R_{\cm/\Q}(G)(\R)$ on a complex vector-space, given by its highest weight $\mu=(\mu_v)_{v\in S_\infty}$. Throughout this paper such a representation will be assumed to be algebraic: In terms of the standard choice of a maximal torus and positivity on the corresponding set of roots, this means that $\mu_v=(\mu_{\iota_v},\mu_{\bar\iota_v})\in \Z^n\times\Z^n$ and each component weight $\mu_{\iota_v}$ and $\mu_{\bar\iota_v}$ consists of a decreasing sequence of entries $\mu_{\iota_v,j}\geq \mu_{\iota_v,j+1}$ and $\mu_{\bar\iota_v,j}\geq \mu_{\bar\iota_v,j+1}$ for all $1\leq j \leq n-1$.\\

Similarly, given a unitary group $H=U(V_n)$, we let $\cF_\lambda$ be an irreducible finite-dimensional representation of the real Lie group $H_\infty=R_{\tr/\Q}(U(V_n))(\R)$ on a complex vector-space, given by its highest weight $\lambda=(\lambda_v)_{v\in S_\infty}$. Again, every such representation is assumed to be algebraic, which means that each component $\lambda_v\in\Z^n$. Moreover, one has $\lambda_{v,j}\geq \lambda_{v,j+1}$ for all $1\leq j \leq n-1$ and $v\in S_\infty$.\\

A representation $\Pi_\infty$ of $G_\infty$ is said to be {\it cohomological} if there is a highest weight module $\EE_\mu$ as above such that $H^*(\g_\infty,K_\infty,\Pi_\infty\otimes \EE_\mu)\neq 0$. Analogously, a representation $\pi_\infty$ of $H_\infty$ is said to be {\it cohomological} if there is a highest weight module $\cF_\lambda$ as above such that $H^*(\h_\infty,C_\infty,\pi_\infty\otimes \cF_\lambda)\neq 0$. See \cite{bowa},\S I, for details.

\subsubsection{An action of $\textrm{\emph{Aut}}(\C)$ on finite-dimensional representations}
Let $\G$ be a connected reductive group over $\Q$ and let $\EE$ be a finite-dimensional complex vector space on which $\G(\C)$ acts by linear transformations, i.e., there is a group homomorphism $\epsilon:\G(\C)\ra\GL(\EE)$. Given $\sigma\in$ ${\rm Aut}(\C)$, we may define a new linear action ${}^\sigma\!\epsilon:\G(\C)\ra\GL({}^\sigma\!\EE)$ as follows: Its underlying complex vector space is ${}^\sigma\!\EE:=\EE\otimes_{\sigma}\C$, (i.e., the same abelian group as the original space $\EE$, but with a new scalar multiplication $\alpha\star_\sigma v:=\sigma^{-1}(\alpha)\cdot v$) with linear action of $g\in \G(\C)$ defined by
$${}^\sigma\!\epsilon(g)v:= \epsilon(\sigma^{-1}(g)) v.$$
Here, we view $\G(\C)\subseteq\GL_N(\C)$ as being $\Q$-embedded into a (fixed) general linear group over $\Q$, whence applying $\sigma^{-1}$ to the complex matrix entries of $g$ gives rise to a well-defined element $\sigma^{-1}(g)\in\G(\C)$. As then $\sigma^{-1}(g)=g$ for all $g\in\G(\Q)$, this yields a $\sigma$-linear isomorphism of finite-dimensional $\G(\Q)$-representations
\begin{equation}\label{eq:gq2}
\tilde\sigma:\EE\ira {}^\sigma\!\EE
\end{equation}
Obviously, if $\epsilon$ was algebraic, then so is $^\sigma\!\epsilon$ for all $\sigma\in$ ${\rm Aut}(\C)$. If $(\epsilon,\EE)$ is furthermore an {\it irreducible} algebraic representation of $\G(\C)$, then the collection $\{({}^\sigma\!\epsilon, {}^\sigma\!\EE) \ : \ \sigma\in {\rm Aut}(\C) \}$ of equivalence classes of the representations $({}^\sigma\!\epsilon,{}^\sigma\!\EE)$ is finite. This follows from checking the effect of $\sigma$ on the highest weight of $\EE$, which, by assumption, defines an {\it algebraic} character. In particular, for irreducible algebraic representations $\EE$, the subgroup $\mathfrak S(\EE)$ of ${\rm Aut}(\C)$ consisting of all automorphisms $\sigma\in {\rm Aut}(\C)$ for which \eqref{eq:gq2} is an isomorphism of $\G(\Q)$-representations (i.e., linear), has finite index in ${\rm Aut}(\C)$. Hence, the rationality-field of the $\G(\Q)$-representation $\EE$, $\Q(\EE):=\C^{\mathfrak S(\EE)}$ is a number field.\\\\ The above construction applies in particular to irreducible algebraic representations $\EE_\mu$ of $G_\infty$ (resp.\ $\cF_\lambda$ of $H_\infty$) as defined in \S \ref{sect:finitereps}: Both $^\sigma\!\EE_\mu$ and $^\sigma\cF_\lambda$ will define a representation of $G_\infty$ (resp.\ $H_\infty$) by restriction from the respective group of complex points. 
\begin{rem}\label{rem:emu}
As a representation of $G(F)$, $^\sigma\!\EE_\mu$ may be identified with the representation $\EE_{^\sigma\!\mu}$ of highest weight
$^\sigma\!\mu=((\mu_{\sigma^{-1}\circ\iota_v},\mu_{\sigma^{-1}\circ\bar\iota_v})_{v\in S_\infty})$. The reader may be warned that the analogous assertion does not necessarily apply to the representation $^\sigma\cF_\lambda$ of $H(F^+)$.
\end{rem}

\subsubsection{Rational structures on algebraic representations and cohomology}\label{sect:rat}
We recall the following
\begin{defn}\label{def:rat}
An arbitrary (abstract) group representation $\rho$ on a complex vector space $V$ is said to be {\it defined over a subfield $\F\subseteq\C$} or {\it $\F$-rational}, if there exists an $\F$-subspace $V_\F\subseteq V$, which is stable under $\rho$ and such that the natural map $V_\F\otimes_\F \C\ra V$ is an isomorphism of complex vector spaces. The space $V_\F$ is referred to as an {\it $\F$-structure} of $\rho$.
\end{defn}
Let now $\G$ be any reductive algebraic group over $\Q$ which gives rise to a connected complex Lie group $\G(\C)$ and let $\epsilon: \G(\C)\ra \GL(\EE)$ be any irreducible algebraic representation of $\G(\C)$ on a finite-dimensional complex vector space $\EE$. Then we obtain the following

\begin{prop}\label{prop:emu}
Let $L$ be any finite Galois extension of $\Q$ over which $\G$ splits. Then, the restricted representation $\epsilon: \G(\Q)\ra \GL(\EE)$ is defined over the number field $L\cdot\Q(\EE)$.
\end{prop}

Let $\mathscr C_\infty$ be the product of the connected component of the identity of the center of $\G(\R)$ and a maximal compact subgroup of $\G(\R)$ (e.g., $K_\infty$ or $C_\infty$ from \S \ref{sect:alggrp}). The admissible $\G(\A_f)$-module defined by the cohomology group $H^q(S_\G,\EE)$ in degree $q$ of
$$S_\G:=\G(\Q)\backslash\G(\A_\Q)/\mathscr C_\infty$$
with respect to the locally constant sheaf on $S_\G$ given by $\EE$, hence inherits a natural $L\cdot\Q(\EE)$-rational structure from Prop.\ \ref{prop:emu}. Moreover, for each $\sigma\in$ ${\rm Aut}(\C)$, \eqref{eq:gq2} induces a $\sigma$-linear $\G(\A_f)$-equivariant bijection
\begin{equation}\label{eq:gq}
\tilde\sigma^q: H^q(S_\G,\EE) \ira H^q(S_\G,{}^\sigma\!\EE).
\end{equation}

\subsection{Automorphic representations $\pi$ and $\Pi$}
Throughout this paper, as a general rule, $\Pi$ denotes an irreducible automorphic representation of $G(\A_{\cm})=\GL_n(\A_{\cm})$, whereas $\pi$ denotes an irreducible automorphic representation of $H(\A_{\tr})=U(V_n)(\A_{\tr})$, in the sense of \cite{bojac}, \S 4, whose particular additional properties (such as being ``cuspidal'', ``unitary'', ``generic'', ``cohomological'', ``a quadratic base change'', etc.) will be specified at each of its occurrences. For convenience, however, we will not distinguish between a cuspidal automorphic representation, its smooth LF-space completion, cf.\ \cite{grob_zun}, or its (non-smooth) Hilbert space completion in the $L^2$-spectrum.

\subsubsection{Cohomological automorphic representations}\label{sect:pi}
Let $\Pi(r):=\Pi\cdot \|\!\det\!\|^{r}$, with $\Pi$ a unitary automorphic representation of $G(\A_{\cm})$ and ${r}\in\R$ (i.e, $\Pi(r)$ is essentially unitary automorphic), which is cohomological with respect to $\EE_\mu$. Suppose $\Pi(r)$ is generic at each $v\in S_\infty$, then
\begin{equation}\label{eq:ind}
\Pi(r)_v\cong {\rm Ind}_{B(\C)}^{G(\C)}[z_1^{\ell_{v,1}+{r}}\bar z^{-\ell_{v,1}+{r}}_1\otimes ...\otimes z_n^{\ell_{v,n}+{r}}\bar z^{-\ell_{v,n}+{r}}_n],
\end{equation}
where
$$\ell_{v,j}:=\ell(\mu_{\iota_v},j):=-\mu_{\iota_v,n-j+1}-{r}+\frac{n+1}{2}-j$$
and induction from the standard Borel subgroup $B=TN$ is unitary, cf.\ \cite[Thm.\ 6.1]{enright} (See also \cite[\S 5.5]{grob-ragh} for a detailed exposition). In particular, such a $\Pi(r)$ is essentially tempered at all $v\in S_\infty$. Observe that also the contrary holds: An essentially unitary automorphic representation of $G(\A_{\cm})$, which has an essentially tempered and cohomological archimedean factor $\Pi(r)_\infty$ is generic at all places $v\in S_\infty$. One has
\begin{equation}\label{eq:cohh}
H^q(\g_v,K_v,\Pi(r)_v\otimes\EE_{\mu_v})\cong \bigwedge^{q-\frac{n(n-1)}{2}}\C^{n-1}.
\end{equation}
Slightly more general, let us also consider the automorphic twists $\Pi(r)\phi^e$, with $e\in\{0,1\}$. Then it is easy to see from the very definition of $\phi$ that, $\Pi(r)$ is cohomological if and only if $\Pi(r)\phi$ is (the respective highest weight modules arise form each other by adding $-e$ to each entry of the $\iota_v$-components, $v\in S_\infty$). Indeed, the respective cohomology groups of both the representations $(\Pi(r)\phi^e)_v$, $e=0,1$, are isomorphic (and hence described by \eqref{eq:cohh}.)\\\\
For $\Pi(r)$ as above, we define the notion of {\it infinity-type} with repsect to our chosen CM-type $\Sigma$: Abbreviating
$a_{\iota,i}:=\ell(\mu_{\iota},i)+r$ and $a_{\bar\iota,i}:=-\ell(\mu_{\iota},i)+r$, the infinity-type of $\Pi(r)$ at $\iota\in\Sigma$ is the set of inducing characters $\{z^{a_{\iota,i}}\bar{z}^{a_{\bar\iota,i}}\}_{1\leq i\leq n}$. Hence, our notion of infinity-type recovers the ``{\it type \`a l'infini}'' defined in \cite{clozel}, \S 3.3, by substracting $\tfrac{n-1}{2}$ from each entry of the pairs $(a_{\iota,i},a_{\bar\iota,i})\in (\tfrac{n-1}{2}+\Z)^2$.\\\\
If $\pi$ is a unitary automorphic representation of $H(\A_{\tr})=U(V_n)(\A_{\tr})$, which is tempered and cohomological with respect to $\cF_\lambda$, then each of its archimedean component-representations $\pi_v$ of $U(r_v,s_v)$ is isomorphic to one of the ${n \choose r_v}$ inequivalent discrete series representations of infinitesimal character $\chi_{\lambda^{\sf v}_v+\rho_v}$, \cite{vozu}. 

\subsubsection{$\textrm{\emph{Aut}}(\C)$-twisted automorphic representations and attached number fields}\label{sect:twists}

Let $\pi$ (resp.\ $\Pi$) be a cohomological cuspidal automorphic representation of $H(\A_{\tr})$ (resp.\ $G(\A_{\cm})$). Then the $\sigma$-linear isomorphism \eqref{eq:gq} together with the well-known ``sandwich-property'', which stacks cuspidal, interior and square-integrable automorphic cohomology, see, e.g., \cite{schw1447}, p.\ 11, gives rise to a {\it $\sigma$-twisted} square-integrable cohomological automorphic representation $^\sigma\!\pi$ (resp.\ $^\sigma\Pi$) of $H(\A_{\tr})$ (resp.\ $G(\A_{\cm})$), whose finite component $({}^\sigma\!\pi)_f$ (resp.\ $({}^\sigma\Pi)_f$) allow an equivariant $\sigma$-linear isomorphism to $\pi_f$ (resp.\ $\Pi_f$). In terms of \cite{waldsp} \S I.1, this amounts to $({}^\sigma\!\pi)_f\cong {}^\sigma(\pi_f)$ (resp.\ $({}^\sigma\Pi)_f\cong{}^\sigma(\Pi_f)$). As prescribed by \eqref{eq:gq} the archimedean component $^\sigma\!\pi_\infty$ (resp.\ $^\sigma\Pi_\infty$) is unique up to $L$-packets (i.e., has predetermined infinitesimal character, namely the one of $^\sigma\!\EE_\mu^{\sf v}$, resp.\ $^\sigma\cF_\lambda^{\sf v}$).  
The $\sigma$-twists are cuspidal, if the original representation has no SL$_2(\C)$-factors in its global Arthur-parameter, \cite{arthur,mok}. 
In particular, $^\sigma\Pi$ will always be cuspidal and its archimedean component $^\sigma\Pi_\infty$ is uniquely determined, see \cite{grob-app} where this is made explicit. \\\\
As a consequence, the {\it rationality-fields} $\Q(\pi_f)$ and $\Q(\Pi_f)$ (defined as the fixed field in $\C$ of all automorphisms $\sigma$, which leave the finite part of the given automorphic representation stable, cf.\ \cite{waldsp}, \S I.1) are finite extensions of $\Q$ for all cohomological cuspidal automorphic representations $\pi$ and $\Pi$: For $\Pi$ this is proved implicitly in \cite{clozel}, Thm.\ 3.13 (and explicitly in \cite{grob-ragh} Thm.\ 8.1), while for $\pi$ the argument given in \cite{waldsp} Cor.\ I.8.3, \cite{grob-ragh} Thm.\ 8.1 or \cite{grobner-arithres}, Cor. 3.6 transfers verbatim. Since $\Pi$ satisfies Strong Multiplicity One, it is consistent to write $\Q(\Pi)=\Q(\Pi_f)$. Moreover, $\Q(\Pi_f)\supseteq\Q(\EE_\mu)$ for the same reason, see the proof of \cite{grob-ragh} Cor.\ 8.7. \\\\
Let now $E(\pi_f)$ be an appropriate finite extension of $\Q(\pi_f)$ over which $\pi_f$ is defined in the sense of \cite{waldsp}: That means that there is a $H(\A_f)$-stable $E(\pi_f)$-subspace $\pi_{E(\pi_f)}\subseteq\pi_f$ such that the natural map $\pi_{E(\pi_f)}\otimes_{E(\pi_f)}\C\ra\pi_f$ is an isomorphism. (For the existence of such an extension $E(\pi_f)$ and $\pi_{E(\pi_f)}$ see, e.g., \cite{grob-seb}, Thm.\ A.2.4). We will use the following abbreviations
\bigskip

\begin{center}
\framebox{$E(\pi):=\Q(\cF_\lambda)\cdot E(\pi_f) \cdot F^{Gal}, \quad\quad {\rm and} \quad\quad E(\Pi):=\Q(\Pi) \cdot F^{Gal},$}
\end{center}

\bigskip \noindent recalling that $F^{Gal}$ denotes the Galois closure of $F/\Q$ in $\bar\Q$. By what we have just said all these fields are number fields.

\subsubsection{Eisenstein representations}\label{sect:Eisen}
The automorphic representations of $\GL_n(\A_F)$, which we will mainly be considering in this paper, will be cohomological {\it Eisenstein representations}, i.e., cohomological automorphic representations of the form
$$\Pi'(r)\phi^e:=\Pi'\|\!\det\!\|^r\phi^e,\quad r\in\R, e\in\{0,1\}$$
where $\Pi'$ is an isobaric automorphic sum
\begin{equation}\label{rr}
\Pi':=\Pi_1\boxplus ... \boxplus\Pi_k\cong{\rm Ind}_{P(\A)}^{G_n(\A)}[\Pi_1\otimes...\otimes\Pi_k],
\end{equation}
which we assume to be fully-induced from distinct unitary cuspidal automorphic representations $\Pi_i$ of general linear groups $\GL_{n_i}(\A_{\cm})$,  $1\leq i\leq k$. Here, we let $P=L_PN_P$ be the standard parabolic subgroup of $G_n=\GL_n/{\cm}$ with Levi subgroup $L=L_P$ isomorphic to $\prod_{i=1}^k\GL_{n_i}$. As a paradigmatic example, a cohomological automorphic representation of $\GL_{n}(\A_{\cm})$, which is obtained by quadratic base change from a quasi-split unitary group as in \cite{ckpssh}, p.\ 122, is of the form of $\Pi'$, cf.\ \cite[Thm.\ 6.1]{ckpssh}. \\\\
Abbreviate $\tau:=\Pi_1\otimes...\otimes\Pi_k$, so by our conventions $\tau(r)\phi^e:=\Pi_1(r)\phi^e\otimes...\otimes\Pi_k(r)\phi^e$ is the cuspidal representation of $L(\A_F)$, whose parabolic induction is isomorphic to $\Pi'(r)\phi^e$ by assumption. It is worth noting, however, that the isomorphism in \eqref{rr} between the (only abstract!) global induced representation ${\rm Ind}_{P(\A_F)}^{G(\A_F)}[\tau(r)\phi^e]$ and the actual automorphic representation $\Pi'(r)\phi^e$ is given by computing the Eisenstein series $E_P(h,\lambda)$ attached to a $K_\infty$-finite section $h\in {\rm Ind}_{P(\A_F)}^{G(\A_F)}[\tau(r)\phi^e]$, formally defined as
$$E_P(h,\lambda)(g):=\sum_{\gamma\in P(F)\backslash G(F)} h(\gamma g)(id) \ e^{\<\lambda,H_P(\gamma g)\>},$$
$H_P$ begin the Harish-Chandra height function, $\lambda\in \check\a_{P,\C}:=X^*(L_P)\otimes_\Z\C$, followed by evaluating $E_P(h,\lambda)$ at the point $\lambda=0$, see \cite{langlands}, proof of Prop.\ 2. In fact, in \cite{langlands} a regularization $q(\lambda) E_P(h,\lambda)$ (the regularizing non-zero holomorphic function $q(\lambda)$ being defined as in \cite{moewal}, Lem.\ I.4.10, for instance) had to be used in order to obtain the desired isomorphism between $\Pi'(r)\phi^e$ and ${\rm Ind}_{P(\A_F)}^{G(\A_F)}[\tau(r)\phi^e]$. Hence our description of this latter isomorphism only follows knowing the following lemma:

\begin{lem}\label{lem:eisen}
For any $K_\infty$-finite section $h_{r,\phi^e}\in {\rm Ind}_{P(\A)}^{G(\A)}[\tau(r)\phi^e]$, the Eisenstein series $E_P(h_{r,\phi^e},\lambda)$ is holomorphic at the point of evaluation $\lambda=0$.
\end{lem}
\begin{proof}
Indeed, one has
\begin{eqnarray*}
E_P(h_{r,\phi^e},\lambda)(g) & = & \sum_{\gamma\in P(F)\backslash G(F)} h_{r,\phi^e}(\gamma g)(id) \ e^{\<\lambda,H_P(\gamma g)\>}\\
& = & \sum_{\gamma\in P(F)\backslash G(F)} h_{0,\eta^e}(\gamma g)(id) \ \|\!\det(\gamma g)\!\|^{r+e/2} \ e^{\<\lambda,H_P(\gamma g)\>} \\
& = & \|\!\det(g)\!\|^{r+e/2} \ \sum_{\gamma\in P(F)\backslash G(F)} h_{0,\eta^e}(\gamma g)(id) \ e^{\<\lambda,H_P(\gamma g)\>} \\
& = & \|\!\det(g)\!\|^{r+e/2} \ E_P(h_{0,\eta^e},\lambda)(g)
\end{eqnarray*}
for some $K_\infty$-finite $h_{0,\eta^e}\in {\rm Ind}_{P(\A)}^{G(\A)}[\tau\eta^e]$. As $\tau\eta^e$ is unitary, the latter Eisenstein series $E_P(h_{0,\eta^e},\lambda)(g)$ is holomorphic at $\lambda=0$ as it is well-known by \cite[IV.1.11]{moewal}.
\end{proof}
We denote the, finally well-defined isomorphism (of underlying $(\g_\infty, K_\infty, G(\A_f))$-modules) by
$${\rm Eis}_0: {\rm Ind}_{P(\A_F)}^{G(\A_F)}[\tau(r)\phi^e]\ira \Pi'(r)\phi^e$$
$$h\mapsto E_P(h,0).$$
It is due to this realization of $\Pi'(r)\phi^e$ in the space of cohomological automorphic forms $\mathcal A(G)$ that we have chosen the name Eisenstein representation. As it is obvious by its very definition, the family of Eisenstein representations contains all essentially unitary cohomological cuspidal automorphic representations $\Pi(r)$ (by letting $k=1$, $e=0$) and all cohomological fully-indiced isobaric sums $\Pi'$ (by letting $r=e=0$) as above. Moreover, now knowing the way we view $\Pi'(r)\phi^e$ as a subrepresentation of $\mathcal A(G)$ through the injection ${\rm Eis}_0$,  the arguments given in \cite[Prop.\ 7.1.3, Thm.\ 3.5.12 and Rem.\ 3.5.14]{shahidi-book} imply that $\Pi'(r)\phi^e$ is globally $\psi$-generic, i.e., the $\psi$-Fourier coefficient $W^\psi$ does not vanish on an Eisenstein representation $\Pi'(r)\phi^e$.\\\\
We write $\rho_i$ for the restriction $\rho_P|_{\GL_{n_i}(\A_F)}$, so we get explicitly $\rho_i=\|\det_{\GL_{n_i}}\|^{\tfrac{a_i}{2}}$, with
\begin{equation}\label{eq:ai}
a_i=\sum_{j=1}^{k-i}n_{i+j} - \sum_{j=1}^{i-1} n_j = n_{i+1}+n_{i+2}+...+n_k - n_1 - n_2 - ... - n_{i-1}.
\end{equation}
Hence, $a_i\equiv n-n_i \mod 2$ and so $\rho_i$ is algebraic if and only if $n\equiv n_i \mod 2$. As by assumption $\Pi'(r)\phi^e$ is cohomological, \cite{bowa}, Thm.\ III.3.3, implies that $\Pi_i(r)\phi^e\rho_i=\Pi_i(r+\tfrac{a_i}{2})\phi^e$ are cohomological for $1\leq i \leq k$. Hence, for all $\sigma\in$ ${\rm Aut}(\C)$, there are well-defined pairwise different, in general non-unitary cuspidal automorphic representations ${}^\sigma(\Pi_i(r)\phi^e \rho_i)$ of $\GL_{n_i}(\A_F)$, which are cohomological with respect to the $\sigma$-permuted coefficient module of $\GL_{n_i,\infty}$, see \S\ref{sect:twists}. We abbreviate
\begin{equation}\label{eq:isotwist}
(\Pi_i(r)\phi^e)^\sigma:={}^\sigma\!\left(\Pi_i(r)\phi^e \rho_i\right) \cdot \rho_i^{-1}.
\end{equation}
Then it is proved as in \cite{grob-app} \S 1.2.5 that
$${}^\sigma\Pi'(r)\phi^e:=(\Pi_1(r)\phi^e)^\sigma\boxplus ... \boxplus(\Pi_k(r)\phi^e)^\sigma$$
is an isobaric automorphic representation, which is again fully-induced from the pairwise different, in general non-unitary cuspidal automorphic representations $(\Pi_i(r)\phi^e)^\sigma$. If $\Pi'(r)\phi^e$ is cohomological with respect to $\EE_\mu$, then ${}^\sigma\Pi'(r)\phi^e$ is cohomological with respect to ${}^\sigma\!\EE_\mu$ and it satisfies $({}^\sigma\Pi'(r)\phi^e)_f\cong {}^\sigma((\Pi'(r)\phi^e)_f)$ for all $\sigma\in$ ${\rm Aut}(\C)$. As ${}^\sigma\Pi'(r)\phi^e\cong \Pi'(r)\phi^e$ if and only if
$$\{(\Pi_1(r)\phi^e)^\sigma,...,(\Pi_k(r)\phi^e)^\sigma\}=\{\Pi_1(r)\phi^e,...,\Pi_k(r)\phi^e\},$$
cf.\ \cite{jacshal2}, Thm.\ 4.4, the rationality field $\Q(\Pi'(r)\phi^e)$ of an Eisenstein representation is contained in the composition of number fields $\prod_{i=1}^k\Q(\Pi_i(r)\phi^e\rho_i)$ of the cohomological, cuspidal automorphic summands $\Pi_i(r)\phi^e\rho_i$, an hence a finite extension of $\Q$ itself.

\subsection{Whittaker periods for Eisenstein representations and critical values of automorphic $L$-functions}

\subsubsection{Abstract Whittaker periods and automorphic cohomology}\label{sect:whittaker}
Let $\Pi_f$ be an irreducible admissible representation of $G(\A_f)=\GL_n(\A_f)$ and suppose that ${}^\sigma(\Pi_f)$ is generic for all $\sigma\in$Aut$(\C$). Let $W({}^\sigma\Pi_f)$ be the Whittaker model with respect to our fixed non-trivial additive character $\psi:{\cm}\backslash\A_{\cm}\ra\C^\times$. For $\sigma\in$ ${\rm Aut}(\C)$ and $v\notin S_\infty$ let $t_{\sigma,v}$ be the unique diagonal matrix of type $t_{\sigma,v}={\rm diag}(x_1,...,x_{n-1},1)\in T(\O_{F_v})$ such that $\sigma(\psi_v(n))=\psi_v(t^{-1}_{\sigma,v}nt_{\sigma,v})$ for all $n\in N({\cm}_v)$ and let $t_\sigma=(t_{\sigma,v})_{v\notin S_\infty}\in G(\A_f)$. Then for every $\sigma\in$ ${\rm Aut}(\C)$ the $\sigma$-linear, $G(\A_f)$-equivariant bijection
$$\tilde\sigma_{\Pi_f}: {\rm Ind}_{N(\A_f)}^{G(\A_f)}[\psi_f]\ira {\rm Ind}_{N(\A_f)}^{G(\A_f)}[\psi_f]$$
$$\xi\mapsto \tilde\sigma_{\Pi_f}(\xi): g \mapsto \sigma(\xi(t_\sigma\cdot g))$$
will map $W(\Pi_f)$ onto $W({}^\sigma\Pi_f)$ by the uniqueness of local Whittaker models. See \cite[\S 3.2]{raghuram-shahidi-imrn} or \cite[\S 4.1]{grob_harris_lapid}. Hence, taking the space of Aut$(\C/\Q(\Pi_f))$-invariant vectors $W(\Pi_f)_{\Q(\Pi_f)}$ in $W(\Pi_f)$ will define a $\Q(\Pi_f)$-structure on $W(\Pi_f)$: Therefore, the only non-trivial observation needed is that $W(\Pi_f)_{\Q(\Pi_f)}$ contains the vector $\xi^\circ:=\otimes_v\xi_v$, where $\xi_v$ denotes the unique newvector such that $\xi_v(id_{G(F_v)})=1$ (cf.\ \cite[Thm.\ (5.1)]{jac-ps-shalika-mathann} and \cite[Cor.\ 4.4]{miyauchi}) and is hence non-zero.\\\\
Let $S_n:=S_G=G({\cm})\backslash G(\A_{\cm})/K_\infty.$ Then $H^q(S_n,\EE_\mu)$, the admissible $G(\A_f)$-module defined by the cohomology of $S_n$, has a natural $L$-rational structure over any field extension $\C\supseteq L\supseteq\Q(\EE_\mu)$. See \S\ref{sect:rat} (with $\G=R_{F/\Q}(G)$) above. Let $H^q(S_n,\EE_\mu)_L$ be this natural $L$-structure and recall the $\sigma$-linear $G(\A_f)$-equivariant bijections $\tilde\sigma^q:H^q(S_n,\EE_\mu)\ra H^q(S_n,^\sigma\!\EE_\mu)$ for all $\sigma\in$ ${\rm Aut}(\C)$ from \eqref{eq:gq}.\\\\
Suppose that $H_\bullet^q(S_n,\EE_\mu)$ is a submodule of $H^q(S_n,\EE_\mu)$. 
Obviously, if ${\rm Hom}_{G(\A_f)}(W(\Pi_f),H_\bullet^q(S_n,\EE_\mu))$ is one-dimensional, then the image $\Upsilon(W(\Pi_f))=:H^q(S_n,\EE_\mu)(\Pi_f)$ of a non-zero homomorphism $\Upsilon\in {\rm Hom}_{G(\A_f)}(W(\Pi_f),H_\bullet^q(S_n,\EE_\mu))$ is independent of $\Upsilon$ and isomorphic to $\Pi_f$.

\begin{prop}\label{prop:Whittakerperiods}
Let $\Pi_f$ be the finite part of an irreducible admissible representation $\Pi$ of $G(\A_F)=\GL_n(\A_F)$, such that ${}^\sigma(\Pi_f)$ is generic for all $\sigma\in {\rm Aut}(\C)$. Assume that there is an irreducible algebraic coefficient module $\EE_\mu$ and a family of submodules $H^q_{{}^\sigma\Pi_f}(S_n,{}^\sigma\!\EE_\mu)\subseteq H^q(S_n,{}^\sigma\!\EE_\mu)$, $\sigma\in {\rm Aut}(\C)$, only depending on the isomorphism classes of ${}^\sigma\Pi_f$ and ${}^\sigma\!\EE_\mu$, such that $H_{{}^\sigma\Pi_f}^q(S_n,{}^\sigma\!\EE_\mu)=\tilde\sigma^q(H_{\Pi_f}^q(S_n,\EE_\mu))$ and $\dim{\rm Hom}_{G(\A_f)}(W({}^\sigma\Pi_f),H_{{}^\sigma\Pi_f}^q(S_n,{}^\sigma\!\EE_\mu))=1$. Then, for all $\sigma\in {\rm Aut}(\C)$, the following hold:
\begin{enumerate}
\item Let ${}^\sigma\! L:=\Q({}^\sigma\!\EE_\mu) \ \Q({}^\sigma \Pi_f)$. The space
$$H^q(S_n,{}^\sigma\!\EE_\mu)({}^\sigma\Pi_f)_{{}^\sigma\!L}:=H^q(S_n,{}^\sigma\!\EE_\mu)({}^\sigma\Pi_f)\cap H^q(S_n,{}^\sigma\!\EE_\mu)_{{}^\sigma\! L}$$
defines an ${}^\sigma\! L$-structure on $H^q(S_n,{}^\sigma\!\EE_\mu)({}^\sigma\Pi_f)$.
\item  For each $\sigma\in {\rm Aut}(\C)$, choose embeddings ${}^\sigma\Upsilon\in {\rm Hom}_{G(\A_f)}(W({}^\sigma\Pi_f),H_{{}^\sigma\Pi_f}^q(S_n,{}^\sigma\!\EE_\mu))$ subject to the condition that $\Upsilon={}^\sigma\Upsilon$, whenever $\Pi_f\cong{}^\sigma\Pi_f$. Then there are non-zero complex numbers $p({}^\sigma\Pi):=p({}^\sigma\Pi_f,{}^\sigma\Upsilon)$, depending on a chosen embedding ${}^\sigma\Upsilon$ and the chosen, fixed ${}^\sigma\!L$-structures on its domain and image, such that
\begin{equation}\label{diawhit}
\xymatrix{
W(\Pi_f) \ar[rrr]^{p(\Pi)^{-1}\cdot\Upsilon}\ar[d]^{\tilde\sigma_{\Pi_f}}  & & &
H^q(S_n,\EE_\mu)(\Pi_f)\ar[d]^{\tilde\sigma^q} \\
W({}^\sigma\Pi_f)
\ar[rrr]^{p({}^\sigma\Pi)^{-1}\cdot{}^\sigma\Upsilon}
 & & &
H^q(S_n,{}^\sigma\!\EE_\mu)({}^\sigma\Pi_f)
}
\end{equation}
commutes. Replacing ${}^\sigma\Upsilon$ by another embedding ${}^\sigma\Upsilon'$ (which by the one-dimensionality of ${\rm Hom}_{G(\A_f)}(W({}^\sigma\Pi_f),H_{{}^\sigma\Pi_f}^q(S_n,{}^\sigma\!\EE_\mu))$ necessarily satisfies ${}^\sigma\Upsilon' = t\cdot {}^\sigma\Upsilon$ for some $t\in\C^\times$) one obtains $p({}^\sigma\Pi_f,{}^\sigma\Upsilon') = t\cdot p({}^\sigma\Pi_f,{}^\sigma\Upsilon)$, i.e., $p({}^\sigma\Pi)$ is multiplied by the same non-zero complex number $t$.
\end{enumerate}
\end{prop}
\begin{proof}
One has ${}^\sigma(H^q(S_n,\EE_\mu))=H^q(S_n,{}^\sigma\!\EE_\mu)=\tilde\sigma^q(H^q(S_n,\EE_\mu))$ by definition of $^\sigma\!\EE_\mu$, hence we obtain ${}^\sigma(H^q(S_n,\EE_\mu)(\Pi_f))=\tilde\sigma^q(H^q(S_n,\EE_\mu)(\Pi_f))$. As $H^q_{{}^\sigma\Pi_f}(S_n,{}^\sigma\!\EE_\mu)=\tilde\sigma^q(H_{\Pi_f}^q(S_n,\EE_\mu))$, the modules $H^q(S_n,{}^\sigma\!\EE_\mu)({}^\sigma\Pi_f)$ and $\tilde\sigma^q(H^q(S_n,\EE_\mu)(\Pi_f))$ are both submodules of $H^q_{{}^\sigma\Pi_f}(S_n,{}^\sigma\!\EE_\mu)$, which are isomorphic to ${}^\sigma\Pi_f$, hence equal due to our assumption that $\dim{\rm Hom}_{G(\A_f)}(W({}^\sigma\Pi_f),H_{{}^\sigma\Pi_f}^q(S_n,{}^\sigma\!\EE_\mu))=1$. Therefore, ${}^\sigma(H^q(S_n,\EE_\mu)(\Pi_f))=H^q(S_n,{}^\sigma\!\EE_\mu)({}^\sigma\Pi_f)$. As $H^q(S_n,{}^\sigma\!\EE_\mu)({}^\sigma\Pi_f)$ only depends on the isomorphism classes of ${}^\sigma\Pi_f$ and ${}^\sigma\!\EE_\mu$, we have ${}^\sigma(H^q(S_n,\EE_\mu)(\Pi_f))= H^q(S_n,\EE_\mu)(\Pi_f)$ whenever these isomorphism classes are stabilized by $\sigma$, i.e., when $\sigma\in$ Aut$(\C/\Q(\EE_\mu)\Q(\Pi_f))$. Therefore, (1) follows from \cite{clozel}, Lem.\ 3.2.1.\\
Having fixed an ${}^\sigma\! L$-structure on the irreducible $G(\A_f)$-module $H^q(S_n,{}^\sigma\!\EE_\mu)({}^\sigma\Pi_f)$, the existence of the non-zero numbers $p({}^\sigma\Pi)=p({}^\sigma\Pi_f,{}^\sigma\Upsilon)$ as well as the commutativity the above square follows word for word as in \cite[Definition/Proposition 4.2.1]{raghuram-shahidi-imrn}, replacing what is called $w^0$ there by our $\xi^\circ$ from above.
\end{proof}

\begin{defn}
Let $\Pi$ be any representation of $G(\A_{\cm})=\GL_n(\A_{\cm})$ as in Prop.\ \ref{prop:Whittakerperiods}. We call the non-zero complex number $p(\Pi)$ the {\it Whittaker period attached to $\Pi$} (and some fixed choice of embedding $\Upsilon$).
\end{defn}

\begin{rem}\label{rem:upto}
By construction, even for fixed $\Upsilon$, $p(\Pi)$ is well-defined only up to multiplication with elements in $ L^\times$. In fact, changing $p(\Pi)$ to $q\cdot p(\Pi)$ for some $q\in L^\times$ will change $p({}^\sigma\Pi)$ to $\sigma(q)\cdot p({}^\sigma\Pi)$: See the last line in the proof of \cite{raghuram-shahidi-imrn}, Def./Prop. 3.3 and recall that ${}^\sigma\! L = \sigma(L)$.
\end{rem}

\subsubsection{An important family of examples - The Eisenstein representations}\label{mindeg}
Let us now make the above proposition concrete for a large family of automorphic representations. More precisely, we consider our family of Eisenstein representations
$$\Pi'(r)\phi^e=\Pi'\|\!\det\!\|^r\phi^e,\quad r\in\R, e\in\{0,1\}$$
as defined in \S\ref{sect:Eisen} above. Let $\EE_\mu$ be the irreducible algebraic representation of $G_{\infty}$ with respect to which $\Pi'(r)\phi^e$ is cohomological. Since $\Pi'(r)\phi^e$ is globally generic, $(\Pi'(r)\phi^e)_\infty$ is of the form described in \S\ref{sect:pi} and has one-dimensional $(\g_\infty,K_\infty)$-cohomology in its lowest non-vanishing degree
$$b_n:=d\frac{n(n-1)}{2} $$ by \eqref{eq:cohh}. Therefore, the $G(\A_f)$-module
\begin{equation}\label{eq:onedim}
H^{b_n}(\g_\infty, K_\infty, \Pi'(r)\phi^e\otimes\EE_\mu)\cong (\Pi'(r)\phi^e)_f
\end{equation}
is irreducible. As the choice of our Whittaker period $p(\Pi'(r)\phi^e)$ all depends on the particular map $\Upsilon$, merging the irreducible $G(\A_f)$-representations $H^{b_n}(\g_\infty, K_\infty, \Pi'(r)\phi^e\otimes\EE_\mu)$ and $W((\Pi'(r)\phi^e)_f)$ as the same submodule of $H^q(S_n,\EE_\mu)$, the various choices made entering the definition of $\Upsilon$ shall now be specified. Knowing these choices explicitly will be of great significance in \S\ref{s:Eisen}. \\\\
Firstly, recall that we have realized (the space of $K_\infty$-finite vectors of) $\Pi'(r)\phi^e$ as a submodule of the space $\mathcal A(G)$ of automorphic forms, namely as the image of the isomorphism ${\rm Eis}_0$ from \S\ref{sect:Eisen}. In view of Prop.\ \ref{prop:Whittakerperiods} we shall refine our description. To this end, let $S(\check\a^{G}_{P,\C})$ be the symmetric algebra of the orthogonal complement $\check\a^{G}_{P,\C}$ of $\check\a_{G,\C}=X^*(G)\otimes_\Z\C=\{\det^s, s\in\C\}$ in $\check\a_{P,\C}$. It may be interpreted as the algebra of differential operators $\partial^{\underline n}/\partial\lambda^{\underline n}$, (${\underline n}=(n_1,...,n_{\dim\check\a^{G}_P})$ being a multi-index with respect to some fixed basis of $\check\a^{G}_{P,\C}$) on $\check\a^{G}_{P,\C}$, cf.\ \cite{schwfr} or \cite{grobner-EisRes}. Furthermore, let $\mathcal J$ be the ideal of the center of the universal enveloping algebra $\mathfrak U(\g_{\infty,\C})$, which annihilates the contragredient representation $\EE^{\sf v}_\mu$ of $\EE_\mu$ and let $\varphi_P$ be the associate class of the cuspidal automorphic representation $\tau(r)\phi^e$.

A $(\g_\infty, K_\infty, G(\A_f))$-module $\mathcal A_{\mathcal J,\{P\},\varphi_{P}}(G)$ has then been defined in \cite{schwfr} 1.3 as the space of automorphic forms of $G(\A_F)$ supported in $\varphi_P$ and annihilated by a power of $\mathcal J$: More precisely, it is the span of the holomorphic values at the point $\lambda={\rm pr}_{\check\a_{P,\C}\ra\check\a^{G}_{P,\C}}((r+e/2,...,r+e/2))=0$ of the Eisenstein series $E_P(h,\lambda)$, $h$ running through all $K_\infty$-finite sections $h\in {\rm Ind}_{P(\A)}^{G(\A)}[\tau(r)\phi^e]$, together with all their derivatives in the parameter $\lambda$, i.e.,
$\mathcal A_{\mathcal J,\{P\},\varphi_{P}}(G)$ is the image of the surjective map
$${\rm Eis}_\partial: {\rm Ind}_{P(\A_F)}^{G(\A_F)}[\tau(r)\phi^e\otimes S(\check\a^{G}_{P,\C})] \twoheadrightarrow \mathcal A_{\mathcal J,\{P\},\varphi_{P}}(G)$$
$$h\otimes \frac{\partial^{\underline n}}{\partial\lambda^{\underline n}}\mapsto \frac{\partial^{\underline n}}{\partial\lambda^{\underline n}} E_P(h,\lambda)|_{\lambda=0}.$$
See \cite{schwfr} 3.3 or \cite{grobner-EisRes} 2.3.\\\\
Consequently, $\Pi'(r)\phi^e={\rm Eis}_0({\rm Ind}_{P(\A_F)}^{G(\A_F)}[\tau(r)\phi^e])$ is a submodule of $\mathcal A_{\mathcal J,\{P\},\varphi_{P}}(G)$. We denote by $\imath_{\Pi'(r)\phi^e}: \Pi'(r)\phi^e\hra\mathcal A_{\mathcal J,\{P\},\varphi_{P}}(G)$ the natural inclusion sending an automorphic form to itself and by
$$\imath^q_{\Pi'(r)\phi^e}: H^q(\g_\infty, K_\infty, \Pi'(r)\phi^e\otimes \EE_\mu) \ra H^q(\g_\infty, K_\infty, \mathcal A_{\mathcal J,\{P\},\varphi_{P}}(G)\otimes \EE_\mu)$$
the natural induced morphism of $G(\A_f)$-modules.\\\\
As shown in \cite{schwfr} Thm.\ 1.4, $\mathcal A_{\mathcal J,\{P\},\varphi_{P}}(G)$ is a direct summand of the $(\g_\infty, K_\infty, G(\A_f))$-module $\mathcal A_{\mathcal J}(G)$ of automorphic forms annihilated by a power of $\mathcal J$. Hence, the $G(\A_f)$-module defined by $H^{q}(\g_\infty, K_\infty,\mathcal A_{\mathcal J,\{P\},\varphi_{P}}(G)\otimes\EE_\mu)$ appears as a direct summand of $H^{q}(\g_\infty, K_\infty,\mathcal A_{\mathcal J}(G)\otimes\EE_\mu)$ and so the fundamental isomorphism of $G(\A_f)$-modules
$$\mathscr F^q: H^{q}(\g_\infty, K_\infty,\mathcal A_{\mathcal J}(G)\otimes\EE_\mu)\ira H^q(S_n,\EE_\mu)$$
from \cite{franke}, Thm.\ 18 defines an injection $\mathscr F^q: H^{q}(\g_\infty, K_\infty,\mathcal A_{\mathcal J,\{P\},\varphi_{P}}(G)\otimes\EE_\mu)\hra H^q(S_n,\EE_\mu)$. By Strong Multiplicity one (\cite{jacshal2} Thm.\ 4.4) together with Multiplicity One (\cite{schwfr} \S 3.3 \& \cite{shalika, jaclang}) for isobaric automorphic representations, its image depends only on the isomorphism class of $(\Pi'(r)\phi^e)_f$ (and $\EE_\mu$) justifying the notation
$$H^q_{(\Pi'(r)\phi^e)_f}(S_n,\EE_\mu):=\mathscr F^q(H^{q}(\g_\infty, K_\infty,\mathcal A_{\mathcal J,\{P\},\varphi_{P}}(G)\otimes\EE_\mu)).$$
Next we fix a basis element of the one-dimensional vector space $H^{b_n}(\g_\infty, K_\infty,W((\Pi'(r)\phi^e)_\infty)\otimes\EE_\mu)$. By \cite{bowa}, Prop.\ II.3.1, this basis element is a $K_\infty$-homomorphism $[(\Pi'(r)\phi^e)_{\infty}]: \Lambda^{b_n}\g_\infty/\k_\infty \ra W((\Pi'(r)\phi^e)_\infty)\otimes\EE_\mu.$ We view this basis homomorphism as an element in the space of $K_\infty$-invariant vectors $\left(\Lambda^{b_n}(\g_\infty/\k_\infty)^* \otimes W((\Pi'(r)\phi^e)_\infty)\otimes\EE_\mu\right)^{K_\infty}$ and write it explicitly as

\begin{equation}\label{eq:gen}
[(\Pi'(r)\phi^e)_{\infty}]:=\sum_{\underline i=(i_1,...,i_{b_n})} \sum_{\alpha=1}^{\dim \EE_\mu} X^*_{\underline i}\otimes\xi_{\infty,\underline i, \alpha}\otimes e_\alpha.
\end{equation}
Here, $X^*_{\underline i}:=X^*_{i_1}\wedge ... \wedge X^*_{i_{b_n}}$ for some fixed $\Q$-basis $\{X_j\}$ of $\g_\infty/\k_\infty$, $e_\alpha:=\otimes_{v\in S_\infty} e_{\alpha,v}\in \EE_\mu=\otimes_{v\in S_\infty}\EE_{\mu_v}$, such that $\{e_{\alpha,v}\}_{\alpha}$ defines a $\Q(\EE_{\mu_v})$-basis of $\EE_{\mu_v}$ for all $v\in S_\infty$; whereas $\xi_{\infty,\underline i, \alpha}\in W((\Pi'(r)\phi^e)_\infty)$ are Whittaker functionals chosen accordingly. We record one property of this construction in the following 

\begin{rem}\label{rem:same}
Two Eisenstein representations $(\Pi_1'(r_1)\phi^{e_1})$ and $(\Pi_2'(r_2)\phi^{e_2})$ with {\it isomorphic} archimedean components $(\Pi_1'(r_1)\phi^{e_1})_\infty\cong (\Pi_2'(r_2)\phi^{e_2})_\infty$ give rise {\it identical} smoth archimedean Whittaker models $W((\Pi_1'(r_1)\phi^{e_1})_\infty)= W((\Pi_2'(r_2)\phi^{e_2})_\infty)$, due to their uniqueness \cite{shalika,chm00}. Therefore, fixing a generator  of $H^{b_n}(\g_\infty, K_\infty,W((\Pi_1'(r_1)\phi^{e_1})_\infty)\otimes\EE_\mu)=H^{b_n}(\g_\infty, K_\infty,W((\Pi_2'(r_2)\phi^{e_2})_\infty)\otimes\EE_\mu)$ as described above, uniformly fixes {\it the same generator} for all Eisenstein representations $(\Pi_1'(r_1)\phi^{e_1})$ and $(\Pi_2'(r_2)\phi^{e_2})$ with $(\Pi_1'(r_1)\phi^{e_1})_\infty\cong (\Pi_2'(r_2)\phi^{e_2})_\infty$. (This also justifies the notation $[(\Pi'(r)\phi^e)_{\infty}]$, which only sees the isomorphism class of the respective local archimedean component.)
\end{rem}

Now, in order to finally describe our choice of $\Upsilon$, recall that $W^\psi$ denotes the injective map computing the $\psi$-Fourier coefficient of an element in $\Pi'(r)\phi^e$. We denote by
$$(W^\psi)^{-1,q}: H^{q}(\g_\infty, K_\infty,W(\Pi'(r)\phi^e)\otimes\EE_\mu) \ira H^{q}(\g_\infty, K_\infty, \Pi'(r)\phi^e\otimes\EE_\mu)$$
the natural isomorphism of $G(\A_f)$-modules induced by its inversion.
We obtain
\begin{prop}\label{prop:Eisen}
There is the following sequence of $G(\A_f)$-isomorphisms
\begin{eqnarray*}
W((\Pi'(r)\phi^e)_f) & \ira_{[(\Pi'(r)\phi^e)_{\infty}]\otimes}  & H^{b_n}(\g_\infty, K_\infty,W((\Pi'(r)\phi^e)_\infty)\otimes\EE_\mu) \otimes W((\Pi'(r)\phi^e)_f)\\
& \ira &  H^{b_n}(\g_\infty, K_\infty,W(\Pi'(r)\phi^e)\otimes\EE_\mu)\\
& \ira_{(W^\psi)^{-1,b_n}} &  H^{b_n}(\g_\infty, K_\infty, \Pi'(r)\phi^e\otimes\EE_\mu)\\
& \ira_{\imath_{\Pi'(r)\phi^e}^{b_n}} & H^{b_n}(\g_\infty, K_\infty,\mathcal A_{\mathcal J,\{P\},\varphi_{P}}(G)\otimes\EE_\mu)\\
& \ira_{\mathscr{F}^{b_n}} & H^{b_n}_{(\Pi'(r)\phi^e)_f}(S_n,\EE_\mu)
\end{eqnarray*}
the unspecified map being (a fixed choice of) the obvious one.
\end{prop}
\begin{proof}
By the preceding discussion we only need to show that $\imath_{\Pi'(r)\phi^e}^{b_n}$ is an isomorphism.
For this we observe that by our Lem.\ \ref{lem:eisen}, i.e., by the holomorphy of the Eisenstein series spanning $\Pi'(r)\phi^e$ at their common point of evaluation $\lambda=0$ and the irreducibility of ${\rm Ind}_{P(\A_F)}^{G(\A_F)}[\tau(r)\phi^e]$, the length of the filtration of $\mathcal A_{\mathcal J,\{P\},\varphi_{P}}(G)$, as defined in \cite{grobner-EisRes}, \S 3.1, may be chosen to be $m(\{P\})=0$, see \cite[\S 3.1]{grobner-EisRes} or \cite{franke}, Rem.\ 2, p.\ 242. Hence, as established in \cite{grobner-EisRes}, Cor.\ 16, $H^{q}(\g_\infty, K_\infty,\mathcal A_{\mathcal J,\{P\},\varphi_{P}}(G)\otimes\EE_\mu)$ decomposes as a direct sum, which -- invoking \cite{moewalgln}, II \& III and the fact that all summands $\Pi_i(r)\phi^e$ of $\Pi'(r)\phi^e$ are different and unitary modulo a shift independent of the index $i$ -- degenerates to one single summand, namely $H^{q}(\g_\infty, K_\infty, {\rm Ind}_{P(\A_F)}^{G(\A_F)}[\tau(r)\phi^e\otimes S(\check\a^{G}_{P,\C})]\otimes\EE_\mu)$. As a consequence of the minimality of the degree $q=b_n$ we hence obtain
$$H^{b_n}(\g_\infty, K_\infty,\mathcal A_{\mathcal J,\{P\},\varphi_{P}}(G)\otimes\EE_\mu)\cong  (\Pi'(r)\phi^e)_f.$$
see, e.g., \cite{franke} pp.\ 256--257. Now recall from \eqref{eq:onedim} that also $H^{b_n}(\g_\infty, K_\infty, \Pi'(r)\phi^e\otimes\EE_\mu)\cong (\Pi'(r)\phi^e)_f$. By irreducibility it is hence enough to show that $\imath_{\Pi'(r)\phi^e}^{b_n}$ is non-zero in order to reveal $\imath_{\Pi'(r)\phi^e}^{b_n}$ as an isomorphism. To this end, observe that by construction Im$\imath_{\Pi'(r)\phi^e}^{b_n}={\rm Im Eis}^{b_n}_0$, where ${\rm Eis}^{b_n}_0$ equals the natural map in cohomology in degree $q=b_n$ induced from Eisenstein summation ${\rm Eis}_0$. By the minimality of the degree $q=b_n$ and Lem.\ \ref{lem:eisen}, revealing the Eisenstein series spanning $\Pi'(r)\phi^e$ as holomorphic at their common point of evaluation $\lambda=0$, this map ${\rm Eis}^{b_n}_0$ is injective by \cite{schwLNM}, Satz 4.11. See also \ \cite{bor2}, 2.9 and \cite{speh}, Thm.\ 1. This shows the claim.
\end{proof}

\begin{defn}\label{def:upsilon}
We let $\Upsilon=\Upsilon_{\Pi'(r)\phi^e}$ be the composition of the maps in Prop.\ \ref{prop:Eisen}. This way ${}^\sigma\Upsilon=\Upsilon_{{}^\sigma\Pi'(r)\phi^e}$ is a non-trivial element in ${\rm Hom}_{G(\A_f)}(W(({}^\sigma\Pi'(r)\phi^e)_f),H_{({}^\sigma\Pi'(r)\phi^e)_f}^{b_n}(S_n,{}^\sigma\!\EE_\mu))$ for each $\sigma\in$ Aut$(\C)$.
\end{defn}

\begin{prop}\label{prop:Upsilon}
One has the equality $\tilde\sigma^{b_n}(H_{(\Pi'(r)\phi^e)_f}^{b_n}(S_n,\EE_\mu))=H_{({}^\sigma\Pi'(r)\phi^e)_f}^{b_n}(S_n,^\sigma\!\EE_\mu)$ for all $\sigma\in {\rm Aut}(\C)$. Moreover,
${\rm Hom}_{G(\A_f)}(W(({}^\sigma\Pi'(r)\phi^e)_f),H_{({}^\sigma\Pi'(r)\phi^e)_f}^{b_n}(S_n,{}^\sigma\!\EE_\mu))$ is one-dimensional.
\end{prop}
\begin{proof}
The first assertion follows precisely as in \cite{grob-app}, Prop.\ 1.7. One-dimensionality of the space ${\rm Hom}_{G(\A_f)}(W(({}^\sigma\Pi'(r)\phi^e)_f),H_{({}^\sigma\Pi'(r)\phi^e)_f}^{b_n}(S_n,{}^\sigma\!\EE_\mu))$ is obvious by irreducibility (and the existence of ${}^\sigma\Upsilon=\Upsilon_{{}^\sigma\Pi'(r)\phi^e}\neq 0$).
\end{proof}

\begin{cor}\label{cor:Whittakerperiods}
For each Eisenstein representation $\Pi'(r)\phi^e$ the bottom-degree Whittaker periods $p(\Pi'(r)\phi^e)=p(\Pi'(r)\phi^e,\Upsilon)$ are well-defined. In particular, putting $e=0$, and $k=1$, every cohomological, essentially unitary cuspidal automorphic representation $\Pi$ and putting $e=r=0$ every unitary cohomological isobaric sum $\Pi'$ as in \S\ref{sect:Eisen}, satisfies the assumptions of Prop.\ \ref{prop:Whittakerperiods} in its minimal, non-vanishing degree $b_n$ and with $\Upsilon$ as above.
\end{cor}

\begin{rem}
In the case of cusp forms our construction recovers the Whittaker periods constructed in \cite{raghuram-shahidi-imrn}, whereas in the case of unitary isobaric sums $\Pi'$ we retrieve the periods considered in \cite[\S 1.5.2]{grob-app}.
\end{rem}

\subsubsection{Canonical choices of generators}\label{sect:gens}
We will now specify our (yet abstract) choices of generators $[\Pi'_\infty]$ for Eisenstein representations $\Pi'$. This will finally fix our choice of the attached map $\Upsilon=\Upsilon_{\Pi'}$, cf.\ Def.\ \ref{def:upsilon}, and consequently the Whittaker periods $p(\Pi')=p(\Pi',\Upsilon)$ up to an element in the rationality field $\Q(\Pi')$, cf.\ Rem.\ \ref{rem:upto}. We now explain the details of this construction of generators.\\\\
We first consider algebraic Hecke characters $\chi: \GL_1(F)\backslash \GL_1(\A_F)\ra\C^*$. As this is a very special instance of an Eisenstein representation, we have defined a (yet abstract) generator $[\chi_\infty]$ in \eqref{eq:gen}. By construction, $b_1=0$ and $\EE_\mu = \chi^{-1}_\infty$ is a character in this case, whence $[\chi_\infty]$ has only one summand: As it is clear from \eqref{eq:gen} we may write $[\chi_\infty]= 1 \otimes \xi_\infty \otimes 1$, where the first ``$1$'' is a basis of $\bigwedge^0(\gl_{1,\infty}/\k_{1,\infty})=\R$ and the last ``$1$'' is a basis of the representation space $\C$ of the character $\EE_\mu= \chi^{-1}_\infty$; and some, yet unspecified, non-zero Whittaker function $\xi_\infty\in W(\chi_\infty)$. We will now make a specific choice for $\xi_\infty$. \\\\
To this end, observe that the smooth Whittaker model $W(\chi_\infty)$ is given by the space of smooth functions $f_{\chi_\infty,z}: g_\infty \mapsto f_{\chi_\infty,z}(g_\infty):=\chi_\infty(g_\infty)\cdot z$, $z\in\C$. Clearly, any $f_{\chi_\infty,z}$ with $z\in\C^\times$ serves as a possible choice for the Whittaker function $\xi_\infty$ in our generator $[\chi_\infty]$.  Our specific choice is now made once and for all by our

\begin{aspt}\label{conv:2}
For each algebraic Hecke character $\chi: \GL_1(F)\backslash \GL_1(\A_F)\ra\C^*$, we put $z=1$. I.e., we choose the generator $[\chi_\infty]= 1 \otimes f_{\chi_\infty,1} \otimes 1$.
\end{aspt}

This choice/convention has the following effect on the attached Whittaker period $p(\chi)$:

\begin{lem}\label{lem:chiq}
For all algebraic Hecke characters $\chi: \GL_1(F)\backslash \GL_1(\A_F)\ra\C^*$, the Whittaker period $p(\chi)$ is an element of $\Q(\chi)^\ast$.
\end{lem}
\begin{proof}
Let us consider diagram \eqref{diawhit}, by which we defined our Whittaker periods, in the case of $\Pi=\chi$ and recall how a $\sigma\in$ Aut$(\C)$ acts on $W(\chi_f)$ as well as on $H^0(S_1,\EE_\mu)(\chi_f)$. For the action of Aut$(\C)$ on $W(\chi_f)$ we look up the map $\tilde\sigma_{\chi_f}$ from \S \ref{sect:whittaker}: As $n=1$ now, we get $t_\sigma=1$, and so $\sigma$ acts on a non-archimedean Whittaker function $\xi_f\in W(\chi_f)$ simply by applying it to its values $\tilde\sigma_{\chi_f}(\xi_f)(g_f)=\sigma(\xi_f(g_f))$. For the action of Aut$(\C)$ on $H^0(S_1,\EE_\mu)(\chi_f)$, we look up the map $\sigma^q$ (with $q=0$) defined in \eqref{eq:gq}: Applying $\sigma$ to the cohomology class defined by $[\chi_\infty]\otimes\chi_f$ amounts to forming the class $[({}^\sigma\chi)_\infty]\otimes(\sigma\circ\chi_f)$, where $[({}^\sigma\chi)_\infty]$ denotes our choice of a generator attached to the archimedean component $({}^\sigma\chi)_\infty$ of the global, $\sigma$-twisted Hecke character ${}^\sigma\chi$. 

For the reader, who seeks a concrete description of $[({}^\sigma\chi)_\infty]$, we provide the following explicit characterization of this generator: To this end, let us write $\chi_\infty(g_\infty)$ as a product $\chi_\infty(g_\infty)=\prod_{v\in S_\infty} \chi_v(g_v) = \prod_{v\in S_\infty} g^{a_{\iota_v}}_v \bar g^{b_{\iota_v}}_v$, for some integers $a_{\iota_v}, b_{\iota_v}$, (recall that $g_v\in \GL_1(F_v)\cong\C^\ast$). Moreover, recall from Rem.\ \ref{rem:emu} that as a representation of $G(F)$ we get ${}^\sigma\EE_\mu=\EE_{{}^\sigma\!\mu}$, where the latter is the representation of highest weight $^\sigma\!\mu=((\mu_{\sigma^{-1}\circ\iota_v},\mu_{\sigma^{-1}\circ\bar\iota_v})_{v\in S_\infty})$. Therefore, as ${}^\sigma\chi$ is cohomological with respect to $\EE_{{}^\sigma\!\mu}$, we must have $({}^\sigma\chi)_\infty(g_\infty)=\prod_{v\in S_\infty} g^{a_{\sigma^{-1}\circ\iota_v}}_v \bar g^{b_{\sigma^{-1}\circ\iota_v}}_v$, and so $[({}^\sigma\chi)_\infty]=1 \otimes f_{({}^\sigma\chi)_\infty,1} \otimes 1$ with $({}^\sigma\chi)_\infty$ as right above.

Now let us reconsider diagram \eqref{diawhit}: Comparing the two actions of Aut$(\C)$, which we have just described, we see that in the special case of $\Pi=\chi$ already the following diagram 
$$\xymatrix{
W(\chi_f) \ar[rrr]^{\Upsilon_\chi}\ar[d]^{\tilde\sigma_{\chi_f}}  & & &
H^0(S_1,\EE_\mu)(\chi_f)\ar[d]^{\tilde\sigma^0} \\
W({}^\sigma\chi_f)
\ar[rrr]^{{}^\sigma\Upsilon_\chi}
 & & &
H^0(S_1,{}^\sigma\!\EE_\mu)({}^\sigma\chi_f)
}$$
where no periods $p(\chi)$ resp.\ $p({}^\sigma\chi)$ have been inserted, commutes for all $\sigma\in$ Aut$(\C)$. Therefore, $p(\chi)$ must be in the subfield of $\C$, which is fixed by all $\sigma$, for which $({}^\sigma\chi)_\infty=\chi_\infty$ and $(\sigma\circ\chi_f)=\chi_f$. But as $(\sigma\circ\chi_f)={}^\sigma\chi_f$, this subfield is $\Q(\chi)$ by definition. 

\end{proof}

\begin{rem}
As it turns out by the proof of Lem.\ \ref{lem:chiq}, instead of $f_{\chi_\infty,1}$ we could have chosen any $f_{\chi_\infty,q}$, with $q\in\Q(\EE_\mu)$, in oder to fix our generator $[\chi_\infty]$. Indeed, as $\Q(\chi)\supseteq\Q(\EE_\mu)$ (cf.\ \cite{grob-ragh} Cor.\ 8.7), this would not have changed the assertion. The reader is invited to compare this to \cite{mahnkopf}, last line of p.\ 596, where the analogous statement can be found for $\GL_n/\Q$ (in which case $\Q(\EE_\mu)=\Q$ for all $\mu$).
\end{rem}

As it is obvious by construction, $f_{\chi_\infty,1}(g_\infty)=\prod_{v\in S_\infty}f_{\chi_v,1}(g_v)$, where $f_{\chi_v,1}(g_v)=\chi_v(g_v)\cdot 1$ is a non-zero Whittaker functional in $W(\chi_v)$.\\\\
We now turn back to the general case of unitary cohomological Eisenstein representations. So, let $\Pi'=\Pi_1\boxplus...\boxplus\Pi_k$ be such an Eisenstein representation of $G_n(\A_F)$ as defined in \ref{sect:Eisen}, cohomological with respect to $\EE_\mu$. We recall from \eqref{eq:ind} that at each place $v\in S_\infty$
\begin{equation}\label{eq:ind2}
\Pi'_v\cong {\rm Ind}_{B(\C)}^{G(\C)}[z_1^{\ell_{v,1}}\bar z^{-\ell_{v,1}}_1\otimes ...\otimes z_n^{\ell_{v,n}}\bar z^{-\ell_{v,n}}_n]
\end{equation}
with $\ell_{v,j}=-\mu_{\iota_v,n-j+1}+\frac{n+1}{2}-j$. For each $1\leq j\leq n$, abbreviate $\alpha_{j,v}(z_v):=z_v^{\ell_{v,j}}\bar z_v^{-\ell_{v,j}}$ and set $\alpha_{j,\infty}:=\otimes_{v\in S_\infty} \alpha_{j,v}$. Furthermore, put $\chi_{j,v}:=\alpha_{j,v}\ \|\cdot\|^{\tfrac{n-2j+1}{2}}_v$, i.e., $\chi_{j,v}$ is the product of the $j$-th factor in the inducing datum of $\Pi'_v$ and $j$-th entry of the half sum of positive roots $\rho_v=\rho_{B,v}$. As the exponents $\pm\ell_{v,j}+\tfrac{n-2j+1}{2}$ are all integers, we see that $\chi_{j,\infty}=\otimes_{v\in S_\infty} \chi_{j,v}$ is the archimedean component of an algebraic Hecke character $\chi_j$. Therefore, we have pinned down a generator $[\chi_{j,\infty}]$ in Conv.\ \ref{conv:2}. We recall from above that the cohomology, spanned by $[\chi_{j,\infty}]$, is $H^0(\gl_{1,\infty}, K_{1,\infty}, W(\chi_{j,\infty})\otimes \EE_{\mu_j})$, where $\EE_{\mu_j}=\chi^{-1}_{j,\infty}$.\\\\
In order to fix our choice of a generator $[\Pi'_\infty]$ we first recall a well-known, general result of Borel-Wallach (cf.\ \cite{bowa}, Thm.\ III.3.3.(ii)), which expresses $(\g_\infty,K_\infty)$-cohomology of an induced representation in terms of the cohomology of the inducing datum: More precisely, let $P=LN$ be a standard parabolic subgroup of $G$, $L\cong\GL_{n_1}\times...\times\GL_{n_k}$, and let $\pi_\infty\cong\otimes_{j=1}^k\pi_{\pi_{j,\infty}}$ be a smooth representation of $L_\infty$ (In this paper, we will only be interested in the case, when $\pi_\infty$ acts on the Fr\'echet space of smooth vectors of a unitary representation on a Hilbert space). If the induced representation ${\rm Ind}_{P_\infty}^{G_\infty}[\pi_\infty]$ is cohomological with respect to a coefficient module $\EE_\lambda$, then, combining \cite{bowa}, Thm.\ III.3.3.(ii) and the canonical isomorphism provided by the K\"unneth rule, \cite{bowa}, 1.3.(2), one obtains a canonical isomorphism
$$I^P_{\pi_\infty}: H^{q}(\g_\infty, K_\infty, {\rm Ind}_{P_\infty}^{G_\infty}[\pi_\infty]\otimes\EE_\lambda) \ira \bigoplus_{q_1+...+q_k=q-\ell(s)}\bigotimes_{j=1}^k H^{q_j}(\gl_{n_j,\infty}, K_{n_j,\infty}, \pi_{j,\infty}\rho_{j,\infty}\otimes \EE_{\lambda_j}).$$
Here, $\rho_{j,\infty}=\rho_{P}|_{\GL_{n_j,\infty}}$ and $\EE_{\lambda_j}$ is the unique irreducible algebraic coefficient module of $\GL_{n_j,\infty}$ of highest weight $\lambda_j:=s(\lambda+\rho)-\rho|_{\GL_{n_j,\infty}}$, for a certain, uniquely determined Kostant representative $s\in {\bf W}^P$ (cf.\ \cite{bowa} III.1.4 \& Thm.\ III.3.3.(i)). We kindly refer the reader to p.\ 65--66 in \cite{bowa}, where all the details of the construction of this isomorphism may be found. 

We recall that if $\pi_\infty$ is also generic, Shahidi has defined a non-zero Whittaker functional $\lambda_{\psi_B}(0,\pi_\infty): {\rm Ind}_{P_\infty}^{G_\infty}[\pi_\infty]\ra\C$ in \cite{shahidi-book}, \S 3.6 (see (3.6.8) for the very definition and Cor.\ 3.6.11 for analytic continuation to an entire function). For a section $h_\infty\in {\rm Ind}_{P_\infty}^{G_\infty}[\pi_\infty]$ we let $W_{h_\infty}\in W({\rm Ind}_{P_\infty}^{G_\infty}[\pi_\infty])$ be the attached smooth Whittaker function 
\begin{equation}\label{eq:archwhit}
W_{h_\infty}(g_\infty):=\lambda_{\psi_B}(0,\pi_\infty)(\pi_\infty(g_\infty) h_\infty) ,
\end{equation}
and we will write
$$\mathcal W^{q}_P: H^{q}(\g_\infty, K_\infty, {\rm Ind}_{P_\infty}^{G_\infty}[\pi_\infty]\otimes\EE_\lambda) \longrightarrow
H^{q}(\g_\infty, K_\infty, W({\rm Ind}_{P_\infty}^{G_\infty}[\pi_\infty])\otimes\EE_\lambda) $$ 
in order to denote the map in cohomology induced by $W_\bullet$. If ${\rm Ind}_{P_\infty}^{G_\infty}[\pi_\infty]$ is irreducible, then $W_\bullet$ is an isomorphism, which follows from \cite{shahidi-book}, Cor.\ 3.6.11, hence, so is $\mathcal W^q_P$ in this case.

In order to finally explain our choice of a generator $[\Pi'_\infty]$, we specify the isomorphism $I^P_{\pi_\infty}$ from above using \eqref{eq:ind2}: We set $P=B$ and $\pi_{j,\infty}:=W(\alpha_{j,\infty})=W(\chi_{j,\infty})\rho^{-1}_{j,\infty}$ and observe that in this case the unique Kostant representative $s\in {\bf W}^B$, satisfiying \cite{bowa} Thm.\ III.3.3.(i), has length $\ell(s)=b_n$. As $\Pi'_\infty$ is irreducible, this finally provides us a triangle of isomorphisms 
$$\xymatrix{
H^{b_n}(\g_\infty, K_\infty, {\rm Ind}_{B_\infty}^{G_\infty}[\otimes_{j=1}^n W(\alpha_{j,\infty})]\otimes\EE_\mu) \ar[rrr]^{\mathcal W^{b_n}_B}  & & &
H^{b_n}(\g_\infty, K_\infty, W(\Pi'_\infty)\otimes\EE_\mu) \\ &&&\\
\bigotimes_{j=1}^n H^0(\gl_{1,\infty}, K_{1,\infty}, W(\chi_{j,\infty})\otimes \EE_{\mu_j})\ar[uu]^{(I^B_{\otimes_{j=1}^n W(\alpha_{j,\infty})})^{-1}}
\ar@{.>}[uurrr]^{}
 & & &
}$$
Let $I_{\Pi'_\infty}$ denote the map attached to the dashed arrow, i.e., the isomorphism given by the composition $I_{\Pi'_\infty}:=\mathcal W^{b_n}_B\circ (I^B_{\otimes_{j=1}^n W(\alpha_{j,\infty})})^{-1}$. Our last convention, fixing $[\Pi'_\infty]$, is now stated as

\begin{aspt}\label{diag}
For a unitary cohomological Eisenstein representation $\Pi'$, we choose the generator
$$[\Pi'_{\infty}]:= I_{\Pi'_\infty}\left(\otimes_{j=1}^n [\chi_{j,\infty}]\right) = \mathcal W^{b_n}_B \left((I^B_{\otimes_{j=1}^n W(\alpha_{j,\infty})})^{-1}\left(\otimes_{j=1}^n [\chi_{j,\infty}]\right)\right).$$
\end{aspt}
Observe that by the uniqueness of the Langlands datum of $\Pi'_\infty$, cf.\ \cite{bowa}, Thm.\ IV.4.11, the sets $\{\chi_{j,v}\}_{j=1}^n$ are uniquely determined, i.e., the only other way the write $\Pi'_v$ as an induced representation from characters is by permuting the factors $\chi_{j,v}$. However, $[\Pi'_{\infty}]= I_{\Pi'_\infty}\left(\otimes_{j=1}^n\otimes_{v\in S_\infty} 1\otimes f_{\chi_{j,v},1}\otimes 1\right)$ is independent of the ordering in the tensor products ``$\otimes_j$'' as well as in ``$\otimes_v$'', as it follows from \cite{shahidi-book} (3.6.1). As a consequence, $[\Pi'_{\infty}]$ is well-defined.

\subsection{Relations of algebraicity}

\begin{defn}\label{def:rel}
Let $\F\subset\C$ be any subfield and let $x, y\in\C$ be two complex numbers. We say $x\sim_{\F} y$ if there is an $a\in \F$ such that $x=ay$ or $ax=y$.
\end{defn}

\begin{rem}\label{transitive}
Note that this relation is symmetric, but not transitive unless all sides of the relation are non zero. More precisely, if $x,y,z \in \C$ such that $x\sim_{\F} y$ and $y \sim_{\F} z$, then we do not have $x\sim_{\F} z$ in general, unless $xyz\neq 0$.
\end{rem}

The main goal of this paper is to prove several such relations among different $L$-values and various periods. As we have seen in the previous sections, the automorphism group ${\rm Aut}(\C)$ acts on the set of representations, hence it will also act on the set of $L$-values and periods. The relations that we will prove behave well under the action of ${\rm Aut}(\C)$ in the following sense.

\begin{defn}\label{definition algebraic relation}
Let $\F,L\subset\C$ be two subfields. Let $x=\{x(\sigma)\}_{\sigma\in {\rm Aut}(\C)}$ and $y=\{y(\sigma)\}_{\sigma\in {\rm Aut}(\C)}$ be two families of complex numbers. We say $x\sim_{L} y$ (and this relation) {\it is equivariant under} Aut$(\C/\F)$, if either $y(\sigma)=0$ for all $\sigma\in {\rm Aut}(\C)$, or if $y(\sigma)\neq 0$ for all $\sigma\in {\rm Aut}(\C)$ and the following two conditions are verified:
\begin{enumerate}
\item $x(\sigma)\sim_{\sigma(L)} y(\sigma)$ for all $\sigma$;
\item $\sigma\left(\cfrac{x(\tau)}{y(\tau)}\right)=\cfrac{x(\sigma\tau)}{y(\sigma\tau)} $  for all $\sigma\in {\rm Aut}(\C/\F)$ and all $\tau\in {\rm Aut}(\C)$.
\end{enumerate}

\end{defn}

\begin{rem}
One can replace the first condition by requiring $x(\sigma)\sim_{\sigma(L)} y(\sigma)$ for all $\sigma$ running through a choice of representatives of ${\rm Aut}(\C)/{\rm Aut}(\C/\F)$. In particular, if $\F=\Q$, instead of verifying the first condition for all $\sigma\in {\rm Aut}(\C)$, one only needs to verify it for a single fixed $\sigma_{0}\in {\rm Aut}(\C)$.
\end{rem}

\begin{lem}[Minimizing-Lemma]\label{minimize}
Let $\F\subset\C$ be any number field and let $L\subset \C$ be a number field, containing $\F^{Gal}$. Let $x=\{x(\sigma)\}_{\sigma\in {\rm Aut}(\C)}$ and $y=\{y(\sigma)\}_{\sigma\in {\rm Aut}(\C)}$ be as in Def.\ \ref{definition algebraic relation} and suppose that $y(\sigma)\neq 0$ for all $\sigma\in {\rm Aut}(\C)$. If the complex numbers $x(\sigma)$ and $y(\sigma)$ depend only on the restriction of $\sigma$ to $L$, then the second condition of Def.\ \ref{definition algebraic relation} implies the first.
\end{lem}

\begin{proof}
Fix $\sigma_0\in {\rm Aut}(\C)$. For any $\sigma\in {\rm Aut}(\C)$ fixing $\sigma_0(L)$, one has $\sigma\sigma_0\mid_{L}=\sigma_0\mid_{L}$. Hence $x(\sigma\sigma_0)=x(\sigma_0)$ and $y(\sigma\sigma_0)=y(\sigma_0)$ by our assumptions. Moreover, since $L\supset \F^{Gal}$, we know $\sigma\in$ Aut$(\C/\F)$. By the second condition, we have: $$\sigma\left(\cfrac{x(\sigma_0)}{y(\sigma_0)}\right)=\cfrac{x(\sigma\sigma_0)}{y(\sigma\sigma_0)} =\cfrac{x(\sigma_0)}{y(\sigma_0)}.$$
Therefore $\cfrac{x(\sigma_0)}{y(\sigma_0)}\in \sigma_0(L)$ for all $\sigma_0$ as expected.

\end{proof}

\begin{rem}
The previous lemma allows us to minimize the number field $\F$ in the relation. This will be very useful in the proof of the main theorems. Its omnipresence in our arguments is one of the main reasons why our formulas only hold over $F^{Gal}$.
\end{rem}

As a first, useful lemma, these notions imply

\begin{lem}\label{Gauss sum}
Interpreted as a family of complex numbers as in Def.\ \ref{definition algebraic relation}, $\mathcal{G}(\varepsilon_{f})\sim_{F^{Gal}}1$ is equivariant under the action of $\emph{\textrm{Aut}}(\C/F^{Gal})$.
\end{lem}
\begin{proof}
Recall that $L(s,\varepsilon_f)=\cfrac{\zeta_{F}(s)}{\zeta_{F^{+}}(s)}$. Let $Reg_{F}$ and $Reg_{F^{+}}$ be the regulator of $F$ and $F^{+}$ respectively. By Proposition $3.7$ of \cite{pazuki}, we know $Reg_{F}\sim_{\Q}Reg_{F^{+}}$. Hence, denoting the absolute discriminant of $F$ (resp.\ $F^+$) by $D_F$ (resp.\ $D_{F^+}$), the class number formula implies that $$L(1,\varepsilon_f)=\cfrac{Res_{s=1}\zeta_{F}(s)}{Res_{s=1}\zeta_{F^{+}}(s)}\sim_{\Q} \cfrac{(2\pi)^{d}Reg_{F}|D_{F^{+}}|^{1/2}}{Reg_{F^{+}}|D_{F}|^{1/2}}\sim_{\Q} \cfrac{(2\pi)^{d}|D_{F^{+}}|^{1/2}}{|D_{F}|^{1/2}}.$$
On the other hand, a classical result of Siegel, \cite{siegel} (revealing $L(1-m,\varepsilon_f)\in\Q$ for $m\geq 1$), combined with the functional equation, cf., e.g., \cite{Bump}, \S 3.1, shows that $L(m,\varepsilon_f)\sim_\Q \mathcal{G}(\varepsilon_{f}) (2\pi i)^{md}$, for odd $m\geq 1$. Consequently, we obtain
$$\mathcal{G}(\varepsilon_{f}) \sim_{\Q} i^{d}\cfrac{|D_{F^{+}}|^{1/2}}{|D_{F}|^{1/2}}.$$
Since $F^{+}$ is totally real, we know that $ |D_{F^{+}}|^{1/2}=\pm D_{F^{+}}^{1/2} \sim_{F^{+,Gal}} 1$.  It remains to show that $|D_{F}|^{1/2}\sim_{F^{Gal}}i^{d}$. To this end, let $\alpha\in F$ be a purely imaginary element, i.e., $\bar{\alpha}=-\alpha$. Since $-2\alpha=\det
\begin{pmatrix}
1 & \alpha\\
1 & -\alpha\end{pmatrix}
$, it is easy to see that $|N_{F^{+}/\Q} D_{F/F^{+}}|^{1/2 }\sim_{\Q}  \prod_{v\in S_\infty}|\iota_v(\alpha)|$ where $D_{F/F^{+}}$ is the relative discriminant with respect to $F/F^{+}$. So, $|D_{F}|^{1/2} \sim_{\Q} \prod_{v\in S_\infty}|\iota_v(\alpha)| \cdot |D_{F^{+}}|^{1/2} \sim_{F^{+,Gal}} \prod_{v\in S_\infty}|\iota_v(\alpha)|.$
We know that $\prod_{v\in S_\infty}\iota_v(\alpha)$ is an algebraic number giving rise to an extension of $\Q$ of degree 2. Its complex conjugate equals $(-1)^{d}\prod_{v\in S_\infty}\iota_v(\alpha)$. Hence if $d$ is even, then it is real quadratic and $\prod_{v\in S_\infty}|\iota_v(\alpha)|=\pm  \prod_{v\in S_\infty}\iota_v(\alpha)\sim_{F^{Gal}} 1=i^{d}$; otherwise, it is imaginary quadratic and $\prod_{v\in S_\infty}|\iota_v(\alpha)|=\pm i \prod_{v\in S_\infty}\iota_v(\alpha)\sim_{F^{Gal}} i\sim_{F^{Gal}}i^{d}$ as expected.
\end{proof}

As a consequence of the above discussion, one has
\begin{eqnarray}
\zeta_{F^{+}}(m)& \sim_{F^{Gal}} &  (2\pi i)^{md} \text{ if } m\geq 2 \text{ is even}\label{Dedekind}\\ \label{quadratic}
L(m,\varepsilon_{f}) &\sim_{F^{Gal}}& (2\pi i)^{md} \text{ if } m\geq 1 \text{ is odd}
\end{eqnarray}
both relations being equivariant under the action of Aut$(\C/F^{Gal})$. Indeed, the first relation follows from applying the functional equation and the fact that $ |D_{F^{+}}|^{1/2}\sim_{F^{Gal}} 1$, explained in the proof of Lem.\ \ref{Gauss sum}, to \cite{siegel}; whereas the latter follows directly from Lem.\ \ref{Gauss sum} and the fact that $L(m,\varepsilon_f)\sim_\Q \mathcal{G}(\varepsilon_{f}) (2\pi i)^{md}$, for odd $m\geq 1$, as explained in the proof of of Lem.\ \ref{Gauss sum}.

\subsection{Critical values of automorphic $L$-functions}
Deligne has defined the notion of {\it critical values} for the $L$-function $L(s,\mathbb M)$ of a motive $\mathbb M$ in terms of Hodge types, cf.\ \cite{deligne79}. Assuming that $\mathbb M$ corresponds to an automorphic representation $\pi$, such that $L(s,\mathbb M)=L(s-\alpha,\pi_f)$ for some unique shift of the argument $\alpha=\alpha(\pi)\in\R$, we may translate Deligne's original definition to automorphic $L$-functions.

\begin{defn}
A complex number $s_0\in \alpha(\pi)+\Z$ is called {\it critical} for $L(s,\pi)$ if both $L(s,\pi_\infty)$ and $L(1-s,\pi_\infty^{\sf v})$ are holomorphic at $s=s_0$.
\end{defn}

More explicitly, if $\pi=\pi_N\otimes\pi_M$ is tensor product of two automorphic representations of $\GL_N(\A_F)\times\GL_M(\A_F)$, then $\alpha(\pi)=\tfrac{N-M}{2}$, whereas if $\pi={\rm As}^{\pm}(\pi_N)$ is an automorphic representation defined by applying the local Langlands correspondence to the Asai-lifts, then $\alpha(\pi)=0$.

\subsubsection{Critical points for Rankin-Selberg $L$-function}\label{sect:critpts}

Let $\Pi(r)$ (resp.\ $\Pi'(s)$) be an essentially unitary, generic automorphic representation of $\GL_n(\A_F)$ (resp.\ $\GL_{n-1}(\A_F)$), cohomological with respect to $\EE_\mu$ (resp.\ $\EE_{\mu'}$). For $\iota\in \Sigma$ let us write $a_{\iota,i}:=\ell(\mu_{\iota},i)+r$ and $b_{\iota,j}:=\ell(\mu'_{\iota},j)+s$ for their respective infinity-types at $\iota$, cf.\ \S\ref{sect:pi}. We assume that
\begin{equation}\label{no middle class}
a_{\iota,i}+b_{\iota,j}\neq r+s \text{ for any }1\leq i\leq n, 1\leq j\leq n-1\text{ and any } \iota \in \Sigma.
\end{equation}
The critical points for the Rankin-Selberg $L$-function $L(s,\Pi\times \Pi')$ are the half integers $\tfrac12+m$ with $m\in \Z$ such that
$$-\min_{i,j,\iota}\{| a_{\iota,i}+b_{\iota,j}-r-s |\} <\tfrac12+m+r+s \leq \min_{i,j,\iota}\{| a_{\iota,i}+b_{\iota,j}-r-s |\}.$$

This can be easily seen using Deligne's approach in terms of Hodge types of motives, cf.\ \cite{deligne79} and \cite{jie_michael}. For a direct computation involving only representation theory of real reductive groups, see \cite{ragh16}, Cor.\ 2.35. It is easy to see that the set of critical points is non empty. We denote it by Crit$(\Pi(r)\times\Pi'(s))$. For example, if $r=s=0$, then $\tfrac12$ is always a critical point.

\begin{rem}
\begin{enumerate}
\item There are no critical points for $\Pi(r)\times \Pi'(s)$ if the inequality (\ref{no middle class}) is not satisfied (cf.\ $1.7$ of \cite{harris97}).
\item The precise relation between highest weights and infinity-types is given in subsection \ref{sect:pi}. One can check easily that (\ref{no middle class}) is automatically satisfied if the {\it piano-hypothesis}, i.e., Hypothesis \ref{hypo:piano} below, holds for $\mu$ and $\mu'$. If this is the case and if both representations $\Pi(r)$ and $\Pi'(s)$ are Eisenstein representations, cf.\ \S\ref{sect:Eisen}, then Crit$(\Pi(r)\times\Pi'(s))=$ Crit$({}^\sigma\Pi(r)\times{}^\sigma\Pi'(s))$ for all $\sigma\in$Aut$(\C)$: For this combine \cite{grob-app}, 1.5-1.6 and \cite{grob-harr}, Lem.\ 3.5.
\end{enumerate}
\end{rem}

\subsubsection{Critical points for the Asai $L$-functions}
Consider now $\Pi=\Pi(0)$, i.e., a unitary generic automorphic representation of $\GL_n(\A_F)$, which is cohomological with respect to $\EE_\mu$. We shall additionally assume now that $\Pi$ is conjugate self-dual. Then the calculations in section $1.3$ of \cite{harris_adjoint} show that the critical points of $L(s,\Pi,{\rm As}^{(-1)^{n}})$ (resp.\ $L(s,\Pi,{\rm As}^{(-1)^{n-1}})$ are the positive odd or non-positive even integers (resp.\ positive even or negative odd) $m$ such that
$$
\max_{i,j,\iota}\{a_{\iota,i}-a_{\iota,j} \mid a_{\iota,i}-a_{\iota,j}<0\}< m \leq \min_{i,j,\iota}\{a_{\iota,i}-a_{\iota,j} \mid a_{\iota,i}-a_{\iota,j}>0\}.
$$
In particular, the integers $0$ and $1$ are always critical for $L(s,\Pi,{\rm As}^{(-1)^{n}})$ and never for $L(s,\Pi,{\rm As}^{(-1)^{n-1}})$.

\subsection{Two rationality theorems revisited}

We recall now two rationality-results for critical $L$-values, proved in \cite{grob_harris_lapid} and \cite{grob-app}, which are the starting point of our investigations:

\begin{thm}[\cite{grob_harris_lapid} Thm.\ 7.1]\label{thm:gro-har-lap}
Let $\Pi$ be a conjugate self-dual, cuspidal automorphic representation of $G(\A_{\cm})=\GL_n(\A_{\cm})$, which is cohomological with respect to $\EE_\mu$. Let $p(\Pi)$ be the bottom-degree Whittaker period defined in Cor.\ \ref{cor:Whittakerperiods}. Then, the following holds:
\begin{enumerate}
\item For all $\sigma\in {\rm Aut}(\C/F^{Gal})$, there exists a non-zero constant $a({}^\sigma\Pi_\infty)$, only depending on the choice of a generator $[\Pi_\infty]$, such that
$$
\sigma\left(\frac{L^S(1,\Pi,{\rm As}^{(-1)^{n}})}{p(\Pi)a(\Pi_\infty)}\right) \ = \
\frac{L^S(1,{}^\sigma\Pi,{\rm As}^{(-1)^{n}})}{p({}^\sigma\Pi)a({}^\sigma\Pi_\infty)}.
$$

\item Equivalently, interpreting both sides below as a family of numbers $x=\{x(\sigma)\}_{\sigma\in{\rm Aut}(\C)}$ and $y=\{y(\sigma)\}_{\sigma\in{\rm Aut}(\C)}$ as in Def.\ \ref{definition algebraic relation},
$$
L^{S}(1,\Pi,{\rm As}^{(-1)^{n}})\ \sim_{E(\Pi)} a(\Pi_\infty) \ p(\Pi),
$$
is equivariant under $\textrm{\emph{Aut}}(\C/F^{Gal})$.
\end{enumerate}
\end{thm}

\begin{rem}
Theorem \ref{thm:gro-har-lap} has been stated in \cite{grob_harris_lapid} without reference to the Whittaker periods $p(\Pi)$. For the convenience of the reader we sketch how one obtains the above transcription: Thm.\ 7.1 of \cite{grob_harris_lapid} is a consequence of Thm. 5.3 and 6.4 {\it ibidem}. Thm. 6.4 can be rewritten as an Aut$(\C)$-equivariant relation of the residue Res$_{s=1}(L^S(s,\Pi\times\Pi^\vee))$ with an archimedean period, the bottom-degree Whittaker period $p(\Pi)$ and a top-degree version of the latter, as explained in \cite{grob-extrat}, Thm.\ 8.5. Similarly, Thm.\ 5.3 in \cite{grob_harris_lapid} can be restated as an Aut$(\C/F^{Gal})$-equivariant relation of the residue Res$_{s=1}(L^{S}(s,\Pi,{\rm As}^{(-1)^{n-1}}))$, an archimedean period and the above mentioned top-degree Whittaker period. Taking the quotient of the first and the second relation gives the theorem.
\end{rem}

For the next rationality-result we note that a pair of irreducible algebraic representations $\EE_\mu$ of $G_{n,\infty}$ and $\EE_{\mu'}$ of $G_{n-1,\infty}$, given by their highest weights $\mu=(\mu_v)_{v\in S_\infty}$ and $\mu'=(\mu'_v)_{v\in S_\infty}$ satisfy the {\it piano-hypothesis}, if the following holds

\begin{hyp}[Piano-hypothesis]\label{hypo:piano}
At each archimedean place $v=(\iota_v,\bar\iota_v)\in S_\infty$,
$$\mu_{\iota_v,1}\geq -\mu'_{\iota_v,n-1}\geq \mu_{\iota_v,2}\geq -\mu'_{\iota_v,n-2}\geq...\geq -\mu'_{\iota_v,1}\geq \mu_{\iota_v,n}$$
$$\mu^{\sf v}_{\bar\iota_v,1}\geq -\mu'^{\sf v}_{\bar\iota_v,n-1}\geq \mu^{\sf v}_{\bar\iota_v,2}\geq -\mu'^{\sf v}_{\bar\iota_v,n-2}\geq...\geq -\mu'^{\sf v}_{\bar\iota_v,1}\geq \mu^{\sf v}_{\bar\iota_v,n}.$$
\end{hyp}

This condition has been called ``{\it interlacing-hypothesis}'' in \cite{grob-app}. Here we prefer to call it after the more poetical picture of a piano's keys, however: The white ones (``whole-tones'')  taking the role of the coordinates of $\mu$, the black keys (``half-tones'') taking the role of the coordinates of $\mu'$.

\begin{thm}[\cite{grob-app} Thm.\ 1.9]\label{thm:main-app}
Let $\Pi$ be a cuspidal automorphic representation of $\GL_n(\A_F)$ which is cohomological with respect to $\EE_\mu$ and let $\Pi'$ by an isobaric automorphic representation of $\GL_{n-1}(\A_F)$, as in \S\ref{sect:Eisen}, cohomological with respect to $\EE_{\mu'}$. Let $p(\Pi)$ and $p(\Pi')$ be the bottom-degree Whittaker periods defined in Cor.\ \ref{cor:Whittakerperiods}. Assume that $\Pi'$ has central character $\omega_{\Pi'}$ and that the highest weights $\mu=(\mu_v)_{v\in S_\infty}$ and $\mu'=(\mu'_v)_{v\in S_\infty}$ satisfy the piano-hypothesis, cf.\ Hypothesis \ref{hypo:piano}. Then, the following holds:
\begin{enumerate}
\item For every $\sigma\in {\rm Aut}(\C)$ and all critical values $\tfrac12+m\in\textrm{\emph{Crit}}(\Pi\times\Pi')=\textrm{\emph{Crit}}({}^\sigma\Pi\times{}^\sigma\Pi')$, there exists a non-zero constant $p(m,{}^\sigma\Pi_\infty,{}^\sigma\Pi'_\infty)$, only depending on $m$ and the choice of generators $[\Pi_\infty]$ and $[\Pi'_\infty]$ such that
$$
\sigma\left(
\frac{L^S(\tfrac 12+m,\Pi \times \Pi')}{p(\Pi)\ p(\Pi')\ p(m,\Pi_\infty,\Pi'_\infty) \ \mathcal{G}(\omega_{\Pi'_{f}})}\right) \ = \
\frac{L^S(\tfrac 12+m,{}^\sigma\Pi \times {}^\sigma\Pi')}{p({}^\sigma\Pi)\ p({}^\sigma\Pi')\ p(m,{}^\sigma\Pi_\infty,{}^\sigma\Pi'_\infty) \ \mathcal{G}(\omega_{{}^\sigma\Pi'_{f}})},
$$

\item Equivalently, interpreting both sides below as a family of numbers $x=\{x(\sigma)\}_{{\rm Aut}(\C)}$ and $y=\{y(\sigma)\}_{{\rm Aut}(\C)}$ as in Def.\ \ref{definition algebraic relation},
$$
L^{S}(\tfrac 12+m,\Pi \times \Pi') \ \sim_{\Q(\Pi)\Q(\Pi')} p(\Pi)\ p(\Pi') \ p(m,\Pi_\infty,\Pi'_\infty) \ \mathcal{G}(\omega_{\Pi'_{f}})
$$
is equivariant under $\textrm{\emph{Aut}}(\C)$.
\end{enumerate}
\end{thm}

\begin{rem}\label{remove Gauss sum}
We can replace the Gau\ss{} sum $\mathcal{G}(\omega_{\Pi'_{f}})$ by $ \mathcal{G}(\omega_{\Pi'_{f}}|_{\A_F^{+}})$ (see (\ref{remove Gauss sum 1}) and (\ref{remove Gauss sum 2}) below). In particular, when $\Pi'$ is conjugate self-dual, $\omega_{\Pi'_{f}}$ is trivial on $N_{\Acm/\Atr}(\Acm^{\times})$, so $\omega_{\Pi'_{f}}|_{\A_F^{+}}$ is either trivial or the finite part of the quadratic character $\varepsilon$ associate to the extension $F/F^{+}$. However, as $\Pi'$ is assumed to be cohomological, necessarily $\omega_{\Pi'_{f}}|_{\A_F^{+}}=\triv_f$. Therefore we can savely remove the Gau\ss{} sum $\mathcal{G}(\omega_{\Pi'_{f}})$ in Thm.\ \ref{thm:main-app} when $\Pi'$ is conjugate self-dual.
\end{rem}

\begin{cor}\label{cor:nonvan}
Let $\Pi$ and $\Pi'$ be as in Thm.\ \ref{thm:main-app} and assume in addition that $\Pi$ is unitary. Then $L^S(\tfrac 12,\Pi \times \Pi')$ is non-zero if and only if $L^S(\tfrac 12,{}^\sigma\Pi \times {}^\sigma\Pi')$ is non-zero for all $\sigma\in {\rm Aut}(\C)$.
\end{cor}
\begin{proof}
Recalling that under the present assumption on $\Pi$ and $\Pi'$, $s=\tfrac12$ is always critical for $L(s,\Pi\times \Pi')$, see \S \ref{sect:critpts}, the corollary follows directly from Thm.\ \ref{thm:main-app} (1).
\end{proof}

\section{Period relations for isobaric sums}\label{s:Eisen}

\subsection{Eisenstein representations and boundary cohomology}\label{se:eis1}
Let $\Pi'=\Pi_1\boxplus...\boxplus\Pi_k$ be an Eisenstein representation as in \ref{sect:Eisen}, cohomological with respect to $\EE_\mu$. Then there is a unique Kostant representative $w\in W^P$ (cf.\ \cite{bowa} III.1.4 \& Thm.\ III.3.3) such that $\Pi_i  \rho_i$ is cohomological with respect to the irreducible algebraic coefficient module $\EE_{\mu_w,i}:=\EE_{w(\mu+\rho)-\rho}|_{\GL_{n_i,\infty}}$. Since the point of evaluation of the Eisenstein series in $\Pi'$ is always centered at $\lambda=0$, see the proof of Lem.\ \ref{lem:eisen}, the length of this Kostant representative is $\ell(w)=\tfrac12\dim_\R N_{P,\infty}$, cf.\ \cite{bor2}, Lem.\ 2.12 and hence minimal by \cite{grobner-EisRes}, Prop.\ 12.\\\\
Let $\partial_{P}S_{G}:=P(F)\backslash G(\A_F)/K_\infty$ be the face corresponding to the parabolic subgroup $P\subseteq G$ in the Borel--Serre--compactification of $S_G$, cf.\ \cite{bor-serre}, \cite{rohlfs}. It is well-known (cf.\ \cite{schw1447}, 7.1--7.2) that there is an isomorphism of $G(\A_f)$-modules
\begin{eqnarray*}
H^q(\partial_{P}S_{G},\mathcal E_\mu) & \ira & {}^a\textrm{Ind}^{G(\A_f)}_{P(\A_f)}
\left [\bigoplus_{w\in W^{P}} H^{q-\ell(w)}(L(F)\backslash L(\A_F)/(K_{\infty}\cap L_\infty),\mathcal{E}_{\mu_w})\right ]
\end{eqnarray*}
``${}^a\textrm{Ind}$'' denoting un-normalized or algebraic induction. Hence the K\"unneth rule implies that there is furthermore an isomorphism of $G(\A_f)$-modules
\begin{equation}\label{eq:coh3}
I^q_{\partial_P}: H^q(\partial_{P}S_{G},\mathcal E_\mu) \ira {}^a\textrm{Ind}^{G(\A_f)}_{P(\A_f)}\left [\bigoplus_{w\in W^{P}} \bigoplus_{q_1+...+q_k=q-\ell(w)} \bigotimes_{i=1}^k H^{q_i}(S_{n_i},\mathcal{E}_{\mu_w,i})\right ],
\end{equation}
 \\\\
Let $W(A_P):=N_{G(F)}(A_P(F))/L(F)$, interpreted as the set of automorphisms $A_P\ira A_P$, which are given by conjugation by an element in $G(F)$ and which induce an isomorphism $L\ira L$. Hence, the elements ${\tilde w}\in W(A_P)$ act naturally on the inducing representation $\tau=\Pi_1\otimes...\otimes\Pi_k$ of $\Pi'$, by permuting its factors: $\tau^{\tilde w}=\Pi_{{\tilde w}^{-1}(1)}\otimes...\otimes\Pi_{{\tilde w}^{-1}(k)}$.\\\\
Let
$$res_P: H^{b_n}(S_G,\EE_\mu) \ra H^{b_n}(\partial_{P}S_{G},\mathcal E_\mu)$$
be the natural restriction of classes to the face $\partial_PS_G$ in the boundary of the Borel--Serre--compactification of $S_G$. It is obviously Aut$(\C)$-equivariant. The image under $res_P$ of a class $[\omega]\in H^{b_n}(S_n,\EE_\mu)(\Pi'_f)$, see \S\ref{mindeg}, is given by the class represented by the constant term along $P$ of the Eisenstein series representing $[\omega]$, cf.\ \cite{schwLNM}, Satz 1.10. Hence, recalling the well-defined (i.e., holomorphic at $s=0$) intertwining operators
$$M(\tau,w)=M(\tau,s,\tilde w)|_{s=0}: \textrm{Ind}^{G(\A_F)}_{P(\A_F)}\left [\tau \right ] \ra \textrm{Ind}^{G(\A_F)}_{P(\A_F)}\left [\tau^{\tilde w} \right ],$$
from \cite{moewal}, II.1.6 (holomorphy at $s=0$ following from \cite[IV.1.11]{moewal}) and the description of the constant term of an Eisenstein series $E_P(h,0)\in\Pi'$ from II.1.7, {\it ibidem},  \cite{bor2}, 2.9--2.13, implies that the image of $res_P([\omega])$ under $I^{b_n}_{\partial_P}$ lies inside the direct sum
\begin{equation}\label{eq:ccc}
I^{b_n}_{\partial_P}(res_P([\omega]))\in \bigoplus_{\tilde w\in W(A_P)} {}^a\textrm{Ind}^{G(\A_f)}_{P(\A_f)}\left [\bigotimes_{i=1}^k H^{b_{n_i}}(S_{n_i},\mathcal{E}_{\mu_w,i})((\Pi_{\tilde w^{-1}(i)}\rho_i)_f)\right ]
\end{equation}
Here we also used the minimality of the length $\ell(w)$ of the unique Kostant representative $w=w(\tau^{\tilde w},b_n)\in W^P$ giving rise to the coefficients modules $\mathcal{E}_{\mu_w,i}$ with respect to which $\Pi_{\tilde w^{-1}(i)}\rho_i$ is cohomological. However, since $id\in W(A_P)$ and since the attached intertwining operator $M(\tau,id)=\triv$ is the identity-map, \eqref{eq:ccc} implies that the composition $I^{b_n}_{\partial_P}\circ res_P$ induces an isomorphism of $G(\A_f)$-modules

\begin{equation}\label{eq:coh4}
r_{\Pi',P}: H^{b_n}(S_n,\EE_\mu)(\Pi'_f) \ira {}^a\textrm{Ind}^{G(\A_f)}_{P(\A_f)}\left [\bigotimes_{i=1}^k H^{b_{n_i}}(S_{n_i},\mathcal{E}_{\mu_w,i})((\Pi_{i}\rho_i)_f)\right ],
\end{equation}
by projecting onto the summand indexed by $\tilde w = id$. The following lemma is then obvious by construction.

\begin{lem}\label{lem:1}
The map $r_{\Pi',P}$ is ${\rm Aut}(\C)$-equivariant, i.e., for all $\sigma\in{\rm Aut}(\C)$,
$${}^a \textrm{\emph{Ind}}^{G(\A_f)}_{P(\A_f)} \left [\otimes_{i=1}^k\tilde\sigma^{b_{n_i}}\right]\circ r_{\Pi',P} = r_{{}^\sigma\Pi',P}\circ\tilde\sigma^{b_{n}}$$
\end{lem}

\subsection{A theorem on period relations for isobaric sums}\label{sect:app1}

In this section we want to relate the period $p(\Pi')$ of a cohomological Eisenstein representation $\Pi'=\Pi_1\boxplus ... \boxplus\Pi_k$ of $\GL_n(\A_F)$ as in \S\ref{sect:Eisen} with the product of the periods $p(\Pi_i\rho_i)$ of its (twisted) summands $\Pi_i\rho_i$. Firstly, we observe that the latter periods are defined: Indeed, $\Pi'$ being cohomological implies by \cite{bowa}, Thm.\ III.3.3, that $\Pi_i\rho_i$ is cohomological for all $1\leq i \leq k$, whence each $p(\Pi_i\rho_i)$ exists by our Cor.\ \ref{cor:Whittakerperiods}.\\\\
In order to compare $p(\Pi')$ with the product $\prod_{1\leq i \leq k} p(\Pi_i\rho_i)$, we recall that we have fixed choices of generators $[\Pi'_{\infty}]$ and $[(\Pi_i\rho_i)_{\infty}]$ in \S \ref{sect:gens}, see Conv.\ \ref{conv:2} and Conv.\ \ref{diag}, with the effect that each Whittaker period is now uniquely determined up to non-zero multiples in the respective rationality field. These choices were purely local at the archimedean places. We will now consider a special, global cohomolgy class in $H^{b_n}(S_n,\EE_\mu)(\Pi'_f)$, which will be used in the proof of our Thm.\ \ref{thm:periods}.\\\
To this end, we start with a non-trivial section $\Omega\in {}^a\textrm{Ind}^{G(\A_f)}_{P(\A_f)}\left [\bigotimes_{i=1}^k H^{b_{n_i}}(S_{n_i},\mathcal{E}_{\mu_w,i})((\Pi_{i}\rho_i)_f)\right ]$, which we view as the image $\Omega =r_{\Pi',P}([\omega])$ of a uniquely determined Eisenstein cohomology class $[\omega]\in H^{b_n}(S_n,\EE_\mu)(\Pi'_f)$. By construction, a sheaf-theoretical differential form $\omega$ representing the latter is of the form
$$\omega= \sum_{\substack{\underline i =(i_1,...,i_{b_n})\\ \alpha}} \left( E_P(h_{\infty, \underline i,\alpha}\otimes h_f,\lambda)|_{\lambda=0}\otimes e_\alpha \right) \ dx_{\underline i}, $$
for appropriate $K_\infty$-finite, global sections $h_{\infty, \underline i,\alpha}\otimes h_f\in\textrm{Ind}^{G(\A_F)}_{P(\A_F)}\left [\bigotimes_{i=1}^k \Pi_{i}\right]$. Hence, evaluating our section $\Omega$ at $g_f\in G(\A_f)$ defines a cohomology class $\Omega(g_f)$ in $\bigotimes_{i=1}^k H^{b_{n_i}}(S_{n_i},\mathcal{E}_{\mu_w,i})((\Pi_{i}\rho_i)_f)$, which is represented by a differential form of the type
$$\bigotimes_{i=1}^k\sum_{\substack{\underline j =(j_1,...,j_{b_{n_i}})\\ \alpha_i}} \left( (\varphi_{\infty, \underline j,\alpha_i}\otimes h_f(g_f))\otimes e_{\alpha_i}\right) \ dx_{\underline j},$$
where $\varphi_{\infty, \underline j,\alpha_i}\otimes h_f(g_f)$ is a cuspidal automorphic form in $\Pi_i\rho_i$ for all $g_f$, $\underline j$ and $\alpha_i$. Observe that its non-archimedean component $h_f(g_f)$ is independent of the indices $\underline j$ and $\alpha_i$, while $\varphi_{\infty, \underline j,\alpha_i}$ is independent of $g_f$. \\\\
Now suppose that $h_f=\otimes_{v\notin S_\infty} h_v$ is a factorizable vector, whose local components are normalized newvectors. Here, by a newvector we understand a vector which is invariant under the mirahoric subgroup $K(m_v)\subset G(\O_v)$ as defined in \cite[Thm.\ (5.1)]{jac-ps-shalika-mathann}, $m_v$ being the conductor of $\textrm{Ind}^{G(F_v)}_{P(F_v)}\left [\bigotimes_{i=1}^k (\Pi_{i})_v\right]$. We normalize the local component $h_v$ by the condition $h_v(id_{G(F_v)})(id_{L_P(F_v)})=1$. Observe that this pins down the whole section $h_v$ in a unique way.

In order to fix $\Omega$, we still need to treat the archimedean contribution to $\Omega$. To this end observe that after applying the inverse of the last four maps described in Prop.\ \ref{prop:Eisen} to the class $\Omega(id_{G(\A_f)})$, we obtain a cohomology class in
$\bigotimes_{i=1}^k H^{b_{n_i}}(\g_{n_i,\infty},K_{n_i,\infty}, W((\Pi_i\rho_i)_\infty)\otimes \EE_{\mu_w,i})\otimes W((\Pi_i\rho_i)_f)$, and we assume that this class equals
\begin{equation}\label{eq:ccd}
\bigotimes_{i=1}^k [(\Pi_i\rho_i)_{\infty}] \otimes \xi_{h_f(id_{G(\A_f)})},
\end{equation}
for our fixed choice of generators $ [(\Pi_i\rho_i)_{\infty}]$ (given by our Conv.\ \ref{conv:2} and Conv.\ \ref{diag}) and the Whittaker function $\xi_{h_f(id_{G(\A_f)})}$ attached to $h_f(id_{G(\A_f)})$ (cf.\ \cite{shahidi-book}, p.\ 122). This finally pins down a global section $\Omega$. Now, as our special global cohomology class $[\omega]\in H^{b_n}(S_n,\EE_\mu)(\Pi'_f)$ we take the preimage $[\omega]:=r_{\Pi',P}^{-1}(\Omega)$. By \eqref{eq:ccd} it depends on $\bigotimes_{i=1}^k [(\Pi_i\rho_i)_{\infty}]$, which in turn has been fixed by our Conv.\ \ref{conv:2} and Conv.\ \ref{diag}. \\\\
We are now ready to prove the following theorem:

\begin{thm}\label{thm:periods}
Let $\Pi'=\Pi_1\boxplus ... \boxplus\Pi_k$ be an Eisenstein representation of $\GL_n(\A_F)$ as in \S\ref{sect:Eisen}, which is cohomological with respect to the irreducible algebraic representation $\EE_{\mu}$ of $G_\infty$. Then,
\begin{enumerate}
\item for all $\sigma\in${\rm Aut}$(\C)$
$$\sigma\left(\frac{p(\Pi') \prod_{1\leq i < j \leq k} L^S(1,\Pi_i\times\Pi^\vee_j)^{-1}}{\prod_{1\leq i \leq k} p(\Pi_i\rho_i)}\right)= \frac{p({}^\sigma\Pi') \prod_{1\leq i < j \leq k} L^S(1,\Pi^\sigma_i\times(\Pi_j^\sigma)^{\vee})^{-1}}{\prod_{1\leq i \leq k} p({}^\sigma(\Pi_i\rho_i))}.$$
\item Equivalently,
$$p(\Pi') \sim_{\prod_{i}\Q(\Pi_i\rho_i)} \prod_{1\leq i \leq k} p(\Pi_i\rho_i) \prod_{1\leq i < j \leq k} L^S(1,\Pi_i\times\Pi^\vee_j)$$
is ${\rm Aut}(\C)$-equivariant. If all summands $\Pi_i$ are conjugate self-dual then the same relation holds over the smaller field $\Q(\Pi')$.
\end{enumerate}
\end{thm}
\begin{proof}
We proceed in several steps.\\
{\it Step 1:} Taking inverses, we obtain two isomorphisms of $G(\A_f)$-modules
$$\Delta_G: H^{b_n}(S_n,\EE_\mu)(\Pi'_f) \ira W(\Pi'_f)$$
defined as
$$\Delta_G:= p(\Pi')\cdot \Upsilon^{-1}_{\Pi'}$$
and
$$\Delta_P: {}^a\textrm{Ind}^{G(\A_f)}_{P(\A_f)}\left [\bigotimes_{i=1}^k H^{b_{n_i}}(S_{n_i},\mathcal{E}_{\mu_w,i})((\Pi_{i}\rho_i)_f)\right ]\ira {}^a\textrm{Ind}^{G(\A_f)}_{P(\A_f)}\left [\bigotimes_{i=1}^k W((\Pi_{i}\rho_i)_f) \right]$$
defined as
$$\Delta_P:= \prod_{i=1}^k p(\Pi_i  \rho_i) \cdot {}^a\textrm{Ind}^{G(\A_f)}_{P(\A_f)}
\left [\otimes_{i=1}^k \Upsilon_{\Pi_i\rho_i}^{-1}\right ]$$
which are Aut$(\C)$-equivariant by definition of the periods. We will next consider their effect on the particular class $[\omega]\in H^{b_n}(S_n,\EE_\mu)(\Pi'_f)$, which we defined above in \S\ref{sect:app1}. Firstly, we apply the inverse of the composition of the last three maps described in Prop.\ \ref{prop:Eisen} to $[\omega]$. It then follows from \cite{shahidi-book}, Prop.\ 7.1.3 that the image of $[\omega]$ in
$H^{b_n}(\g_\infty,K_\infty,W(\Pi')\otimes\EE_\mu)$ equals
$$\sum_{\substack{\underline i =(i_1,...,i_{b_n})\\ \alpha}} X^*_{\underline i}\otimes \left(W_{h_{\infty, \underline i,\alpha}}\otimes W_{h_f}\right) \otimes e_\alpha,$$
where we recall that the Whittaker functions $W_{h_{\infty, \underline i,\alpha}}$ resp.\ $W_{h_f}$ are defined trough their local components as in \cite{shahidi-book}, (7.1.2): In particular, by definition, at the archimedean places we just retrieve the type of Whittaker functions, which we already considered in \eqref{eq:archwhit} (and which we denoted in the same way). As the isomorphisms $I^P_{\pi_\infty}$ from Borel-Wallach and $\mathcal W^{\ast}_P$ from Shahidi, cf.\ \S\ref{sect:gens}, commute with induction in stages, \eqref{eq:ccd} implies that indeed $W_{h_{\infty, \underline i,\alpha}}=\xi_{\infty, \underline i,\alpha}$, in the notation from \S\ref{mindeg}, i.e., the vectors $W_{h_{\infty, \underline i,\alpha}}$ are those, which appear in the generator $[\Pi'_\infty]$. So, finally, $\Delta_G([\omega])$ equals the Whittaker functional
$$\Delta_G([\omega]): g_f\mapsto p(\Pi')  \ W_{h_f}(g_f)=p(\Pi')  \ \prod_{v\notin S_\infty} W_{h_v}(g_v).$$
We note that by our choice of $h_f$ we obtain
$$\prod_{v\notin S} W_{h_v}(id_{G(F_v)}) = \prod_{1\leq i < j \leq k} L^S(1,\Pi_i\times\Pi_j^\vee)^{-1},$$
and $\prod_{v\in S\backslash S_\infty} W_{h_v}(id_{G(F_v)}) \neq 0$: The first statement is the contents of \cite{shahidi-book}, Prop.\ 7.1.4, while the second is shown in \cite[Cor.\ 4.4]{miyauchi}.\\
Let us write ${}^\sigma\! h_f$ for the corresponding section giving rise to the cohomology-class $\tilde\sigma^{b_n}([\omega])=\tilde\sigma^{b_n}\left(\left[\sum_{\substack{\underline i =(i_1,...,i_{b_n})\\ \alpha}} \left( E_P(h_{\infty, \underline i,\alpha}\otimes h_f,0)\otimes e_\alpha \right) \ dx_{\underline i}\right]\right)$. Then, Lem.\ \ref{lem:1} shows that this vector ${}^\sigma\! h_f$ is the non-archimedean component vector of ${}^a \textrm{Ind}^{G(\A_f)}_{P(\A_f)} \left [\otimes_{i=1}^k\tilde\sigma^{b_{n_i}}\right](r_{\Pi',P}([\omega]))$, and hence we can read of the action of $\sigma\in$Aut$(\C)$ on the function $h_v$ 
directly in terms of the actions on its values, see \cite{grob-harr}, Rem.\ 2.6. In particular, ${}^\sigma\! h_v(id_{G(F_v)})(id_{L_P(F_v)})=\sigma(h_v(id_{G(F_v)})(id_{L_P(F_v)}))=1$ at $v\notin S$, so
$$\prod_{v\notin S} W_{{}^\sigma\! h_v}(id_{G(F_v)}) = \prod_{1\leq i < j \leq k} L^S(1,\Pi_i^\sigma\times(\Pi_j^\sigma)^\vee)^{-1}$$
where we abbreviated $\Pi_i^\sigma:={}^\sigma(\Pi_i\rho_i)\rho^{-1}_i$ for the cuspidal isobaric summands of ${}^\sigma\Pi'$ as in \eqref{eq:isotwist}.
Analogously to above, we obtain $\prod_{v\in S\backslash S_\infty} W_{{}^\sigma\! h_v}(id_{G(F_v)}) \neq 0$. Recall the diagonal matrix $t_{\sigma,v}={\rm diag}(x_1,...,x_{n-1},1)\in T(\O_{F_v})$ from \S\ref{sect:whittaker}. Obviously, $t_{\sigma,v}$ is contained in every mirabolic subgroup $K_G(m_v)$. Hence, by definition, the $\sigma$-twisted newvector satisfies $\tilde\sigma_{\Pi_f}( W_{h_v})(id_{G(F_v)})=\sigma(W_{h_v}(id_{G(F_v)}))$, see again \S\ref{sect:whittaker}. The Aut$(\C)$-equivariance of $\Delta_G$ hence finally shows that
\begin{equation}\label{eq:pr1}
p({}^\sigma\Pi') \prod_{v\in S\backslash S_\infty} W_{{}^\sigma\! h_v}(id_{G(F_v)}) \prod_{1\leq i < j \leq k} L^S(1,\Pi_i^\sigma\times(\Pi_j^\sigma)^\vee)^{-1} =
\end{equation}
$$\sigma\left( p(\Pi') \prod_{v\in S\backslash S_\infty} W_{h_v}(id_{G(F_v)}) \prod_{1\leq i < j \leq k} L^S(1,\Pi_i\times\Pi_j^\vee)^{-1}\right).
$$
{\it Step 2}: The Whittaker function $W_{h_v}\in W(\Pi'_v)=W({}^a\textrm{Ind}^{G(F_v)}_{P(F_v)}\left [\bigotimes_{i=1}^k (\Pi_{i}\rho_i)_v \right])$, $v\notin S$, from above shall not be confused with the spherical Whittaker function $\xi_{h_v(id_{G(F_v)})}\in W(\bigotimes_{i=1}^k (\Pi_{i}\rho_i)_v)$, which is $1$ on $id_{L_P(F_v)}$ and attached to the local spherical vector $h_v(id_{G(F_v)})\in \bigotimes_{i=1}^k (\Pi_{i}\rho_i)_v$. Applying $\Delta_P$ to the restriction $r_{\Pi',P}([\omega])$, with $[\omega]$ as above, the Aut$(\C)$-equivariance of $\Delta_P$ and our standing assumptions Conv.\ \ref{conv:2} and Conv.\ \ref{diag}, hence show that

\begin{equation}\label{eq:pr2}
\prod_{1\leq i \leq k} p({}^\sigma(\Pi_i\rho_i)) \prod_{v\in S\backslash S_\infty} \xi_{{}^\sigma\! h_v(id_{G(F_v)})}(id_{L_P(F_v)})  = \sigma\left(\prod_{1\leq i \leq k} p(\Pi_i\rho_i)\prod_{v\in S\backslash S_\infty} \xi_{h_v(id_{G(F_v)})}(id_{L_P(F_v)}) \right).
\end{equation}
Invoking our standing assumptions Conv.\ \ref{conv:2} and Conv.\ \ref{diag}, once more, one shows that
$$\prod_{v\in S\backslash S_\infty} \frac{W_{{}^\sigma\! h_v}(id) }{\xi_{{}^\sigma\! h_v(id)}(id)}= \sigma\left( \prod_{v\in S\backslash S_\infty} \frac{W_{h_v}(id) }{\xi_{h_v(id)}(id)}\right)$$
for all $\sigma\in$ Aut$(\C)$. Hence, dividing \eqref{eq:pr1} by \eqref{eq:pr2} yields
\begin{equation}\label{eq:step1}
\frac{p({}^\sigma\Pi') \prod_{1\leq i < j \leq k} L^S(1,\Pi_i^\sigma\times(\Pi_j^\sigma)^\vee)^{-1}}{\prod_{1\leq i \leq k} p({}^\sigma(\Pi_i\rho_i)) } = \sigma\left( \frac{p(\Pi') \prod_{1\leq i < j \leq k} L^S(1,\Pi_i\times\Pi_j^\vee)^{-1}}{\prod_{1\leq i \leq k} p(\Pi_i\rho_i)} \right)
\end{equation}
for all $\sigma\in$Aut$(\C)$. This proves (1).\\\\
{\it Step 3:} The first assertion of (2) follows from (1) (and is in fact equivalent to it) by Strong Multiplicity One and Multiplicity One for cuspidal automorphic representation. For the second assertion, observe that ${}^\sigma\Pi'_f\cong \Pi'_f$ implies that ${}^\sigma\Pi'=\Pi'$ (equality!) due to Strong Multiplicity One (\cite{jacshal2}, Thm.\ 4.4) together with Multiplicity One (\cite{schwfr} \S 3.3 \& \cite{shalika, jaclang}) for isobaric automorphic representations. Hence, $p({}^\sigma\Pi')=p(\Pi')$ for such $\sigma$. Moreover, as ${}^\sigma\Pi'\cong\Pi'$ if and only if $\{\Pi^\sigma_1,...,\Pi^\sigma_k\}=\{\Pi_1,...,\Pi_k\},$ see \S\ref{sect:Eisen},
$$\prod_{1\leq i<j\leq k} L^S(1,\Pi^\sigma_i\times(\Pi_j^\sigma)^{\vee})=\prod_{1\leq i<j\leq k} L^S(1,\Pi_i\times(\Pi_j)^{\vee})$$
for all such $\sigma$, because the $\Pi_i$ are now all conjugate self-dual. 
So, finally, in order to reduce the above relation to $\Q(\Pi')$, we need to show that ${}^\sigma\Pi'\cong\Pi'$ also implies that
\begin{equation}\label{eq:prod}
\prod_{1\leq j \leq k} p(\Pi_j\rho_j)=\prod_{1\leq i \leq k} p({}^\sigma(\Pi_i\rho_i)).
\end{equation}
To this end, let $i$, $1\leq i \leq k$, be arbitrary. Then, as ${}^\sigma\Pi'\cong\Pi'$ by assumption, there is a unique $j$, $1\leq j \leq k$, such $\Pi^\sigma_i\cong \Pi_j$. This implies that $n_i=n_j$ and hence \eqref{eq:ai} and Multiplicity One shows that ${}^\sigma(\Pi_i\rho_i)=\Pi_j\rho_j\cdot\|\cdot\|^b$ for some integer $b\in\Z$, whence $p({}^\sigma(\Pi_i\rho_i))= p(\Pi_j\rho_j\cdot\|\cdot\|^b)$. But since ${}^\sigma\|\cdot\|^b=\|\cdot\|^b$ for all $\sigma\in$ Aut$(\C)$,
$$p(\Pi_j\rho_j\cdot\|\cdot\|^b)= p(\Pi_j\rho_j)$$
by construction, see Prop.\ \ref{prop:Whittakerperiods}. Hence, finally, we see that if $\sigma\in$ Aut$(\C)$ is any automorphism such that $\{\Pi^\sigma_1,...,\Pi^\sigma_k\}=\{\Pi_1,...,\Pi_k\}$, then \eqref{eq:prod} holds. This shows the last claim.
\end{proof}

\begin{rem}
The above theorem also holds, when the isobaric summands of $\Pi'$ are isobaric sums themselves: Let $1\leq \ell\leq k$ and let $\bigcup_{1\leq j\leq  \ell}\{i_1,...,i_j\}=\{1,...,k\}$ be a partition of $\{1,...,k\}$ into $\ell$ disjoint sets. We set $\ell_j:=n_{i_1}+...+n_{i_j}$ and let $P'$ be the standard parabolic subgroup of $\GL_n$ with Levi factor equal to $L'\cong\prod_{1\leq j\leq \ell} \GL_{\ell_j}$. Set $\Pi'_j:=\Pi_{i_1}\boxplus...\boxplus\Pi_{i_j}$ and $\rho'_j$ to be equal to the restriction of $\rho_{P'}$ to $\GL_{\ell_j}(\A_F)$. With only a little more work one can in fact show that
$$p(\Pi') \sim_{\prod_{1\leq j \leq \ell}\Q(\Pi'_j\rho'_j)} \prod_{1\leq j \leq \ell} p(\Pi'_j\rho'_j) \prod_{1\leq i < j \leq \ell} L^S(1,\Pi'_i\times\Pi'^\vee_j)$$
is ${\rm Aut}(\C)$-equivariant.
\end{rem}

Let $\Pi'=\Pi_1\boxplus ... \boxplus\Pi_k$ be a cohomological isobaric sum as in \S\ref{sect:Eisen} and assume that all cuspidal summands $\Pi_i$ are conjugate self-dual. Put
\begin{equation}\label{eq:pialg}
\Pi^{\sf alg}_i:=\left\{\begin{array}{ll}
 \Pi_i & \textrm{if $n\equiv n_i \mod 2$,} \\
 \Pi_i\otimes\eta & \textrm{otherwise.}
\end{array}
\right.
\end{equation}
These are unitary, conjugate self-dual, cohomological cuspidal automorphic representations, which can be seen as follows: $\Pi'$ being cohomological by assumption implies that $\Pi_i\rho_i$ is cohomological for all $i$. Now, recalling that $\rho_i$ is algebraic if and only if $n\equiv n_i \mod 2$ shows that $\Pi_i$ is cohomological itself if and only if $n$ and $n_i$ have the same parity, or, otherwise said, that $\Pi_i\otimes\eta$ is cohomological if and only if $n$ and $n_i$ do not have the same parity. Since $\eta$ is unitary and conjugate self-dual, this shows that $\Pi^{\sf alg}_i$ is unitary, conjugate self-dual, cohomological cuspidal for all $1\leq i\leq k$. The above theorem now yields the following corollary.

\begin{cor}\label{isobaric sum Whittaker corollary}
Let $\Pi'=\Pi_1\boxplus ... \boxplus\Pi_k$ be a cohomological isobaric sum as in \S\ref{sect:Eisen} and assume that all cuspidal summands $\Pi_i$ are conjugate self-dual. Then,
$$p(\Pi') \sim_{E(\Pi')E(\phi)}  \ \prod_{1\leq i \leq k} p(\Pi^{\sf alg}_i) \prod_{1\leq i < j \leq k} L^S(1,\Pi_i\times\Pi^\vee_j),$$
is $\textrm{\emph{Aut}}(\C/F^{Gal})$-equivariant.
\end{cor}
\begin{proof}
By Thm.\ \ref{thm:periods}
$$p(\Pi') \sim_{\Q(\Pi')} \prod_{1\leq i \leq k} p(\Pi_i\rho_i) \prod_{1\leq i < j \leq k} L^S(1,\Pi_i\times\Pi^\vee_j).$$
One easily checks that \eqref{eq:ai} implies that $\Pi_i\rho_i=\Pi^{\sf alg}_i\phi^{e_i}\cdot \|\cdot\|^{b_i}$ for some $b_i\in\Z$ and
$$e_i=\left\{ \begin{array}{ll}
 0 & \textrm{if $n\equiv n_i \mod 2$,} \\
 -1 & \textrm{otherwise.}
\end{array}
\right.$$
Hence, \cite{raghuram-shahidi-imrn}, Thm.\ 4.1, shows that $p(\Pi_i\rho_i)\sim_{\Q(\Pi^{\sf alg}_i)\Q(\phi^{e_i})} p(\Pi^{\sf alg}_i) \mathcal G(\phi^{e_i}_f)^{n_i(n_i-1)/2}$ observing that $\mathcal G(\|\cdot\|_f^{b_i})=1$ and $\Q(\|\cdot\|^{b_i})=\Q$. At the cost of adjusting $p(\Pi^{\sf alg}_i)$ by an element in $\Q(\Pi^{\sf alg}_i)$, we may hence assume that $p(\Pi_i\rho_i)\sim_{\Q(\phi^{e_i})} p(\Pi^{\sf alg}_i) \mathcal G(\phi^{e_i}_f)^{n_i(n_i-1)/2}$. In order to finish the proof, it is hence enough to show that $\mathcal G(\phi^{-1}_f)\in E(\phi)$. To this end, let $\sigma\in$ Aut$(\C)$. There exists $t^{\sigma} \in \hat{\Z}^{\times}\subseteq\hat{\mathcal O}^\times_{F^+}\subset\hat{\mathcal O}^\times_{F}\subset\A^\times_{f}$ associated to $\sigma$ given by the cyclotomic character (cf.\ \S 3.2 of \cite{raghuram-shahidi-imrn}). It is easy to verify that (cf., e.g., the proof of Theorem 3.9 of \cite{grob-harr})
\begin{equation}\label{remove Gauss sum 1}
\cfrac{\sigma(\mathcal G(\phi^{-1}_f))}{\mathcal G({}^{\sigma}\phi^{-1}_f)}=\phi^{-1}_f (t^{\sigma}).
\end{equation}

Since $t^{\sigma}\in \hat{\Z}$, we know $\phi^{-1}_f (t^{\sigma}) = \phi\mid_{\Atr,f}^{-1} (t^{\sigma})$. Recall that $\phi=\eta\ \|\cdot\|^{1/2}$ and $\eta\mid_{\Atr}=\varepsilon$ by definition, so $\phi^{-1}_f (t^{\sigma})=\varepsilon_{f} (t^{\sigma})$. This implies that
\begin{equation}\label{remove Gauss sum 2}
\cfrac{\sigma(\mathcal G(\phi^{-1}_f))}{\mathcal G({}^{\sigma}\phi^{-1}_f)}=\cfrac{\sigma(\mathcal G(\varepsilon_f))}{\mathcal G({}^{\sigma}\varepsilon_f)},
\end{equation}
and so $\mathcal G(\phi^{-1}_f) \sim_{\Q(\phi)\Q(\varepsilon)} \mathcal G(\varepsilon_f)$. However, $\varepsilon$ being quadratic forces $\Q(\varepsilon)=\Q(\{\pm1\})=\Q$, so we finally
conclude by Lemma \ref{Gauss sum} that $\mathcal G(\phi^{-1}_f) \sim_{\Q(\phi)} \mathcal G(\varepsilon_f) \sim_{F^{Gal}} 1$ and hence $\mathcal G(\phi^{-1}_f)\in E(\phi)=\Q(\phi)\cdot F^{Gal}$.
\end{proof}


\section{Asai $L$-functions of Langlands transfers}\label{sect:asai}

In this section, we will calculate the Asai $L$-functions of conjugate self-dual Eisenstein representations and of a representation automorphically induced from suitable Hecke characters. Having detailed knowledge about these Asai $L$-functions will turn out to be crucial for the proof of the main theorems.

\subsection{The Asai $L$-function of an isobaric sum}\label{sect:AsaiISobaric}
Let $S^+$ be the finite set of places of $\tr$, containing the restrictions of the places in $S$. Let $\Pi'=\Pi_1\boxplus...\boxplus\Pi_k$ be an Eisenstein representation of $\GL_{n}(\Acm)$ as in \ref{sect:Eisen} and assume that each $\Pi_i$ is conjugate self-dual. With these assumptions, unitary conjugate self-dual, cohomological cuspidal automorphic representations $\Pi_i^{\sf alg}$ have been defined for all $1\leq i\leq k$ in \eqref{eq:pialg}. We obtain

\begin{lem}\label{lem:poles_asai_lfunction}
$L^S(s,\Pi_i^{\sf alg},{\rm As}^{(-1)^{n_i}})$ is holomorphic and non-vanishing at $s=1$.
\end{lem}
\begin{proof}
As all $\Pi_i^{\sf alg}$ are cohomological and conjugate self-dual unitary cuspidal automorphic representation of $\GL_{n_i}(\A_F)$, $L^S(s,\Pi_i^{\sf alg},{\rm As}^{(-1)^{n_i}})$ is holomorphic at $s=1$ by the argument given in \cite{mok}, Cor.\ 2.5.9. (See also \cite{grob_harris_lapid}, \S 6.1 for more details.) The non-vanishing follows from \cite{shahidi_certainL}, Thm.\ 5.1.
\end{proof}

\begin{rem}
Observing that twisting by $\eta$ just changes the sign of the Asai-representation, we see that $L^S(s,\Pi_i,{\rm As}^{(-1)^{n}})$,
being equal to $L^S(s,\Pi_i^{\sf alg},{\rm As}^{(-1)^{n_i}})$, is holomorphic and non-vanishing at $s=1$ for all $1\leq i\leq k$.
\end{rem}

\begin{lem}\label{lem:fact_asai_lfunction}
Let $\Pi'=\Pi_1\boxplus...\boxplus\Pi_k$ be an Eisenstein representation of $\GL_{n}(\Acm)$ as in \ref{sect:Eisen} and assume the each $\Pi_i$ is conjugate self-dual. Then
$$L^S(s,\Pi',{\rm As}^{(-1)^{n}})=\prod_{i=1}^k L^S(s,\Pi_i^{\sf alg},{\rm As}^{(-1)^{n_i}})\ \cdot \prod_{1\leq i<j\leq k} L^S(s,\Pi_i\times\Pi^\vee_j).$$
\end{lem}
\begin{proof}
Recalling the previous remark, we will show that $$L^S(s,\Pi',{\rm As}^{(-1)^{n}})=\prod_{i=1}^k L^S(s,\Pi_i,{\rm As}^{(-1)^{n}})\ \cdot \prod_{1\leq i<j\leq k} L^S(s,\Pi_i\times\Pi^\vee_j).$$ We argue locally, distinguishing the two possible types of unramified places $v\notin S^+$. In particular, we will drop the subscript for the local places for all global objects in what follows if this will not cause any confusion. Moreover, let $\gamma$ be the character from \cite{grob-harr}, \S 2.1.2, i.e.,
$$\gamma=\left\{\begin{array}{ll}
 \triv & \textrm{if $n\equiv 0 \mod 2$} \\
 \eta & \textrm{otherwise}
\end{array}
\right.$$
Observe that $\gamma|_{F^+}^2=\triv$ in any case. Let us first suppose that $v$ is inert, extending to one place $w\notin S$. For every $1\leq i\leq k$, denote by $\chi^{(i)}_{\ell}$, $1\leq\ell\leq n_i$, the unramified characters from which $\Pi_i$ is fully induced. Then, by induction in stages and applying Lemma $1$ of \cite{gelb_jac_rog},

\begin{eqnarray*}
&& L(s,\Pi'_v,{\rm As}^{(-1)^{n}})\nonumber\\
 & = & \left(\textrm{$\prod_{i,\ell}$} L(s,\chi^{(i)}_{\ell}\gamma|_{\tr})\right) \cdot \left(\textrm{$\prod_{\substack{i,j,r,t\\
\textrm{$\chi^{(i)}_{r}$ coming before $\chi^{(j)}_{t}$}}}$} L(s,\chi^{(i)}_{r}\gamma\cdot\chi^{(j)}_{t}\gamma) \right)\\ & = & \prod_{i=1}^k\left( \prod_{\ell=1}^{n_i} L(s,\chi^{(i)}_{\ell}\gamma|_{\tr})\cdot\prod_{1\leq r<t\leq n_i} L(s,\chi^{(i)}_{r}\chi^{(i)}_{t})\right)\cdot \prod_{1\leq i<j\leq k}\prod_{\substack{1\leq r\leq n_i\\1\leq t\leq n_j}} L(s,\chi^{(i)}_{r}\chi^{(j)}_{t})\\
& = & \prod_{i=1}^k L(s,\Pi_{i,v},{\rm As}^{(-1)^{n}})\ \cdot \prod_{1\leq i<j\leq k} L(s,\Pi_{i,w}\times\Pi_{j,w}).
\end{eqnarray*}
As $v$ is inert and all $\Pi_j$ are conjugate self-dual, the latter local Rankin-Selberg $L$-function is equal to $L(s,\Pi_{i,w}\times\Pi^\vee_{j,w})$ and hence we obtain the desired relation at inert places $v$.
If $v$ is split, extending to a product of two places $w_1,w_2\notin S$, then $\gamma_{w_1}\gamma_{w_2}=\triv$ and the residue fields are of equal cardinality $q_v=q_{w_1}=q_{w_2}$. So we get independently of $n$,
\begin{eqnarray*}
&& L(s,\Pi'_v,{\rm As}^{(-1)^{n}}) \nonumber\\
& = & \det(id - (A(\Pi'_{w_1})\otimes A(\Pi'_{w_2})) q_{v}^{-s})^{-1}\\
& = & \prod_{1\leq i,j\leq k}\prod_{\substack{1\leq r\leq n_i\\1\leq t\leq n_j}}(1 - \chi^{(i)}_{r,w_1}\chi^{(j)}_{t,w_2}q_{v}^{-s})^{-1} \\
& = & \prod_{i=1}^k\left(\prod_{\substack{1\leq r\leq n_i\\1\leq t\leq n_j}}(1 - \chi^{(i)}_{r,w_1}\chi^{(i)}_{t,w_2}q_{v}^{-s})^{-1}\right)\cdot \prod_{\substack{1\leq i\neq j\leq k\\1\leq r\leq n_i\\1\leq t\leq n_j}}(1 - \chi^{(i)}_{r,w_1}\chi^{(j)}_{t,w_2}q_{v}^{-s})^{-1}\\
& = & \prod_{i=1}^k L(s,\Pi_{i,v},{\rm As}^{(-1)^{n}}) \cdot \prod_{\substack{1\leq i<j\leq k\\1\leq r\leq n_i\\1\leq t\leq n_j}}(1 - \chi^{(i)}_{r,w_1}\chi^{(j)}_{t,w_2}q_{v}^{-s})^{-1}(1-\chi^{(i)}_{r,w_2}\chi^{(j)}_{t,w_1}q_{v}^{-s})^{-1}\\
& = & \prod_{i=1}^k L(s,\Pi_{i,v},{\rm As}^{(-1)^{n}}) \cdot \prod_{\substack{1\leq i<j\leq k\\1\leq r\leq n_i\\1\leq t\leq n_j}}(1 - \chi^{(i)}_{r,w_1}\overline\chi^{(j)}_{t,w_1}q_{w_1}^{-s})^{-1}(1-\chi^{(i)}_{r,w_2}\overline\chi^{(j)}_{t,w_2}q_{w_2}^{-s})^{-1}\\
& = & \prod_{i=1}^k L(s,\Pi_{i,v},{\rm As}^{(-1)^{n}}) \cdot \prod_{1\leq i<j\leq k} L(s,\Pi_{i,w_1}\times\overline\Pi_{j,w_1}) L(s,\Pi_{i,w_2}\times\overline\Pi_{j,w_2}).
\end{eqnarray*}
Since all $\Pi_j$ are conjugate self-dual, this shows the claim at split places $v$.
\end{proof}

\begin{cor}\label{cor:asai-nonvan}
Let $\Pi'=\Pi_1\boxplus...\boxplus\Pi_k$ be an Eisenstein representation of $\GL_{n}(\Acm)$ as in \ref{sect:Eisen} and assume the each $\Pi_i$ is conjugate self-dual. Then $L^S(s,\Pi',{\rm As}^{(-1)^{n}})$ is holomorphic and non-vanishing at $s=1$.
\end{cor}
\begin{proof}
Follows directly from Lem.\ \ref{lem:fact_asai_lfunction} and \ref{lem:poles_asai_lfunction}, recalling that by definition all $\Pi_i$ are pairwise different and unitary.
\end{proof}

\subsection{The Asai $L$-function of an automorphically induced representation}\label{sect:AsaiInduced}

Let $L$ be a cyclic extension of $\cm$ of degree $[L:\cm]=n\geq 1$ which is still a CM field. We write $L^{+}$ for its maximal totally real subfield.
\begin{defn}
Let $\chi$ be an algebraic Hecke character of $L$. We let $\Pi(\chi)$ be the automorphic induction of $\chi$ to $\GL_{n}(\Acm)$ as established in \cite{arthur-clozel}, Chp.\ 3, Thm.\ 6.2. We write
\smallskip

\begin{center}
\framebox{$\Pi_{\chi}:=\left\{\begin{array}{ll}
 \Pi(\chi) & \textrm{if $n=[L:\cm]$ is odd,} \\
 \Pi(\chi)\otimes \eta & \textrm{if $n=[L:\cm]$ is even.}
\end{array}
\right.$}

\end{center}
\end{defn}

Then $\Pi_{\chi}$ is an isobaric automorphic representation of $\GL_{n}(\Acm)$, fully induced from cuspidal automorphic representations, which is algebraic (in the sense of \cite{clozel}, Def.\ 1.8). We remark that the extremely careful reader might bear in mind that \cite{arthur-clozel} contains an ``{\it egregious mistake}'' (a quote of Clozel himself, taken from \cite{clozel_ias}) if $[L:\cm]=n$ is not prime and  which was found by Lapid--Rogawski, see \cite{laprog}, p.\ 176. This mistake, however, concerns Lem.\ 6.3 of \cite{arthur-clozel} on p.\ 217, a result, which is irrelevant for the present paper and hence our reference to \cite{arthur-clozel}, Chp.\ 3, Thm.\ 6.2 is sufficient for our purposes. Moreover, the reader should note that this problem of non-prime degree extensions $L$ has meanwhile been fixed in total generality, see \cite{labesse-bc}, Thm.\ VI.4.1 and by a different method in \cite{henniart}, Thm.\ 3. \label{fn:bc}.

Let $\vartheta$ be a generator of Gal$(L/F)$. It induces an automorphism on $\A_L^{\times}$, denoted by the same letter, and we define $\chi^{\vartheta}$, a Hecke character of $L$, as the composition $\chi \circ \vartheta$. \\\\
Now identify Gal$(L/F)$ with Gal$(L^{+}/F^{+})$. If $n$ is even, we define $L^{\flat}:=L^{\vartheta^{\frac{n}{2}}c}$, an index $2$ subfield of $L$. It is also a CM field with maximal totally real subfield $(L^{+})^{\vartheta^{\frac{n}{2}}}$. We write $\varepsilon_{L/L^{+}}$ and $\varepsilon_{L/L^{\flat}}$ for the quadratic Hecke character associated to $L/L^{+}$ and $L/L^{\flat}$, respectively, by class field theory.

\begin{prop}\label{proposition Asai AI}
Assume that $\chi$ is conjugate self-dual.\\
If $n$ is odd then
\begin{equation} \label{odd Asai}
L^{S}(s,\Pi_{\chi}, {\rm As}^{(-1)^{n}})=\prod\limits_{1\leq k\leq \frac{n-1}{2}} L^{S}(s,\chi\otimes \chi^{\vartheta^{k},c})L^{S}(s,\varepsilon_{L/L^{+}}).
\end{equation}
If $n$ is even then  \begin{equation}  \label{even Asai}
L^{S}(s,\Pi_{\chi}, {\rm As}^{(-1)^{n}})=\prod\limits_{1\leq k\leq \frac{n-2}{2}} L^{S}(s,\chi\otimes \chi^{\vartheta^{k},c})L^{S}(s,\varepsilon_{L/L^{+}}) L^{S}(s,\chi \mid_{\mathbb{A}_{L^{\flat}}}\otimes \varepsilon_{L/L^{\flat}}). \end{equation}
\end{prop}

\begin{proof}
Let $v$ be a non-archimedean place of $F^{+}$ such that every representation at hand is unramified at $v$. We now prove that the products of local $L$-factors at the places over $v$ of both sides of \eqref{odd Asai} and \eqref{even Asai} are equal. In order to ease our assertions, we simply call these products the ``$v$-parts'' of the left hand side, respectively, the right hand side. As a general reference, we refer again to Lemma $1$ of \cite{gelb_jac_rog} where the unramified local factors of the Asai $L$-function have been calculated.\\\\
We write $q_{v}$ for the cardinality of the residue field of $F^{+}_{v}$. We will use similar notations for other finite places of other fields. Let $w$ be a place of $F$ over $v$. Let $w_{1},w_{2},\cdots w_{m}$ be the places of $L$ over $w$. We know $m\mid n$ and we write $l$ for $n/m$. We may assume that $\vartheta(w_{i-1})=w_{i}$ for any $1\leq i\leq m$ where we apply the useful extension of notation defined by $w_{i}:=w_{i \mod m}$ for all $i\in \Z$.\\\\
Let $\zeta$  be a primitive $l$-th root of unity. For each $i$ let $t_{i}$ be the Hecke eigenvalue of $\chi$ at $w_{i}$, where, similar to above, we wrote $t_{i}:=t_{i \mod m}$ for $i\in \Z$. Since $\chi$ is conjugate self-dual, its Hecke eigenvalue at $w_{i}^{c}$ is $t_{i}^{-1}$ where $w_{i}^{c}$ is the complex conjugation of $w_{i}$. If $w_{i}=w_{i}^{c}$ then $t_{i}=\pm 1$. Moreover, since $\chi$ is algebraic and conjugate self-dual, $\chi$ is trivial on $\mathbb{A}_{L^{+}}^{\times}$ (cf.\ Rem.\ \ref{remove Gauss sum}). Hence $t_{i}$ is in fact $1$ in this case. For each $1\leq i\leq m$, we fix a complex $l$-th root of $t_{i}$ and denote it by $t_{i}^{1/l}$. By equation (6.2) in Chapter $3$ of \cite{arthur-clozel}, we know that the Hecke eigenvalues of $\Pi_{\chi}$ at $w$ are $t_{i}^{1/l}\zeta^{a}$ with $ 1\leq i\leq m, 1\leq a\leq l$, and those at $w^{c}$ are $t_{j}^{-1/l}\zeta^{b}$ with $ 1\leq j\leq m, 1\leq b\leq l$.\\\\
{\it Case 1: $n$ odd} In this case, both $l$ and $m$ are odd numbers.

\begin{enumerate}
\item When $v$ is split in $F$, i.e. $v=ww^{c}$, the $v$-part of the left hand side of equation (\ref{odd Asai}) is equal to
\[\prod\limits_{1\leq i,j\leq m}\prod\limits_{1\leq a\leq l}\prod\limits_{1\leq b\leq l}(1-t_{i}^{1/l}\zeta^{a}t_{j}^{-1/l}\zeta^{b}q_{v}^{-s})^{-1} =\prod\limits_{1\leq i,j\leq m}(1-t_{i}t_{j}^{-1}q_{v}^{-ls})^{-l}.\]
The $v$-part of $L(s,\chi\otimes \chi^{\vartheta^{k},c})$ is equal to $\prod\limits_{1\leq i\leq m}(1-t_{i}t_{i+k}^{-1}q_{w_{i}}^{-s})^{-1}(1-t_{i}^{-1}t_{i+k}q_{w_{i}}^{-s})^{-1}$. We know $q_{w_{i}}=q_{w_{i}^{c}}=q_{v}^{l}$. Hence the $v$-part of the right hand side of equation (\ref{odd Asai}) is equal to:
\begin{eqnarray}
&& \left[\prod\limits_{1\leq k\leq \frac{n-1}{2}} \prod\limits_{1\leq i\leq m}(1-t_{i}t_{i+k}^{-1}q_{v}^{-ls})^{-1}(1-t_{i}^{-1}t_{i+k}q_{v}^{-ls})^{-1}\right]\cdot (1-q_{v}^{-ls})^{-m}\nonumber\\
&=&   \prod\limits_{1\leq i\leq m} \left[ \left( \prod\limits_{-\frac{n-1}{2}\leq k\leq \frac{n-1}{2},k\neq 0}(1-t_{i}t_{i+k}^{-1}q_{v}^{-ls})^{-1}\right)\cdot (1-t_{i}t_{i}^{-1}q_{v}^{-ls})^{-1}\right]\nonumber\\
&=&   \prod\limits_{1\leq i\leq m} \prod\limits_{-\frac{n-1}{2}\leq k\leq \frac{n-1}{2}} (1-t_{i}t_{i+k}^{-1}q_{v}^{-ls})^{-1}\nonumber\\
&=&   \prod\limits_{1\leq i\leq m} \prod\limits_{1\leq j\leq m} (1-t_{i}t_{j}^{-1}q_{v}^{-ls})^{-l}\nonumber.
\end{eqnarray}

\item Assume now that $v$ is inert in $F$. Since Gal$(L/F)$ acts transitively on the set $\{w_{i}\mid 1\leq i\leq m\}$ and commutes with complex conjugation, either $w_{i}^{c}= w_{i}$ for all $i$, or $w_{i}^{c} \neq w_{i}$ for all $i$. If the latter is true, then for each $1\leq i\leq m$ there exists $j\neq i$ such that $w_{j}=w_{i}^{c}$. In particular, the set $\{w_{i}\mid 1\leq i\leq m\}$ can be divided into disjoint pairs. But this is impossible since $m$ is odd. Therefore we have that $w_{i}^{c}=w_{i}$ and hence the Hecke eigenvalues $t_{i}=1$ for all $i$.
Keeping this in mind, one can easily show that the $v$-parts of both sides of equation (\ref{odd Asai}) coincide and are in fact equal to $(1+ q_{v}^{-ls})^{-m}(1-q_{v}^{-2ls})^{-(m^{2}l-m)/2}$.

\end{enumerate}
{\it Case 2: $n$ even}
\begin{enumerate}
\item When $v=ww^{c}$ is split, the $v$-part of the left hand side of equation (\ref{even Asai}) is again equal to $\prod\limits_{1\leq i,j\leq m}(1-t_{i}t_{j}^{-1}q_{v}^{-ls})^{-l}.$ In order to evaluate the right hand side, observe that the finite places of $L^{\flat}=L^{\vartheta^{\frac{n}{2}}c}$ over $v$ are $w_{i}w_{\frac{n}{2}+i}^{c}$, $1\leq i\leq m$. Therefor, the Hecke character $\chi\mid_{L^{\flat}}$ has Hecke eigenvalue $t_{i}t_{\frac{n}{2}+i}^{-1}$ at $w_{i}w_{\frac{n}{2}+i}^{c}$. The $v$-part of the right hand side is hence equal to:
\begin{eqnarray}
&& \left[\prod\limits_{1\leq k\leq \frac{n-2}{2}} \prod\limits_{1\leq i\leq m}(1-t_{i}t_{i+k}^{-1}q_{v}^{-ls})^{-1}(1-t_{i}^{-1}t_{i+k}q_{v}^{-ls})^{-1}\right]\cdot (1-q_{v}^{-ls})^{-m}\nonumber\\
&& \hspace{25pt}\cdot \left[ \prod\limits_{1\leq i\leq m}(1-t_{i}t_{\frac{n}{2}+i}^{-1} q^{-ls})^{-1}\right]
\nonumber\\
&=&   \prod\limits_{1\leq i\leq m} \left[ \left( \prod\limits_{-\frac{n-2}{2}\leq k\leq \frac{n}{2},k\neq 0}(1-t_{i}t_{i+k}^{-1}q_{v}^{-ls})^{-1}\right)\cdot (1-t_{i}t_{i}^{-1}q_{v}^{-ls})^{-1}\right]\nonumber\\
&=&   \prod\limits_{1\leq i\leq m} \prod\limits_{-\frac{n-2}{2}\leq k\leq \frac{n}{2}} (1-t_{i}t_{i+k}^{-1}q_{v}^{-ls})^{-1}\nonumber\\
&=&   \prod\limits_{1\leq i\leq n} \prod\limits_{1\leq j\leq m} (1-t_{i}t_{j}^{-1}q_{v}^{-ls})^{-1}\nonumber\\&=&   \prod\limits_{1\leq i\leq m} \prod\limits_{1\leq j\leq m} (1-t_{i}t_{j}^{-1}q_{v}^{-ls})^{-l}\nonumber.
\end{eqnarray}

\item When $v$ is inert and $w_{i}= w_{i}^{c}$ for all $1\leq i\leq m$, we know $t_{i}=1$ for all $i$. So the representation $\Pi_{\chi}=\Pi(\chi)\otimes \eta$ has Hecke eigenvalues $-\zeta^{j}$ for $1\leq j\leq n$.

If $l$ is odd, then $m$ is even. The $v$-part of the left hand side of equation (\ref{even Asai}) is
\[ (1+ q_{v}^{-ls})^{-m}(1-q_{v}^{-2ls})^{-(m^{2}l-m)/2}.\]
In this case, the finite places of $L^{\flat}=L^{\vartheta^{\frac{n}{2}}c}$ over $v$ are $w_{i}w_{\frac{m}{2}+i}$, $1\leq i\leq \frac{m}{2}$. The $v$-part of the right hand side is then equal to:
\begin{eqnarray}\nonumber
&& (1-q_{v}^{-2ls})^{-m(ml-2)/2}(1+q_{v}^{-ls})^{-m}(1-q_{v}^{-2ls})^{-\frac{m}{2}}\\
\nonumber&=&(1+ q_{v}^{-ls})^{-m}(1-q_{v}^{-2ls})^{-(m^{2}l-m)/2}.
\end{eqnarray}
Similarly, if $l$ is even then the $v$-parts of both sides of equation (\ref{even Asai}) are easily seen to be equal to $(1-q_{v}^{-2ls})^{-m^{2}l/2}$.

\item When $v$ is inert and $w_{i}\neq  w_{i}^{c}$ for all $1\leq i\leq m$, there exists $t$, an integer between $2$ and $m-1$, such that $w_{1}^{c}=w_{t+1}$. We apply $\vartheta^{t}$ to both sides and get $w_{t+1}^{c}=w_{2t+1}$. Hence $2t+1\equiv 1 \mod m$. This implies that $2t=m$. In particular, we know then $m$ is even and $w_{i}^{c}=w_{i+\frac{m}{2}}$. The latter implies that $t_{i+\frac{m}{2}}=t_{i}^{-1}$ for all $1\leq i\leq m$.

The representation $\Pi_{\chi}$ has eigenvalues $\{-\zeta^{a}t_{i}^{1/l}\mid 1\leq i\leq m, 1\leq a\leq l\}$ at $v$. the $v$-part of the left hand side of equation (\ref{even Asai}) is:
\begin{equation}
\prod\limits_{1\leq i\leq m}\prod\limits_{1\leq a\leq l}(1+\zeta^{a}t_{i}^{1/l}q_{v}^{-s})^{-1}P(q_{v}^{-2s})^{-1}
\end{equation}
where $P \in \C[X]$ is the unique polynomial such that $P(0)=1$ and
\begin{eqnarray}\nonumber (P(X))^{2}&=&\prod\limits_{1\leq i,j\leq m} \prod\limits_{1\leq a,b\leq m, (i,a)\neq (j,b)} (1-\zeta^{a}t_{i}^{1/l}\zeta^{b}t_{j}^{1/l}X) \\
&=& \cfrac{\prod\limits_{1\leq i,j\leq m}\prod\limits_{1\leq a,b\leq m} (1-\zeta^{a}t_{i}^{1/l}\zeta^{b}t_{j}^{1/l}X) }{\prod\limits_{1\leq i\leq m}\prod\limits_{1\leq a\leq l}(1-\zeta^{2a}t_{i}^{2/l}X)}
\nonumber\\
&=& \cfrac{\prod\limits_{1\leq i,j\leq m} (1-t_{i}t_{j}X^{l})^{l} }{\prod\limits_{1\leq i\leq m}\prod\limits_{1\leq a\leq l}(1-\zeta^{2a}t_{i}^{2/l}X)}.
\nonumber
\end{eqnarray}

If $l$ is odd, then $$ (P(X))^{2}= \prod\limits_{1\leq i\leq m}(1-t_{i}^{2}X^{l})^{l-1}\prod\limits_{1\leq i<j\leq m} (1-t_{i}t_{j}X^{l})^{2l}.$$
Hence, the $v$-part of the left hand side of equation (\ref{even Asai}) is:
\[
\prod\limits_{1\leq i\leq m}(1-t_{i}^{2}q_{v}^{-2ls})^{-(l-1)/2}\cdot \prod\limits_{1\leq i<j \leq m}(1-t_{i}t_{j}q_{v}^{-2ls})^{-l} \cdot \prod\limits_{1\leq i\leq m}(1+t_{i}q_{v}^{-ls})^{-1}.
\]
Moreover, it is easy to see that $(\vartheta^{\frac{n}{2}}c)w_{i}=w_{i+\frac{m}{2}}^{c}=w_{i}$ for all $i$. The intersections of $w_{i}$, $1\leq i\leq m$ with the ring of integers in $L^{\flat}$ are different prime ideals and hence are inert with respect to the extension $L/L^{\flat}$. The $v$-part of the right hand side of
equation (\ref{even Asai}) is then:
\[ \prod\limits_{1\leq k\leq \frac{n-2}{2}}\prod\limits_{1\leq i\leq m}(1-t_{i}t_{i+k}^{-1}q_{v}^{-2sl})^{-1} \cdot (1-q_{v}^{-2sl})^{-m/2} \cdot \prod\limits_{1\leq i\leq m}(1+t_{i}q_{v}^{-ls})^{-1}.
\]
Recall that $t_{i+\frac{m}{2}}=t_{i}^{-1}$ for all $i$. We have:
\begin{eqnarray}
&& \prod\limits_{1\leq k\leq \frac{n-2}{2}}\prod\limits_{1\leq i\leq m}(1-t_{i}t_{i+k}^{-1}q_{v}^{-2sl})^{-1} \cdot (1-q_{v}^{-2sl})^{-m/2} \nonumber\\
&=& \prod\limits_{1\leq k\leq \frac{n-2}{2}}\prod\limits_{1\leq i\leq m}(1-t_{i}t_{i+k+\frac{m}{2}}q_{v}^{-2sl})^{-1} \cdot \prod\limits_{1\leq i\leq m/2}(1-t_{i}t_{\frac{m}{2}+i}q_{v}^{-2sl})^{-1} \nonumber\\
&=& \prod\limits_{1\leq i\leq m}\prod\limits_{\frac{m}{2}+1\leq k\leq \frac{m+n}{2}-1}(1-t_{i}t_{i+k}q_{v}^{-2sl})^{-1} \cdot \prod\limits_{1\leq i\leq m/2}(1-t_{i}t_{\frac{m}{2}+i}q_{v}^{-2sl})^{-1} \nonumber \\
&=& \prod\limits_{1\leq i\leq m}[\prod\limits_{\frac{m}{2}+1\leq k\leq m-1}(1-t_{i}t_{i+k}q_{v}^{-2sl})^{-(l+1)/2} \cdot \prod\limits_{1\leq k\leq \frac{m}{2}-1}(1-t_{i}t_{i+k}q_{v}^{-2sl})^{-(l-1)/2}]\cdot
\nonumber \\
&& \prod\limits_{1\leq i\leq m} (1-t_{i}t_{i+\frac{m}{2}}q_{v}^{-2sl})^{-(l-1)/2}\prod\limits_{1\leq i\leq m} (1-t_{i}^{2}q_{v}^{-2sl})^{-(l-1)/2}  \prod\limits_{1\leq i\leq m/2}(1-t_{i}t_{\frac{m}{2}+i}q_{v}^{-2sl})^{-1}\nonumber \\
&=& \prod\limits_{1\leq i<j\leq m, j-i\neq m/2}(1-t_{i}t_{j}q_{v}^{-2sl})^{-l}\prod\limits_{1\leq i\leq m/2}(1-t_{i}t_{\frac{m}{2}+i}q_{v}^{-2sl})^{-(l-1)}\cdot\nonumber\\
&&\hspace{20pt}\prod\limits_{1\leq i\leq m} (1-t_{i}^{2}q_{v}^{-2sl})^{-(l-1)/2}\prod\limits_{1\leq i\leq m/2}(1-t_{i}t_{\frac{m}{2}+i}q_{v}^{-2sl})^{-1}\nonumber \\
&=& \prod\limits_{1\leq i<j\leq m}(1-t_{i}t_{j}q_{v}^{-2sl})^{-l} \prod\limits_{1\leq i\leq m} (1-t_{i}^{2}q_{v}^{-2sl})^{-(l-1)/2}\nonumber.
\end{eqnarray}
We have deduced that the $v$-parts of the two sides of equation (\ref{even Asai}) coincide if $l$ is odd.\\\\
If $l$ is even, the left hand side of equation (\ref{even Asai}) is equal to
\[\prod\limits_{1\leq i<j\leq m}(1-t_{i}t_{j}q_{v}^{-2sl})^{-l} \prod\limits_{1\leq i\leq m} (1-t_{i}^{2}q_{v}^{-2sl})^{-l/2}.\]
Moreover, we have $(\vartheta^{\frac{n}{2}}c)w_{i}=w_{i}^{c}=w_{i+\frac{m}{2}}$. Hence $w_{i}w_{i+\frac{m}{2}}$ for $1\leq i\leq \frac{m}{2}$ are the places of $L^{\flat}$ over $v$. The corresponding right hand side is equal to
\begin{eqnarray}
&& \prod\limits_{1\leq k\leq \frac{n-2}{2}}\prod\limits_{1\leq i\leq m}(1-t_{i}t_{i+k}^{-1}q_{v}^{-2sl})^{-1} \cdot (1-q_{v}^{-2sl})^{-m/2} \cdot \prod\limits_{1\leq i\leq m/2}(1-t_{i}t_{\frac{m}{2}+i}q_{v}^{-2ls})^{-1}\nonumber \\
&=& \prod\limits_{1\leq i\leq m}[\prod\limits_{\frac{m}{2}+1\leq k\leq m-1}(1-t_{i}t_{i+k}q_{v}^{-2sl})^{-l/2} \cdot \prod\limits_{1\leq k\leq \frac{m}{2}-1}(1-t_{i}t_{i+k}q_{v}^{-2sl})^{-l/2}]\cdot
\nonumber \\
&& \prod\limits_{1\leq i\leq m} (1-t_{i}t_{i+\frac{m}{2}}q_{v}^{-2sl})^{-(l/2-1)}\prod\limits_{1\leq i\leq m} (1-t_{i}^{2}q_{v}^{-2sl})^{-l/2} \cdot  \prod\limits_{1\leq i\leq m/2}(1-t_{i}t_{\frac{m}{2}+i}q_{v}^{-2ls})^{-2}\nonumber \\
&=& \prod\limits_{1\leq i<j\leq m, j-i\neq m/2}(1-t_{i}t_{j}q_{v}^{-2sl})^{-l}\prod\limits_{1\leq i\leq m/2}(1-t_{i}t_{\frac{m}{2}+i}q_{v}^{-2ls})^{-(l-2)-2}\prod\limits_{1\leq i\leq m} (1-t_{i}^{2}q_{v}^{-2sl})^{-l/2} \nonumber \\
&=& \prod\limits_{1\leq i<j\leq m}(1-t_{i}t_{j}q_{v}^{-2sl})^{-l} \prod\limits_{1\leq i\leq m} (1-t_{i}^{2}q_{v}^{-2sl})^{-l/2}\nonumber.
\end{eqnarray}

\end{enumerate}

\end{proof}

\section{Period relations with explicit powers of $(2\pi i)$}\label{sect:2pii}
\subsection{Critical characters and CM-periods}\label{crits}



Let $\chi$ be a Hecke character of $\cm$ with infinity-type $z^{a_{\iota}}\bar{z}^{a_{\bar{\iota}}}$ at $\iota\in \Sigma$. We say that $\chi$ is {\it critical} if it is algebraic and moreover $a_{\iota}\neq a_{\bar{\iota}}$ for all $\iota\in  J_{\cm}$. This is equivalent to the motive associated to $\chi$ having critical points in the sense of Deligne (cf.\ \cite{deligne79}). We remark that $0$ and $1$ are always critical points, if $\chi$ is conjugate selfdual.\\\\
In the critical case, we can define $\Phi_{\chi}$, a subset of $ J_{F}$, as follows: An embedding $\iota\in J_F$ is in $\Phi_{\chi}$ if and only if $a_{\iota} < a_{\bar{\iota}}$. Clearly, $\Phi_{\chi}$ is a CM type of $\cm$. Given any CM-type $\Phi$, we say that $\chi$ is {\it compatible} with $\Phi$ if  $\Phi=\Phi_{\chi}$. \\\\
Let $\chi$ be an algebraic Hecke character of $F$ and let $\Psi\subset  J_{F}$ be any subset such that $\Psi\cap \bar{\Psi}=\emptyset$. Attached to $(\chi,\Psi)$ one may define a CM Shimura-datum as in section 1.1 of \cite{harrisCMperiod}, and a number field $E(\chi,\Psi)$ which contains $\Q(\chi)$ and the reflex field of the CM Shimura datum defined by $\Psi$. Moreover, one may associate a non zero complex number $p_{F}(\chi,\Psi)$ to this datum, which is well defined modulo $E(\chi,\Psi)^{\times}$, called a {\it CM-period}: As CM-periods $p_F(\chi,\Psi)$ will only be a technical ingredient in our arguments, not showing up in the final formulas, we believe that it is justified not to repeat their precise construction here, but refer for the sake if brevity to the appendix of \cite{harrisappendix}. We also write $p(\chi,\Psi)$ instead of $p_{F}(\chi,\Psi)$ if there is no ambiguity concerning the base field $F$. Slightly abusing our notation, we denote
\smallskip

\begin{center}
\framebox{$E(\chi):=\bigcup_{\Psi} E(\chi,\Psi)$.}
\end{center}

\smallskip \noindent It contains $\Q(\chi)\cdot F^{Gal}$ (but may in general be bigger than that).

\begin{rem}
The group Aut$(\C)$ acts on the CM Shimura datum. The CM-periods are defined via certain rational structures of cohomological spaces. We may choose the rational structures equivariantly under the action of Aut$(\C)$, and get a family of the CM-periods $\{p(^{\sigma}\!\chi,\Phi_{^{\sigma}\!\chi})\}_{\sigma\in {\rm Aut}(\C)}$ which only depends on the restriction of $\sigma$ to $E(\chi)$.
\end{rem}

\subsection{Period relations for CM-periods}

\begin{defn}\label{def:iota}
Let $\vartheta$ be an element in Aut$(\cm)$. For $\iota\in J_{\cm}$, we define $\iota^{\vartheta}\in  J_{\cm}$ as $\iota\circ \vartheta$.
\end{defn}
Recall from \S\ref{sect:AsaiInduced} that we may interpret $\vartheta$ also as an automorphism of $\A_F^\times$ and define $\chi^{\vartheta}=\chi\circ\vartheta$ . Applying Def.\ \ref{def:iota}, it has infinity-type $z^{a_{\iota^{\vartheta}}}\bar{z}^{a_{\bar{{\iota}^{\vartheta}}}}$ at $\iota$. In particular, if $\chi$ is algebraic (resp. critical) then so is $\chi^{\vartheta}$. In particular, if $\chi$ is compatible with a CM type $\Psi_{\chi}$ then $\chi^{\vartheta}$ is compatible with the CM type $\Psi_{\chi}^{\vartheta^{-1}}$.

 \begin{prop}\label{proposition CM-periods}
 Let $\chi$ be a critical Hecke character of $\cm$. Let $\Psi$ be a subset of $ J_{\cm}$ such that $\Psi\cap \bar{\Psi}=\emptyset$. Let $\vartheta$ be an element in {\rm Aut}$(F)$. Then we have:
 \[p(\chi,\Psi) \sim_{E(\chi)} p(\chi^{\vartheta},\Psi^{\vartheta^{-1}})\]
 which is equivariant under $\emph{\textrm{Aut}}(\C/F^{Gal})$.
 \end{prop}

\begin{proof}
Let $T_{\cm}:=Res_{\cm/\Q}\mathbb{G}_{m,\cm}$ be a torus. We define a homomorphism $h_{\Psi}:Res_{\C/\R}\mathbb{G}_{m,\C} \rightarrow T_{F,\R}$ such that for each $\iota\in  J_{\cm}$, the Hodge structure induced by $h_{\Psi}$ is of type $(-1,0)$ if $\iota\in \Psi$, of type $(0,-1)$ if $\iota\in \bar\Psi$, and of type $(0,0)$ otherwise. The pair $(T_{\cm},h_{\Psi})$ is then a Shimura datum. The composition with $\vartheta$ induces a morphism of Shimura data $h: (T_{\cm},h_{\Psi}) \rightarrow (T_{\cm},h_{\Psi^{\vartheta^{-1}}})$. Now the expected relation between CM-periods follows as in Lemma 1.6 in \cite{harrisCMperiod}. We also refer to Proposition $1.2$ of \cite{linCRcomplet} for more details.
\end{proof}

We recall some other properties of CM-periods. The proof is similar to the previous proposition and can be found in Proposition $1.1$ of \cite{linCRcomplet}.

 \begin{prop}\label{propCM}
Let $L$ be a CM field containing $\cm$, $\iota\in J_{L}$ and let $\chi$, $\chi'$ be critical Hecke characters of $F$. Let $\Psi$ a subset of $ J_{F}$ such that $\Psi\cap \bar\Psi=\emptyset$ and let $\Psi=\Psi_{1}\sqcup \Psi_{2}$ be a partition of $\Psi$. Then,
 \begin{eqnarray}
 p(\chi\chi',\Psi) &\sim_{E(\chi)E(\chi')}& p(\chi,\Psi) \ p(\chi',\Psi) \nonumber\\
 p(\chi,\Psi)=p(\chi, \Psi_{1}\sqcup \Psi_{2}) &\sim_{E(\chi)}& p(\chi,\Psi_{1}) \ p(\chi, \Psi_{2})    \nonumber\\
 p(\chi,\Psi) &\sim_{E(\chi)}& p(\bar{\chi},\bar{\Psi})   \nonumber\\
 p(\chi \circ N_{\A_L / \A_\cm},\iota) &\sim_{E(\chi)} & p(\chi, \iota|_{\cm})   \nonumber
 \end{eqnarray}
The first three relations are equivariant under the action of $\emph{\textrm{Aut}}(\C/F^{Gal})$, the last one is equivariant under the action of $\emph{\textrm{Aut}}(\C/L^{Gal})$.
\end{prop}

We will also need the following lemma (cf.\ (1.10.9) in \cite{harris97})
\begin{lem}\label{norm}
For any $\iota\in J_{F}$, we have
\begin{equation}
p(\|\cdot\|_{\Acm},\iota)\sim_{\Q} (2\pi i)^{-1}
\end{equation}
which is equivariant under the action of $\emph{\textrm{Aut}}(\C)$.
\end{lem}

\subsection{A result of Blasius}

The special values of an $L$-function for a Hecke character over a CM field can be interpreted in terms of CM-periods. The following theorem was proved by Blasius, presented as in \cite{harrisCMperiod}, Prop.\ 1.8.1 (and the attached erratum \cite{harris97}, p.\ 82).

\begin{thm}\label{Blasius}
Let $\chi$ be a critical Hecke character of $\cm$ and recall $\check{\chi}=\chi^{-1,c}=\bar{\chi}^{\vee}$. For $m$ a critical value of $L(s,\chi)$, we have
$$L^{S}(m,\chi)\sim_{E(\chi)} (2\pi i)^{md}p(\check{\chi},\Phi_{\chi})$$
is equivariant under the action of $\emph{\textrm{Aut}}(\C/F^{Gal})$.
\end{thm}




\subsection{Special $L$-values of automorphically induced representations}\label{special value}

\subsubsection{Cohomological representations $\Pi_\chi$ and $\Pi_{\chi'}$}\label{sect:reps}
Blasius related critical values of Hecke $L$-functions to CM-periods. We now prove two new results of the same form for critical values of Rankin--Selberg and Asai $L$-functions of automorphically induced representations.\\\\
Let $L$ (resp. $L'$) be a cyclic extension over $\cm$ of degree $n$ (resp. $n-1$) which is still a CM field. Let $\chi$ (resp. $\chi'$) be a conjugate self-dual algebraic Hecke character of $L$ (resp. $L'$). We consider $L$ and $L'$ as subfields of $\C$ and denote by $LL'\subset \C$ the compositum of $L$ and $L'$. We write $L^{+}$ (resp. $L'^{+}$) for the maximal totally real subfield of $L$ (resp. $L'$). It is easy to see that $L^{+}L'^{+}$ is an index $2$ subfield of $LL'$. Hence $LL'$ is also a CM field.\\\\
Let $\iota$ be an element inside the CM type $\Sigma$ of $\cm$. We write $\iota_{1},\iota_{2},\cdots, \iota_{n}$ (resp. $\iota'_{1},\iota'_{2},\cdots, \iota'_{n-1}$) for the embeddings of $L$ (resp. $L'$) which extend $\iota$. For each $1\leq i\leq n$ and $1\leq j\leq n-1$, write $\iota_{i,j}$ for the unique embedding of $LL'$ which extends $\iota_{i}$ and $\iota'_{j}$.\\\\
We write the infinity-type of $\chi$ (resp. $\chi'$) at $\iota_{i}$ (resp. $\iota'_{j}$) as $z^{a_{i}}\bar{z}^{-a_{i}}$ (resp. $z^{b_{j}}\bar{z}^{-b_{j}}$) with $a_{i}\in \Z$ (resp. $b_{j}\in \Z )$. By permuting the embeddings we may suppose that the numbers $a_{i}$ (resp. $b_{j}$) are in decreasing order. Next recall the automorphically induced representations $\Pi_{\chi}$ and $\Pi_{\chi'}$, as defined in \S\ref{sect:AsaiInduced}, attached to $\chi$ and $\chi'$. As $\chi$ and $\chi'$ are assumed to be conjugate self-dual, so are $\Pi_{\chi}$ and $\Pi_{\chi'}$.\\\\
Let us assume moreover that the numbers in $\{a_{i}\}_{1\leq i\leq n}$ (resp. in $\{b_{j}\}_{1\leq j\leq n-1})$ are all different, i.e., the infinity-types of $\chi$ and $\chi'$ are {\it regular}. With this extra assumption \cite[Lem.\ 3.14]{clozel} implies that both representations $\Pi_{\chi}$ and $\Pi_{\chi'}$ are in fact cohomological. Hence, recalling the general description of the image of automorphic induction (see again \S\ref{sect:AsaiInduced}, our explanations and the reference \cite{arthur-clozel}, Chp.\ 3, Thm.\ 6.2, therein) $\Pi_{\chi}$ and $\Pi_{\chi'}$ are indeed unitary conjugate self-dual, cohomological isobaric sums, fully induced from different cuspidal automorphic representations (here we remark that the isobaric summands must be different for $\Pi_{\chi}$ and $\Pi_{\chi'}$, because they are cohomological; cf.\ \S\ref{sect:pi}). In other words, $\Pi_{\chi}$ and $\Pi_{\chi'}$ serve as Eisenstein representations as in \S\ref{sect:Eisen}.

\subsubsection{Rationality for the Asai $L$-function of $\Pi_\chi$}

\begin{prop}\label{middle step down}
Let $\chi$ be a conjugate self-dual algebraic Hecke character of $L$ with regular infinity-type. Then,
\begin{equation}
 L^{S}(1,\Pi_{\chi}, {\rm As}^{(-1)^{n}}) \sim_{E(\chi)} (2\pi i)^{n(n+1)d/2}  \prod_{\iota\in\Sigma}\prod\limits_{1\leq i \leq n }
[ p(\check{\chi},\iota_{i})^{i-1} p(\check{\chi},\bar{\iota}_{i} )^{n-i}]
 \end{equation}
equivariant under $\emph{\textrm{Aut}}(\C/L^{Gal})$.
\end{prop}

\begin{proof}
Recall that the left hand side is calculated in Proposition \ref{proposition Asai AI}.

Let $\vartheta$ be a generator of Aut$(L/\cm)$.
For any $1\leq i\leq n$, there exists $1\leq s(i)\leq n$ such that $\vartheta\iota_{i}=\iota_{s(i)}$. Since $\vartheta$ is a generator of Gal$(L/\cm)$, we know $s$ is of order $n$ in the permutation group $\mathfrak{S}_{n}$. For any $1\leq k\leq n-1$, the Hecke character $\chi^{\vartheta^{k}}$ has infinity-type $z^{a_{s^{k}(i)}}\bar{z}^{-a_{s^{k}(i)}}$ at $\iota_{i}$ for any $1\leq i\leq n$. Hence the Hecke character $\chi\otimes \chi^{\vartheta^{k},c}=\chi\otimes \chi^{\vartheta^{k},-1}$ has infinity-type $z^{a_{i}-a_{s^{k}(i)}}\bar{z}^{-a_{i}+a_{s^{k}(i)}}$ at $\iota_{i}$ for any $1\leq i\leq n$. Since $s^{k}(i)\neq i$, $a_{i}-a_{s^{k}(i)}\neq 0$, so we know that the Hecke character $\chi\otimes \chi^{\vartheta^{k},c}$ is critical. For any $1\leq k\leq n-1$, we define $\Psi_{\iota,k}:=\{\iota_{i}\mid 1\leq i\leq n, a_{i}< a_{s^{k}(i)}\}=\{\iota_{i}\mid 1\leq i\leq n, i>s^{k}(i)\}$. We define $\Psi_{\bar{\iota},k}:=\{\bar{\iota_{i}}\mid 1\leq i\leq n, i<s^{k}(i)\}$ and $\Psi_{k}:=\bigcup\limits_{\iota\in\Sigma} \Psi_{\iota,k} \cup  \Psi_{\bar{\iota},k}$ is the CM type of $L$ associated to $\chi\otimes \chi^{\vartheta^{k},-1}$. By Blasius's result, Thm.\ \ref{Blasius}, and Prop.\ \ref{proposition CM-periods} \& \ref{propCM}, we have:
\begin{eqnarray}
L(1,\chi\otimes \chi^{\vartheta^{k},c})&\sim_{E(\chi)} &(2\pi i)^{nd}p(\check{\chi}\otimes \check{\chi}^{\vartheta^{k},c},\Psi_{k})\nonumber \\
&\sim_{E(\chi)} &(2\pi i)^{nd}p(\check{\chi},\Psi_{k}) p( \check{\chi},\overline{\Psi_{k}}^{\vartheta^{k}})\nonumber.
\end{eqnarray}
It is easy to verify that \begin{eqnarray}\nonumber
 \overline{\Psi_{k}}^{\vartheta^{k}}&=& \bigcup\limits_{\iota\in \Sigma}\left[ \{ \bar{\iota}_{s^{k}(i)}\mid 1\leq i\leq n, i>s^{k}(i)\} \cup \{   \iota_{s^{k}(i)} \mid   1\leq i\leq n, i<s^{k}(i)\}\right]\\\nonumber
 &=&\bigcup\limits_{\iota\in \Sigma}\left[ \{ \bar{\iota}_{i}\mid 1\leq i\leq n, i<s^{-k}(i)\} \cup \{   \iota_{i} \mid   1\leq i\leq n, i>s^{-k}(i)\}\right].
 \end{eqnarray}

Hence we deduce that:
\begin{eqnarray} \label{a single k}
 &L(1,\chi\otimes \chi^{\vartheta^{k},c})\sim_{E(\chi)}  & \\
 &(2\pi i)^{nd}\prod\limits_{\iota\in \Sigma}[ \prod\limits_{1\leq i\leq n, i>s^{k}(i)}p(\check{\chi},\iota_{i}) \prod\limits_{1\leq i\leq n, i<s^{k}(i)}p(\check{\chi},\bar{\iota}_{i} )\prod\limits_{1\leq i\leq n, i<s^{-k}(i)}p(\check{\chi},\bar{\iota}_{i})  \prod\limits_{1\leq i\leq n, i>s^{-k}(i)}p(\check{\chi},\iota_{i})].&\nonumber
\end{eqnarray}

We first prove the lemma when $n$ is odd. In this case, we know by Proposition \ref{proposition Asai AI} that
\begin{equation}
 L^{S}(1,\Pi_{\chi}, {\rm As}^{(-1)^{n}})=\prod\limits_{1\leq k\leq \frac{n-1}{2}} L^{S}(1,\chi\otimes \chi^{\vartheta^{k},c})L^{S}(1,\varepsilon_{L/L^{+}}).
\end{equation}

Equation (\ref{a single k}) implies that:
 \begin{eqnarray}
 &&(2\pi i)^{-n(n-1)d/2}\prod\limits_{1\leq k\leq \frac{n-1}{2}}L(1,\chi\otimes \chi^{\vartheta^{k},c}) \nonumber\\ \nonumber
 & \sim_{E(\chi)} &\prod\limits_{1\leq k\leq \frac{n-1}{2}}\prod\limits_{\iota\in \Sigma}[ \prod\limits_{1\leq i\leq n, i>s^{k}(i)}p(\check{\chi},\iota_{i}) \prod\limits_{1\leq i\leq n, i<s^{k}(i)}p(\check{\chi},\bar{\iota}_{i} )\cdot\\
 \nonumber && \hspace{20pt}\prod\limits_{1\leq i\leq n, i<s^{-k}(i)}p(\check{\chi},\bar{\iota}_{i})  \prod\limits_{1\leq i\leq n, i>s^{-k}(i)}p(\check{\chi},\iota_{i})]
 \\\nonumber&\sim_{E(\chi)}& \prod\limits_{\iota\in \Sigma} \prod\limits_{1\leq k\leq n-1}
[ \prod\limits_{1\leq i \leq n, i>s^{k}(i)} p(\check{\chi},\iota_{i}) \prod\limits_{1\leq i \leq n, i<s^{k}(i)}p(\check{\chi},\bar{\iota}_{i} )]
\\\nonumber&\sim_{E(\chi)}&\prod\limits_{\iota\in \Sigma} \prod\limits_{1\leq i \leq n }
[ \prod\limits_{1\leq k\leq n-1, i>s^{k}(i)} p(\check{\chi},\iota_{i}) \prod\limits_{1\leq k\leq n-1,i<s^{k}(i)}p(\check{\chi},\bar{\iota}_{i} )]\\\nonumber
&\sim_{E(\chi)}& \prod\limits_{\iota\in \Sigma} \prod\limits_{1\leq i \leq n }
[ p(\check{\chi},\iota_{i})^{i-1} p(\check{\chi},\bar{\iota}_{i} )^{n-i}].
\end{eqnarray}
Recall that by equation (\ref{quadratic}) we have $L^{S}(1,\varepsilon_{L/L^{+}})\sim_{L^{Gal}} (2\pi i)^{dn}$. We conclude that
 \begin{equation}\label{cuspidal cuspidal down 1}
 L^{S}(1,\Pi_{\chi}, {\rm As}^{(-1)^{n}}) \sim_{E(\chi)} (2\pi i)^{n(n+1)d/2}  \prod_{\iota \in \Sigma}\prod\limits_{1\leq i \leq n }
[ p(\check{\chi},\iota_{i})^{i-1} p(\check{\chi},\bar{\iota}_{i} )^{n-i}].
 \end{equation}
Next, if $n$ is even, again by Proposition \ref{proposition Asai AI}, we have:
\begin{equation}\label{eq:hihi}
L^{S}(1,\Pi_{\chi}, {\rm As}^{(-1)^{n}})=\prod\limits_{1\leq k \leq \frac{n-2}{2}} L^{S}(1,\chi \otimes \chi^{\vartheta^{k},c})L^{S}(1,\varepsilon_{L/L^{+}}) L^{S}(1,\chi \mid_{\mathbb{A}_{L^{\flat}}}\otimes \varepsilon_{L/L^{\flat}}).
\end{equation}
Similar to above, one may deduce from (\ref{a single k}) by a simple calculation that
\begin{eqnarray}
& \prod\limits_{1\leq k\leq \frac{n-2}{2}}L(1,\chi\otimes \chi^{\vartheta^{k},c})  \sim_{E(\chi)} &\nonumber\\
&(2\pi i)^{dn(n-2)/2}\prod\limits_{\iota\in \Sigma} \prod\limits_{1\leq i \leq n}
[ \prod\limits_{1\leq k\leq n-1, k\neq \frac{n}{2}, i>s^{k}(i)} p(\check{\chi},\iota_{i}) \prod\limits_{1\leq k\leq n-1,k\neq \frac{n}{2}, i<s^{k}(i)}p(\check{\chi},\bar{\iota}_{i} )].&\nonumber
 \end{eqnarray}
Recall that $L^{S}(1,\varepsilon_{L/L^{+}})\sim_{L^{Gal}} (2\pi i)^{dn}$. It remains to calculate $L^{S}(1,\chi \mid_{\mathbb{A}_{L^{\flat}}}\otimes \varepsilon_{L/L^{\flat}})$ in \eqref{eq:hihi}. The complex embeddings of the CM field $L^{\flat}:= L^{\vartheta^{\frac{n}{2}}c}$ are $\iota_{i}\mid_{L^{\flat}}$ with $\iota\in \Sigma, 1\leq i \leq n$. We remark that $(\iota_{i}\mid_{L^{\flat}})^{c}=\iota_{s^{n/2}(i)}\mid_{L^{\flat}}$. The Hecke character $\chi\mid_{\mathbb{A}_{L^{\flat}}}$ has infinity-type $z^{a_{i}-a_{s^{n/2}(i)}}\bar{z}^{-a_{i}+a_{s^{n/2}(i)}}$ at $\iota_{i}$, and the Hecke character $\varepsilon_{L/L^{\flat}}$ has trivial infinity-type.

We define $\Psi_{\iota}^{\flat}:=\{\iota_{i}\mid_{L^{\flat}}\mid a_{i}<a_{s^{n/2}(i)},1\leq i \leq n\}=\{\iota_{i}\mid_{L^{\flat}}\mid i>s^{n/2}(i),1\leq i\leq n\}$. Then the Hecke character $\chi \mid_{\mathbb{A}_{L^{\flat}}}\otimes \varepsilon_{L/L^{\flat}}$ is compatible with the CM type $\bigcup\limits_{\iota\in \Sigma} \Psi_{\iota}^{\flat}$.

Using Prop.\ \ref{propCM} we deduce thereof
\begin{eqnarray}
&&L^{S}(1,\chi \mid_{\mathbb{A}_{L^{\flat}}}\otimes \varepsilon_{L/L^{\flat}})\nonumber\\
&\sim_{E(\chi)} & (2\pi i)^{dn/2}  p(\check{\chi} \mid_{\mathbb{A}_{L^{\flat}}}\otimes \check{\eta}_{L/L^{\flat}},\bigcup\limits_{\iota\in \Sigma} \Phi_{\iota}^{b})\nonumber\\
&\sim_{E(\chi)} & (2\pi i)^{dn/2} \prod\limits_{\iota\in \Sigma} \prod\limits_{1\leq i\leq n, i>s^{n/2}(i)}  p(\check{\chi} \mid_{\mathbb{A}_{L^{\flat}}}\otimes \check{\eta}_{L/L^{\flat}},\iota_{i}\mid_{L^{\flat}}) \nonumber \\
&\sim_{E(\chi)} & (2\pi i)^{dn/2} \prod\limits_{\iota\in \Sigma} \prod\limits_{1\leq i\leq n, i>s^{n/2}(i)}  p([\check{\chi} \mid_{\mathbb{A}_{L^{\flat}}}\otimes \check{\eta}_{L/L^{\flat}}]\circ N_{\A_L/\A_{L^{\flat}}},\iota_{i})
\nonumber\\
&\sim_{E(\chi)} & (2\pi i)^{dn/2} \prod\limits_{\iota\in \Sigma}\prod\limits_{1\leq i\leq n, i>s^{n/2}(i)}  p(\check{\chi} \otimes \check{\chi}^{\vartheta^{\frac{n}{2}}c},\iota_{i})
\nonumber\\
&\sim_{E(\chi)} & (2\pi i)^{dn/2} \prod\limits_{\iota\in \Sigma}[\prod\limits_{1\leq i\leq n, i>s^{n/2}(i)}  p(\check{\chi},\iota_{i})\cdot \prod\limits_{1\leq i\leq n, i>s^{n/2}(i)}  p(\check{\chi},\bar{\iota}_{s^{n/2}(i})]\nonumber\\
&\sim_{E(\chi)} & (2\pi i)^{dn/2} \prod\limits_{\iota\in \Sigma}[\prod\limits_{1\leq i\leq n, i>s^{n/2}(i)} p(\check{\chi},\iota_{i})\cdot \prod\limits_{1\leq i\leq n, i<s^{n/2}(i)}    p(\check{\chi},\bar{\iota}_{i})].\nonumber
\end{eqnarray}

We conclude that when $n$ is even we still have the following relation:
\begin{eqnarray}
&&L^{S}(1,\Pi_{\chi}, {\rm As}^{(-1)^{n}})\nonumber\\
&\sim_{E(\chi)} & (2\pi i)^{dn(n+1)/2} \prod\limits_{\iota\in \Sigma} \prod\limits_{1\leq i \leq n }
[ \prod\limits_{i>s^{k}(i),1\leq k\leq n-1} p(\check{\chi},\iota_{i}) \prod\limits_{i<s^{k}(i),1\leq k\leq n-1 }p(\check{\chi},\bar{\iota}_{i} )]\nonumber\\
&\sim_{E(\chi)} & (2\pi i)^{dn(n+1)/2} \prod\limits_{\iota\in \Sigma} \prod\limits_{1\leq i \leq n }
[p(\check{\chi},\iota_{i})^{i-1}p(\check{\chi},\bar{\iota}_{i} )^{n-i}]\nonumber.
\end{eqnarray}

Finally, we remark that all the relations above are equivariant under the action of Aut$(\C/L^{Gal})$.

\end{proof}

The previous lemma and Theorem \ref{thm:gro-har-lap} imply immediately the following period relation for cuspidal automorphically induced representations $\Pi_\chi$:

\begin{cor}\label{corollary period relation}
Let $\chi$ be a conjugate self-dual algebraic Hecke character of $L$ with regular infinity-type. If $\Pi_{\chi}$ is moreover cuspidal, then
$$p(\Pi_{\chi})\sim_{E(\Pi_{\chi})E(\chi)} a(\Pi_{\chi,\infty})^{-1}(2\pi i)^{dn(n+1)/2} \prod\limits_{\iota\in \Sigma} \prod\limits_{1\leq i \leq n }
[p(\check{\chi},\iota_{i})^{i-1}p(\check{\chi},\bar{\iota}_{i} )^{n-i}]$$
which is equivariant under the action of $\emph{\textrm{Aut}}(\C/L^{Gal})$.
\end{cor}

\subsubsection{Rationality for the Rankin--Selberg $L$-function of $\Pi_\chi\times\Pi_{\chi'}$}

We obtain

\begin{prop}\label{middle step up}
Let $\chi$ (resp.\ $\chi'$) be a conjugate self-dual algebraic Hecke character of $L$ (resp.\ $L'$) with regular infinity-types. Assume the $\Pi_\chi$ is cuspidal and that $(\Pi_{\chi},\Pi_{\chi'})$ satisfies the piano-condition, cf.\ Hypothesis \ref{hypo:piano}. Let $\frac{1}{2}+m\in \emph{\textrm{Crit}}(\Pi_{\chi}\times \Pi_{\chi'})$. Then,
  \begin{eqnarray}\nonumber
&L^{S}(\frac{1}{2}+m, \Pi_{\chi}\times \Pi_{\chi'}) \sim_{E(\chi)E(\chi')E(\phi)} & \\
\nonumber & (2\pi i)^{(\frac{1}{2}+m)dn(n-1)}\prod\limits_{\iota\in\Sigma}\left(\prod\limits_{1\leq i\leq n}[p(\check{\chi},\iota_{i})^{i-1}p(\check{\chi},\overline{\iota_{i}})^{n-i}]\prod\limits_{1\leq j\leq n-1}[p(\check{\chi'},\iota'_{j})^{j-1}p(\check{\chi'},\overline{\iota'_{j}})^{n-1-j}]\right)&
 \end{eqnarray}
equivariant under the action of $\emph{\textrm{Aut}}(\C/(LL')^{Gal})$.
\end{prop}

\begin{proof}

We know that:
\begin{equation}\label{cuspidal cuspidal proof 1}
L^{S}(\frac{1}{2}+m, \Pi_{\chi}\times \Pi_{\chi'}) = L^{S}(\frac{1}{2}+m,( \chi\circ N_{\A_{LL'}/\A_L} )( \chi'\circ N_{\A_{LL'}/\A_{L'}})(\eta\circ N_{\A_{LL'}/\A_{\cm}}) )
\end{equation}
$$= L^{S}(m,( \chi\circ N_{\A_{LL'}/\A_L} )( \chi' \circ N_{\A_{LL'}/\A_{L'}})(\phi \circ N_{\A_{LL'}/\A_{\cm}}) )$$
Since $\phi$ has infinity-type $z^1\bar{z}^0$ at each $\iota$, the infinity-type of the Hecke character $\chi^{\#}:=( \chi\circ N_{\A_{LL'}/\A_L} )( \chi'\circ N_{\A_{LL'}/\A_{L'}})(\phi\circ N_{\A_{LL'}/\A_{\cm}})$ at $\iota_{i,j}$ is $z^{a_{i}+b_{j}+1}\overline{z}^{-a_{i}-b_{j}}.$
The piano-condition implies that:
$$ a_{1}>-b_{n-1}-\tfrac{1}{2} >a_{2}>-b_{n-2}-\tfrac{1}{2}>\cdots>-b_{1}-\tfrac{1}{2}>a_{n}.$$
Define $\Phi_{\iota}:=\{\iota_{i,j}\mid a_{i}+b_{j}+ \frac{1}{2}<0\} =\{\iota_{i,j}\mid i+j\geq n+1 \}$ and $\Phi_{\bar{\iota}}:=\{\overline{\iota_{i,j}} \mid i+j\leq n \}$. Then the Hecke character $\chi^{\#}$ is critical with respect to the CM type $\bigcup\limits_{\iota\in \Sigma} \Phi_{\iota}\cup  \Phi_{\bar{\iota}}$. An easy check shows that $\frac{1}{2}+m$ is critical for $\Pi_{\chi}\times \Pi_{\chi'}$ if and only if $m$ is critical for $\chi^{\#}$. By Blasius's result, Thm.\ \ref{Blasius}, one has
$$ L^{S}(m,\chi^{\#}) \sim_{E(\chi^{\#})} (2\pi i)^{mdn(n-1)}p(\check{\chi}^{\#}, \bigcup\limits_{\iota\in \Sigma} \Phi_{\iota}\cup  \Phi_{\bar{\iota}}).$$
Applying Prop.\ \ref{propCM} to the CM-period on the right hand side implies that
\begin{eqnarray}\nonumber
p(\chi^{\#}, \bigcup\limits_{\iota\in \Sigma} \Phi_{\iota}\cup  \Phi_{\bar{\iota}})&\sim_{E(\chi^{\#})} &\prod\limits_{\iota\in\Sigma} p(\check{\chi}^{\#},  \Phi_{\iota}\cup \Phi_{\overline{\iota}})\\&\sim_{E(\chi^{\#})} &\prod\limits_{\iota\in\Sigma} \prod\limits_{i+j\geq n+1}p(\check{\chi}^{\#},\iota_{i,j})\prod\limits_{i+j\leq n}p(\check{\chi}^{\#},\overline{\iota_{i,j}}).\nonumber
\end{eqnarray}
Next observe that
\begin{eqnarray}
&& \prod\limits_{i+j\geq n+1}p(\check{\chi}^{\#},\iota_{i,j})\nonumber\\
\nonumber
&\sim_{E(\chi^{\#})}& \prod\limits_{i+j\geq n+1}[p((\check{\chi}\circ N_{\A_{LL'}/\A_L}),\iota_{i,j})p(\check{\chi'}\circ N_{\A_{LL'}/\A_{L'}},\iota_{i,j})p((\check{\phi}\circ N_{LL'/F}),\iota_{i,j})]\\
\nonumber
&\sim_{E(\chi^{\#})}& \prod\limits_{i+j\geq n+1}[p(\check{\chi},\iota_{i})p(\check{\chi'},\iota'_{j})p(\check{\phi},\iota)]\\
\nonumber
&\sim_{E(\chi^{\#})}& \prod\limits_{1\leq i\leq n}p(\check{\chi},\iota_{i})^{i-1}\cdot\prod\limits_{1\leq j\leq n-1}p(\check{\chi'},\iota'_{j})^{j}\cdot p(\check{\phi},\iota)^{n(n-1)/2}
\end{eqnarray}
Similarly, we have
\[\prod\limits_{i+j\leq n}p(\check{\chi}^{\#},\overline{\iota_{i,j}})
\sim_{E(\chi^{\#})}
 \prod\limits_{1\leq i\leq n}p(\check{\chi},\overline{\iota_{i}})^{n-i}\cdot \prod\limits_{1\leq j\leq n-1}p(\check{\chi'},\overline{\iota'_{j}})^{n-j}\cdot p(\check{\phi},\overline{\iota})^{n(n-1)/2}.\]
Again by Prop.\ \ref{propCM} and Lem.\ \ref{norm}, we know that
\[
p(\check{\phi},\iota)p(\check{\phi},\overline{\iota})\sim_{E(\phi)} p(\check{\phi},\iota)p(\check{\phi}^{c},\iota) \sim_{E(\phi)} p(\|\cdot \|^{-1},\iota) \sim_{E(\phi)} 2\pi i.
\]
We finally deduce that
  \begin{eqnarray}\nonumber
&& (2\pi i)^{-(\frac{1}{2}+m)dn(n-1)}L^{S}(\frac{1}{2}+m, \Pi_{\chi}\times \Pi_{\chi'}) \\
&\sim_{E(\chi)E(\chi')E(\phi)} &\prod\limits_{\iota\in\Sigma}\left(\prod\limits_{1\leq i\leq n}[p(\check{\chi},\iota_{i})^{i-1}p(\check{\chi},\overline{\iota_{i}})^{n-i}]\prod\limits_{1\leq j\leq n-1}[p(\check{\chi'},\iota'_{j})^{j}p(\check{\chi'},\overline{\iota'_{j}})^{n-j}]\right) \nonumber \\
&\sim_{E(\chi)E(\chi')E(\phi)} &\prod\limits_{\iota\in\Sigma}\left(\prod\limits_{1\leq i\leq n}[p(\check{\chi},\iota_{i})^{i-1}p(\check{\chi},\overline{\iota_{i}})^{n-i}]\prod\limits_{1\leq j\leq n-1}[p(\check{\chi'},\iota'_{j})^{j-1}p(\check{\chi'},\overline{\iota'_{j}})^{n-1-j}]\right) \nonumber
 \end{eqnarray}
where the last equality is due to the fact that $\chi'$ is conjugate self-dual, and hence
\[p(\check{\chi'},\iota'_{j})p(\check{\chi'},\bar{\iota'}_{j}) \sim_{E(\chi')} p(\check{\chi'}\otimes \check{\chi'^{c}},\iota'_{j})  \sim_{E(\chi')} p(\triv,\iota'_{j}) \sim_{E(\chi')} 1.\]
It is easy to see that all relations above are in fact equivariant under Aut$(\C/(LL')^{Gal})$.
\end{proof}

\subsection{Explicit determination of the archimedean factors $a(\Pi_\infty)$ and $p(m,\Pi_\infty,\Pi'_\infty)$}\label{archimedean factors}
We recall the archimedean factors $a(\Pi_\infty)$ and $p(m,\Pi_\infty,\Pi'_{\infty})$ from Thm.\ \ref{thm:gro-har-lap} and Thm.\ \ref{thm:main-app}, respectively. Due to a profound theorem of B.\ Sun, we know that both factors are in fact non-zero (for any choice of generators $[\Pi_\infty]$, $[\Pi'_\infty]$ and any critical value $\tfrac12+m$). Here, we will determine them explicitly, revealing them as concrete powers of $(2\pi i)$. To this end, we recall once more our standing assumptions concerning possible choices of generators $[\Pi_\infty]$, namely Conventions \ref{conv:2} and \ref{diag}. \\\\
Our main idea of proof is to replace our original representations $\Pi$ and $\Pi'$ with particularly simple automorphic representations, {\it with isomorphic archimedean components}, hence giving rise to the same archimedean factors $a(\Pi_\infty)$ and $p(m,\Pi_\infty,\Pi'_{\infty})$, cf.\ Rem.\ \ref{rem:same}. In view of the previous sections, these auxiliary automorphic representations shall be constructed by automorphic induction from suitable Hecke characters, on the one hand, and as isobaric sums of Hecke characters, on the other hand: This approach enables us to use all of our calculations of critical $L$-values of Ranking-Selberg- and Asai-$L$-functions from the pervious sections. Here is our theorem

\begin{thm}\label{second proposition}
Let $\Pi$ by a conjugate self-dual cuspidal automorphic representation of $\GL_n(\A_F)$, which is cohomological with respect to $\EE_\mu$. Recall the abstract archimedean factor $a(\Pi_\infty)$ from Thm.\ \ref{thm:gro-har-lap}, which is uniquely determined by $[\Pi_\infty]$. If $\mu$ is sufficiently regular, i.e., if $\mu_{\iota,j}-\mu_{\iota,j+1}\geq 2$ for all $\iota\in\Sigma$ and $1\leq j\leq n-1$, or if not, that Hypotheses \ref{hyp a 0} and \ref{hyp a 1}, to be formulated in the proof below, hold, then
$$a(\Pi_\infty)\sim_{E(\Pi)} (2\pi i)^{dn}$$
which is equivariant under $\emph{\textrm{Aut}}(\C/F^{Gal})$.
\end{thm}

\begin{proof}
As pointed out above, we shall prove this theorem by constructing three auxiliary representations, $\Pi_\chi$, $\Pi_{\chi^\sharp}$ and $\Pi^\flat$, with appropriate archimedean factors.\\\\
{\it Construction of $\Pi_\chi$:} Since we are only concerned with the infinity-type $\{ z^{a_{\iota,i}}\bar{z}^{-a_{\iota,i}}\}_{1\leq i\leq n}$ at $\iota\in\Sigma$ of $\Pi$, we first replace $\Pi$ by a simpler representation with the same infinity-type. We take a CM-field $L$ which is a cyclic extension over $\cm$ of degree $n$. We write $\iota_{1},\cdots,\iota_{n}$ for the elements in $J_{L}$ which extend $\iota$. If $n$ is even, we let $t=\frac{1}{2}$, otherwise we let $t=0$. In any case, $a_{\iota,i}\in \tfrac{n-1}{2}+\Z$, so $a_{\iota,i}-t\in\Z$. By lemma 5.1 of \cite{linCRcomplet}, there exists an algebraic conjugate self-dual Hecke character $\chi$ of $L$, with infinity-type $z^{a_{\iota,i}-t}\bar{z}^{-a_{\iota,i}+t}$ at $\iota_{i}$, such that $(\Pi_{\chi})_{\infty}\cong \Pi_{\infty}$. We recall that if $\chi$ satisfies $\chi^{\theta}\neq \chi$ for any non-trivial $\theta\in$ Gal$(L/\cm)$, then $\Pi_{\chi}$ is cuspidal (cf.\ Chp.\ 3, Lem.\ 6.4 of \cite{arthur-clozel} and its completion in \cite{henniart}, Thm.\ 2). Hence, after twisting by an appropriate finite order Hecke character, we may assume that $\Pi_{\chi}$ is cuspidal.\\\\
{\it Construction of $\Pi_{\chi^\sharp}$:}
For each $\iota$, let $c_{\iota,1}, c_{\iota,2},\cdots,c_{\iota,n+1}\in (\tfrac12-t)+\Z=\tfrac{n}{2}+\Z$ such that
$$c_{\iota,1} > -a_{\iota,n} > c_{\iota,2}>\cdots>-a_{\iota,1} > c_{\iota,n+1}.$$
Recalling that $a_{\iota,i}\in\tfrac{n-1}{2}+\Z$ are all different, such a choice is always possible. We now take another CM field $L^\sharp$ which is a cyclic extension over $\cm$ of degree $n+1$. Let $\chi^\sharp$ be a conjugate self-dual Hecke character of $L^\sharp$ such that $\chi^\sharp_{\iota_i}(z)=z^{c_{\iota,i}-(\tfrac12-t)}\bar{z}^{-c_{\iota,i}+(\tfrac12-t)}$. At the cost of twisting $\chi^\sharp$ by a Hecke character of finite order, our second auxiliary representation $\Pi_{\chi^\sharp}$, automorphically induced from $\chi^\sharp$ to $\GL_{n+1}(\A_F)$, may again be assumed to be cuspidal. By construction, its infinity-type equals $\{z^{c_{\iota,i}}\bar{z}^{-c_{\iota,i}}\}_{1\leq i\leq n+1}$ at $\iota\in \Sigma$, hence the pair $(\Pi_{\chi^\sharp},\Pi_{\chi})$ satisfies the piano-condition, cf.\ Hypothesis \ref{hypo:piano}.\\\\
We may hence apply Theorem \ref{thm:gro-har-lap} and Thm.\ \ref{thm:main-app} to $\Pi_{\chi^\sharp}$ and $\Pi_{\chi}$, and get that for any critical point $\tfrac 12+m\in$Crit$(\Pi_{\chi^\sharp}\times\Pi_\chi)$,
\begin{equation}\label{e1}
\cfrac{L^{S}(\tfrac 12+m,\Pi_{\chi^\sharp} \times \Pi_{\chi})}{L^{S}(1,\Pi_{\chi^\sharp},{\rm As}^{(-1)^{n+1}})L^{S}(1,\Pi_{\chi},{\rm As}^{(-1)^{n}})}\ \sim_{E(\Pi_{\chi^\sharp})E(\Pi_{\chi})} \ \cfrac{p(m,\Pi_{\chi^\sharp,\infty},\Pi_{\chi,\infty})}{a(\Pi_{\chi^\sharp,\infty})a(\Pi_{\chi,\infty})}
\end{equation}
which is equivariant under the action of Aut$(\C/F^{Gal})$. Here we could remove the Gau\ss{} sum $\mathcal G(\omega_{\Pi_{\chi,f}})$ by Remark \ref{remove Gauss sum}. On the other hand, Prop.\ \ref{middle step down} \& \ref{middle step up} imply that the same quotient satisfies the relation
\begin{equation}\label{e2}
\cfrac{L^{S}(\tfrac 12+m,\Pi_{\chi^\sharp} \times \Pi_{\chi})}{L^{S}(1,\Pi_{\chi^\sharp},{\rm As}^{(-1)^{n+1}})L^{S}(1,\Pi_{\chi},{\rm As}^{(-1)^{n}})}  \sim_{E(\chi^\sharp)E(\chi)E(\phi)}(2\pi i)^{(\frac{1}{2}+m)dn(n+1)-\frac{1}{2}d(n+1)(n+2)-\frac{1}{2}dn(n+1)}
\end{equation}
which is equivariant under the action of Aut$(\C/(L^\sharp L)^{Gal})$.\\\\
If $\mu$ is sufficiently regular, i.e., if $\mu_{\iota,j}-\mu_{\iota,j+1}\geq 2$ for all $\iota\in\Sigma$ and $1\leq j\leq n-1$, we may obviously adjust $\chi^\sharp$ such that there exists critical point $\tfrac{1}{2}+m\in$ Crit$(\Pi_{\chi^\sharp} \times \Pi_{\chi})$ with $m\geq 1$. In particular, the critical $L$-value $L^{S}(\tfrac 12+m,\Pi_{\chi^\sharp} \times \Pi_{\chi})$ is non-zero. As a consequence, we can use the transitivity of ``$\sim$'', cf.\ Rem.\ \ref{transitive}, if $\mu$ is sufficiently regular, and compare \eqref{e1} with \eqref{e2}.\\


If, at the contrary, $\mu$ fails to be sufficiently regular, we may always still take $m=0$, cf.\ \S \ref{sect:critpts}. In order to be able to use the transitivity of the relation ``$\sim$'' (so to compare \eqref{e1} with \eqref{e2}) we have then to assume the validity of

\begin{hyp}\label{hyp a 0}
There exists $\chi^\sharp$ and $\chi$ as above such that $L^{S}(\tfrac 12,\Pi_{\chi^\sharp} \times \Pi_{\chi})\neq 0$. Equivalently, $L^{S}(\tfrac 12,( \chi^{\sharp} \circ N_{\A_{L^{\sharp}L}/\A_{L^{\sharp}}})( \chi\circ N_{\A_{L^{\sharp}L}/\A_L} )(\eta \circ N_{\A_{L^{\sharp} L}/\A_\cm}) ) \neq 0.$
\end{hyp}

As a consequence of Cor.\ \ref{cor:nonvan}, the hypothesis actually implies that $L^{S}(\tfrac 12,{}^{\sigma}\Pi_{\chi^\sharp} \times {}^{\sigma}\Pi_{\chi})\neq 0$ for all $\sigma\in$ Aut$(\C)$. (Rephrased in terms of Hecke $L$-functions $L^{S}(\tfrac 12,( \chi^{\sharp} \circ N_{\A_{L^{\sharp}L}/\A_{L^{\sharp}}})( \chi\circ N_{\A_{L^{\sharp}L}/\A_L} )(\eta \circ N_{\A_{L^{\sharp} L}/\A_\cm}) )$ this is also a corollary of \cite{blasius6}, Thm.\ 9.3.1.)\\

Merging \eqref{e1} with \eqref{e2} and applying our Minimizing-Lemma, cf.\ Lem.\ \ref{minimize}, we finally conclude that
\begin{equation} \label{eq:quot}
\cfrac{p(m,\Pi_{\chi^\sharp,\infty},\Pi_{\chi,\infty})}{a(\Pi_{\chi^\sharp,\infty})a(\Pi_{\chi,\infty})}  \sim_{E(\Pi_{\chi^\sharp})E(\Pi_{\chi})}(2\pi i)^{(\frac{1}{2}+m)dn(n+1)-\frac{1}{2}d(n+1)(n+2)-\frac{1}{2}dn(n+1)}
 \end{equation}
which is equivariant under the action of Aut$(\C/(L^\sharp L)^{Gal})$. \\\\
{\it Construction of $\Pi^\flat$}: We now construct another auxiliary representation of $\GL_{n}(\Acm)$. For each $1\leq j\leq n$, let $\chi_{j}$ be a conjugate self-dual Hecke character of $\cm$ with infinity-type $z^{a_{\iota,j}-t}\bar{z}^{-a_{\iota,j}+t}$ at $\iota\in \Sigma$. We define
$$\Pi^\flat:=\left\{\begin{array}{ll}
 \chi_1\boxplus...\boxplus\chi_n & \textrm{if $n+1$ is even} \\
 (\chi_1\eta)\boxplus...\boxplus(\chi_n\eta) & \textrm{if $n+1$ is odd}
\end{array}
\right.$$
The resulting automorphic $\GL_{n}(\Acm)$-representation $\Pi^\flat$ is unitary, cohomological and conjugate self-dual, hence comes under the purview of \S\ref{sect:Eisen}. Moreover, $\Pi^\flat_{\infty}\cong\Pi_{\chi,\infty}$ and so the pair $(\Pi_{\chi^\sharp},\Pi^\flat)$ satisfies the piano-condition by construction.\\\\
Let $m$ be specified as in our second construction-step above. Since Crit$(\Pi_{\chi^\sharp} \times \Pi^\flat)=$Crit$(\Pi_{\chi^\sharp} \times \Pi_{\chi})$, Thm.\ \ref{thm:main-app} and Rem.\ \ref{remove Gauss sum}, imply that
\begin{equation}\label{a step 0}
L^{S}(\tfrac 12+m,\Pi_{\chi^\sharp} \times \Pi^\flat) \sim_{\Q(\Pi_{\chi^\sharp})\Q(\Pi^\flat)} p(\Pi_{\chi^\sharp})p(\Pi^\flat)p(m,\Pi_{\chi^\sharp,\infty},\Pi^\flat_{\infty})
\end{equation}
On the other hand, we know that $$L^{S}(\tfrac 12+m,\Pi_{\chi^\sharp} \times \Pi^\flat) =\prod\limits_{1\leq j\leq n}L^{S}(\tfrac 12+m,\chi^\sharp \otimes (\chi_{j} \eta\circ N_{\mathbb{A}_{L^\sharp}/\Acm}))=\prod\limits_{1\leq j\leq n}L^{S}(m,\chi^\sharp \otimes (\chi_{j} \phi\circ N_{\mathbb{A}_{L^\sharp}/\Acm})).$$
The Hecke character $\chi^\sharp \otimes (\chi_{j} \phi\circ N_{\mathbb{A}_{L^\sharp}/\Acm})$ has infinity-type $z^{a_{\iota,j}+c_{\iota,i}+\tfrac12}\bar{z}^{-a_{\iota,j}-c_{\iota,i}+\tfrac12}$ at $\iota_{i}$. Hence it is compatible, cf.\ \S\ref{crits}, with the CM type $$\bigcup\limits_{\iota\in \Sigma} \{\iota_{i}\mid i\geq  n+2-j\} \cup \{\bar{\iota_{i}}\mid i\leq n+1-j\}.$$
By Blasius's result, Thm.\ \ref{Blasius}, we have:
\begin{eqnarray}
&&L^{S}(m,\chi^\sharp \otimes (\chi_{j} \phi\circ N_{\mathbb{A}_{L^\sharp}/\Acm})) \nonumber\\
&\sim_{E(\chi^\sharp)E(\chi_{j})E(\phi)}& (2\pi i)^{md(n+1)}\prod\limits_{\iota\in\Sigma}\prod\limits_{i\geq n+2-j}p(\check{\chi}^\sharp,\iota_{i})p(\check{\chi}_{j},\iota)p(\check{\phi},\iota)\prod\limits_{i\leq n+1-j}p(\check{\chi}^\sharp,\bar{\iota_{i}})p(\check{\chi}_{j},\bar{\iota})p(\check{\phi},\bar{\iota})\nonumber
\end{eqnarray}
Denote $\prod\limits_{1\leq j\leq n} E(\chi_{j})$ simply by $E'$. Then
\begin{eqnarray}\label{aaa}
&& L^{S}(\tfrac 12+m,\Pi_{\chi^\sharp} \times \Pi^\flat)\\
&=&\prod\limits_{1\leq j\leq n}L^{S}(m,\chi^\sharp \otimes (\chi_{j} \phi\circ N_{\mathbb{A}_{L^\sharp}/\Acm})) \label{a step 1}\nonumber \\
&\sim_{E(\chi^\sharp)E' E(\phi)}&(2\pi i)^{mdn(n+1)} \prod\limits_{\iota\in\Sigma}[(\prod\limits_{1\leq i\leq n+1}p(\check{\chi}^\sharp,\iota_{i})^{i-1}p(\check{\chi}^\sharp,\iota_{i})^{n+1-i}) (\prod\limits_{1\leq j\leq n}p(\check{\chi}_{j},\iota)^{j-1} p(\check{\chi}_{j},\bar{\iota})^{n+1-j})\times\nonumber\\
&& (p(\check{\phi},\iota)^{n(n+1)/2}p(\check{\phi},\bar{\iota})^{n(n+1)/2})]\nonumber\\&\sim_{E(\chi^\sharp)E' E(\phi)}& (2\pi i)^{(\frac{1}{2}+m)dn(n+1)}\prod\limits_{\iota\in\Sigma}[(\prod\limits_{1\leq i\leq n+1}p(\check{\chi}^\sharp,\iota_{i})^{i-1}p(\check{\chi}^\sharp,\iota_{i})^{n+1-i}) (\prod\limits_{1\leq j\leq n}p(\check{\chi}_{j},\iota)^{j-1} p(\check{\chi}_{j},\bar{\iota})^{n+1-j})]\nonumber
\end{eqnarray}
where the last equation is due to the fact that $p(\check{\phi},\iota)p(\check{\phi},\bar{\iota})\sim_{E(\phi)} 2\pi i$.\\\\
On the other hand, by Corollary \ref{corollary period relation} we know that
\begin{equation}\label{a step 2}
p(\Pi_{\chi^\sharp})\sim_{E(\Pi_{\chi^\sharp})E(\chi^\sharp)} a(\Pi_{\chi^\sharp,\infty})^{-1}(2\pi i)^{d(n+1)(n+2)/2} \prod\limits_{\iota\in \Sigma} \prod\limits_{1\leq i \leq n+1 }
[p(\check{\chi}^\sharp,\iota_{i})^{i-1}p(\check{\chi}^\sharp,\bar{\iota}_{i} )^{n+1-i}].\end{equation}
Moreover, by Corollary \ref{isobaric sum Whittaker corollary}, we know
$p(\Pi^\flat)\sim_{E(\Pi^\flat)E(\phi)} \prod\limits_{1\leq j<k\leq n}L(1,\chi_{j}\otimes \chi_{k}^{\vee})$: We point out that here our assumptions Conv.\ \ref{conv:2} and \ref{diag} enter the proof. \\\\
The Hecke character $\chi_{j}\otimes \chi_{k}^{\vee}=\chi_{j}\otimes \chi_{k}^{c}$ has infinity-type $z^{a_{\iota,j}-a_{\iota,k}}\bar{z}^{-a_{\iota,j}+a_{\iota,k}}$. Since $j<k$, we know $a_{\iota,j}-a_{\iota,k}>0$ and the character $\chi_{j}\otimes \chi_{k}^{c}$ is compatible with $\bar{\Sigma}$. Therefore,
\begin{eqnarray*}
\prod\limits_{1\leq j<k\leq n}L(1,\chi_{j}\otimes \chi_{k}^{\vee}) &\sim_{E'}& (2\pi i)^{dn(n-1)/2} \prod\limits_{1\leq j<k\leq n} \prod\limits_{\iota\in\Sigma} p(\check{\chi_{j}}\check{\chi_{k}}^{c},\bar{\iota})\nonumber \\
&\sim_{E'}&  (2\pi i)^{dn(n-1)/2}\prod\limits_{1\leq j<k\leq n} \prod\limits_{\iota\in\Sigma} [p(\check{\chi_{j}},\bar{\iota})p(\check{\chi_{k}},\iota) ]\nonumber \\
&\sim_{E'}&  (2\pi i)^{dn(n-1)/2} \prod\limits_{\iota\in\Sigma} \prod\limits_{1\leq j\leq n}[p(\check{\chi_{j}},\iota)^{j-1} p(\check{\chi_{j}},\bar{\iota})^{n+1-j}]
\end{eqnarray*}
and hence
\begin{equation}\label{a step 3}
p(\Pi^\flat)\sim_{E(\Pi^\flat)E' E(\phi)} (2\pi i)^{dn(n-1)/2} \prod\limits_{\iota\in\Sigma} \prod\limits_{1\leq j\leq n}[p(\check{\chi_{j}},\iota)^{j-1} p(\check{\chi_{j}},\bar{\iota})^{n+1-j}].
\end{equation}

As above, we assume that either $\mu$ is sufficiently regular or, if not, the validity of the following hypothesis
\begin{hyp}\label{hyp a 1}
There exists $\chi^\sharp$ and $\chi_{j}$, $1\leq j\leq n$ as above, such that $L^{S}(\tfrac 12,\Pi_{\chi^\sharp} \times \Pi^\flat)\neq 0$. Equivalently,
$L(\tfrac 12, (\chi^{\sharp}\circ N_{\A_L/\A_F})\chi_{j}\eta)\neq 0$ for all $j$.
\end{hyp}

\noindent As before we notice that this hypothesis, as combined with Cor.\ \ref{cor:nonvan}, implies the non-vanishing of $L^{S}(\tfrac 12,{}^\sigma\Pi_{\chi^\sharp} \times {}^\sigma\Pi^\flat)$ for all $\sigma\in$ Aut$(\C)$. We may now use again the transitivity of the relation ``$\sim$'', see Rem.\ \ref{transitive}. Hence, comparing (\ref{a step 0}), (\ref{aaa}), (\ref{a step 2}) and (\ref{a step 3}), and invoking our Minimizing-Lemma, Lem.\ \ref{minimize}, we deduce that:
\begin{equation}\label{aaab}
\cfrac{p(m,\Pi_{\chi^\sharp,\infty},\Pi^\flat_{\infty}) }{a(\Pi_{\chi^\sharp,\infty})}\sim_{E(\Pi_{\chi^\sharp})E(\Pi^\flat)} (2\pi i)^{(\frac{1}{2}+m)dn(n+1)-\frac{1}{2}d(n+1)(n+2)-\frac{1}{2}dn(n-1)}.
\end{equation}
{\it Conclusion}: Comparing (\ref{eq:quot}) with \eqref{aaab}, invoking Rem.\  \ref{rem:same} and applying our Minimizing-Lemma once more, we conclude that $a(\Pi_{\chi,\infty})\sim_{E(\Pi_{\chi})} (2\pi i)^{dn}$. As $\Pi_{\chi,\infty}\cong \Pi_\infty$, Rem.\  \ref{rem:same} and applying our Minimizing-Lemma finally imply the desired relation $a(\Pi_{\infty})\sim_{E(\Pi)} (2\pi i)^{dn}$ for our given cuspidal representation $\Pi$.\\\\
For the last assertion, observe that the relation for $a(\Pi_{\infty})$ is independent of the choice of field extensions $L^\sharp$ and $L$. Hence, it is in fact equivariant under the union of all groups Aut$(\C/(L^\sharp L)^{Gal})$, taken over all $L^\sharp$ and $L$, which are cyclic CM-extensions of $F$ of prescribed degree. By class field theory, this union is Aut$(\C/F^{Gal})$.
\end{proof}

\begin{rem}\label{rem:3reg}
Instead of the regularity-condition on $\mu$, we could have equivalently assumed that there is a conjugate self-dual cuspidal automorphic representation $\Pi^\sharp$ of $\GL_{n+1}(\A_F)$, satisfying the piano-hypothesis when coupled with $\Pi$, and $\tfrac12+m\in$Crit$(\Pi^\sharp\times\Pi)$ with $m\neq 0$. This assumption, however, just reads far more elaborate than the simple obstruction on the highest weight $\mu$.
\end{rem}

Let now $\Pi$ and $\Pi'$ be automorphic representations as in Thm.\ \ref{thm:main-app} and assume that their archimedean components $\Pi_\infty$ and $\Pi'_\infty$ are conjugate self-dual. Choose conjugate self-dual Hecke characters $\chi$ of $L$ and $\chi'$ of $L'$ as in \S\ref{sect:reps} such that $\Pi_{\chi,\infty}\cong \Pi_{\infty}$ and $\Pi_{\chi',\infty}\cong \Pi'_{\infty}$. By re-adjusting the characters, if necessary, we may impose that $\Pi_\chi$ and likewise $\Pi_{\chi'}$ is cuspidal. For reference in the theorem below we record our last

\begin{hyp}\label{cond}
There are characters $\chi$ and $\chi'$ as above such that $L^S(\tfrac12,\Pi_\chi\times \Pi_{\chi'})\neq 0$. Equivalently, $L^{S}(\frac{1}{2},( \chi\circ N_{\A_{LL'}/\A_L} )( \chi'\circ N_{\A_{LL'}/\A_{L'}})(\eta\circ N_{\A_{LL'}/\A_{\cm}}) )\neq 0.$
\end{hyp}

We obtain

\begin{cor}\label{prop:generalized grob-harr}
Let $\Pi$ and $\Pi'$ be cohomological automorphic representations as in Thm.\ \ref{thm:main-app} and assume that their archimedean components $\Pi_\infty$ and $\Pi'_\infty$ are conjugate self-dual. Let $\frac{1}{2}+m\in \emph{\textrm{Crit}}(\Pi\times\Pi')$ be a critical point. If $m=0$ we assume Hyp. \ref{hyp a 0} \& \ref{hyp a 1} for $\Pi$ and $\Pi'$, whenever $\mu$ or $\mu'$ are not sufficiently regular, and moreover Hyp. \ref{cond}. Then,
$$p(m,\Pi_{\infty},\Pi'_{\infty})\sim_{E(\Pi)E(\Pi')} (2\pi i)^{mdn(n-1)-\frac{1}{2}d(n-1)(n-2)}
$$
is equivariant under the action of $\emph{\textrm{Aut}}(\C/F^{Gal})$.
\end{cor}
\begin{proof}
Choose some algebraic conjugate self-dual Hecke characters $\chi$ and $\chi'$ as in \S\ref{sect:reps}, such that $\Pi_\chi$ and $\Pi_{\chi'}$ are cuspidal with $\Pi_{\chi,\infty}\cong \Pi_{\infty}$ and $\Pi_{\chi',\infty}\cong \Pi'_{\infty}$. As pointed out above, this is always possible. Let $\frac{1}{2}+m\in$Crit$(\Pi\times\Pi')=$Crit$(\Pi_\chi\times\Pi_{\chi'})$. If $m\neq 0$, then our description of the set of critical points for Eisenstein representations, which satisfy the piano-condition, given in \S\ref{sect:critpts}, tells us that the highest weights $\mu$ and $\mu'$ of the finite-dimensional coefficient modules $\EE_\mu$ and $\EE_{\mu'}$, with respect to which $\Pi_\infty$ and $\Pi'_\infty$ are cohomological, are sufficiently regular. Hence, Thm.\ \ref{second proposition} holds for $\Pi_\chi$ and $\Pi_{\chi'}$, i.e., we have $a(\Pi_{\chi,\infty})\sim_{E(\Pi_\chi)} (2\pi i)^{dn}$ and $a(\Pi_{\chi',\infty})\sim_{E(\Pi_{\chi'})} (2\pi i)^{d(n-1)}$. Moreover, as $\tfrac12+m\neq \tfrac12$, the critical $L$-value $L^{S}(\tfrac 12+m,{}^{\sigma}\Pi_{\chi} \times {}^{\sigma}\Pi_{\chi'})$ is non-zero for all $\sigma\in$ Aut$(\C)$. Hence, \eqref{eq:quot} is valid, which yields
$$p(m,\Pi_{\chi,\infty},\Pi_{\chi',\infty}) \sim_{E(\Pi_{\chi})E(\Pi_{\chi'})}(2\pi i)^{mdn(n-1)-\frac{1}{2}d(n-1)(n-2)}.$$
Invoking our Minimizing-Lemma, cf.\ Lem.\ \ref{minimize}, shows the claim for $m\neq 0$. If $m=0$, then our additional assumptions imply that one may in fact argue as for $m\neq 0$. This completes the proof.
\end{proof}

\subsubsection{Some remarks on our hypotheses}\label{sect:nonvan}
The reader will have observed that we invoke our non-vanishing hypotheses Hyp.\ \ref{hyp a 0} and \ref{hyp a 1} in the proof of Thm.\ \ref{second proposition} only, if we have no critical point $s=\tfrac12 + m$, with $m\neq 0$ at our disposal. Similarly, we only invoked Hyp.\ \ref{cond} in the proof of Cor.\ \ref{prop:generalized grob-harr} only, when we wanted to determine $p(m,\Pi_{\infty},\Pi'_{\infty})$ at the center of symmetry, i.e., at $m=0$.\\\\
Indeed, whenever we (may) consider a non-central critical point $s=\tfrac12 + m$, $m\neq 0$, our results, Thm.\ \ref{second proposition} and its corollary, Cor.\ \ref{prop:generalized grob-harr}, {\it hold unconditionally}. Translated into a condition on the highest weight under consideration -- for simplicity let us denote it here uniformly by $\mu=(\mu_\iota)_{\iota\in\Sigma}$ -- the existence of a critical point $s=\tfrac12 + m$ with $m\neq 0$ is equivalent to $\mu_{\iota,j}-\mu_{\iota,j+1}\geq 2$ for all $\iota\in\Sigma$ and $1\leq j\leq n-1$, see \S \ref{sect:critpts}. In other words, if the highest weight at hand is sufficiently regular, then we do not need to assume any further hypotheses in Thm.\ \ref{second proposition} nor do we have to do so, when we determine $p(m,\Pi_{\infty},\Pi'_{\infty})$ at $m\neq 0$ in Cor.\ \ref{prop:generalized grob-harr}.\\\\
In the case when our hypotheses are in force, (i.e., if there is either no non-central critical point, or, if we want to compute $p(m,\Pi_{\infty},\Pi'_{\infty})$ at $m=0$) we have already reformulated the expected non-vanishing of the Rankin-Selberg $L$-function at hand in terms of the non-vanishing of a Hecke $L$-function, see Hyp.\ \ref{hyp a 0}, \ref{hyp a 1} and \ref{cond}. The question of non-vanishing of Hecke $L$-series is largely addressed by the literature, which in fact provides strong evidence that our hypotheses hold true in the required generality: As a first clue for this, we recall a result of Rohrlich. Let $\pi$ be a unitary cuspidal automorphic representation of $\GL_{2}(\Acm)$ or $\GL_{1}(\Acm)$ (i.e., a unitary Hecke character of $F$ in the latter case). Then the main result in \cite{rohrlich} shows the non-vanishing of $L^S(\tfrac{1}{2}, \pi\otimes \chi)$ for infinitely many twists by Hecke characters $\chi$ of $F$. \\\\
Unfortunately, Rohrlich's result is not sufficient for us to actually prove the validity of our hypotheses, since it does not provide the additionally necessary, qualitative information on the twists $\chi$, required by our proof, like being conjugate self-dual, for instance.\\\\
However, Rohrlich's result fits into an increasingly general series of explicit results and conjectures about the non-vanishing of standard $L$-functions -- over CM-fields we refer to \cite{eischen}, Thm.\ 3.15 and Thm.\ 3.16, \cite{ginz-jiang-rallis}, Thm.\ 5.1 and Thm.\ 6.3, and \cite{jiang-zhang}, Conj.\ 1.1. From these references we extract the following, very general conjecture

\begin{conj}\label{general conjecture}
Let $H_n=U(V_n)$ be a unitary group over $\tr$ of rank $n$, as in \S \ref{sect:alggrp}. Let $\pi$ be a cuspidal automorphic representation $H_n(\Atr)$. Then there exists an automorphic character $\alpha: H_1(\A_{F^+})\ra\C^*$ of finite order, such that $L^S(\tfrac 12, \pi\otimes\alpha)\neq 0$.
\end{conj}
If we do not require $\alpha$ to be of finite order, then Conj.\ \ref{general conjecture} is known to hold for quasi-split unitary groups $H_n$ of rank $n\leq 4$ and generic, tempered representations $\pi$. See \cite{jiang-zhang}, Thm.\ 1.4. \\\\
Returning to our hypotheses Hyp.\ \ref{hyp a 0}, \ref{hyp a 1} and \ref{cond}, we observe that Conj.\ \ref{general conjecture} provides far more than what we need:

\begin{lem}
Conj.\ \ref{general conjecture} for $n=1$ implies Hyp. $\ref{hyp a 1}$. Conj.\ \ref{general conjecture} for general $n$ implies Hyp. \ref{hyp a 0} and Hyp. \ref{cond}.
\end{lem}

\begin{proof}
Clearly Conjecture \ref{general conjecture} for $n=1$ implies Hyp.\ \ref{hyp a 1} since one can twist each $\chi_{i}$ in Hyp.\ \ref{hyp a 1} by any conjugate self-dual character finite order.\\
Hyp.\ \ref{hyp a 0} and Hyp.\ \ref{cond} are of the same form. We provide the argument for Hyp. \ref{cond}: For simplicity, we may also assume that $n$ is even. We first take any conjugate self-dual Hecke characters $\chi_{0}$ of $L$ and $\chi'$ of $L'$ as in \S\ref{sect:reps} such that $\Pi_{\chi_{0},\infty}\cong \Pi_{\infty}$ and $\Pi_{\chi',\infty}\cong \Pi'_{\infty}$. Note that $LL'$ is a cyclic extension of $L$ of degree $n-1$. Let $\Pi'_{L}$ be the automorphic induction of $\chi'\circ N_{\A_{LL'}/\A_{L'}}$ to $L$. It is an automorphic representation of $\GL_{n-1}(\mathbb{A}_{L})$. We can asssume that $\Pi'_{L}$ is cuspidal by twisting $\chi'$ by a suitable conjugate self-dual Hecke character finite order. Moreover, $\Pi'_L$ being conjugate self-dual and cohomological ensures that it is in the image of quadratic base change from some unitary group $H_n/F^+$, see \cite{grob_harris_lapid}, \S 6.1. By Conj.\ \ref{general conjecture}, there hence exists a conjugate self-dual Hecke character $\beta$ of $F$ of finite order, such that $L^S(\tfrac{1}{2}, (\Pi'_{L}\otimes \chi_{0} (\eta\circ N_{\A_L/\A_F}))\otimes \beta)\neq 0$. (Explcitly, $\beta$ is simply defined by sending $z\in\A^\times_F$ to $\alpha(z (z^c)^{-1})$ for $\alpha$ as being provided by Conj.\ \ref{general conjecture}.) We note that $L^S(\tfrac{1}{2}, (\Pi'_{L}\otimes \chi_{0} (\eta\circ N_{\A_L/\A_F}))\otimes \beta)=L^{S}(\frac{1}{2}, (\chi_{0} \beta \circ N_{\A_{LL'}/\A_L} )( \chi'\circ N_{\A_{LL'}/\A_{L'}})(\phi \circ N_{\A_{LL'}/\A_F}))$. Therefore, letting $\chi:=\chi_{0}\beta$, we obtain the assertion of Hyp.\ \ref{cond}.
\end{proof}


\section{Our four main theorems for special $L$-values}\label{sect:main}

\subsection{Critical values of Asai $L$-functions}\label{sect:main_ghl}

Our first main theorem for special values has two assets: Firstly, it generalizes Thm.\ \ref{thm:gro-har-lap} to isobaric representations $\Pi'=\Pi_1\boxplus...\boxplus\Pi_k$, with an arbitrary number of conjugate self-dual cuspidal summands $\Pi_i$. Secondly, we are able to explicitly determine the archimedean factor in the resulting relation due to our calculations in \S\ref{special value} as a concrete power of $(2\pi i)$. In what follows, we write $\mu_{i}^{\sf alg}$ for the highest weight of the algebraic representation with respect to which $\Pi_i^{\sf alg}$, cf.\ \ref{eq:pialg}, is cohomological. We recall again that our two conventions, Conventions\ \ref{conv:2} and \ref{diag}, are in force.

\begin{thm}\label{generalized thm:gro-har-lap}
Let $\Pi'=\Pi_1\boxplus...\boxplus\Pi_k$ be a cohomological isobaric automorphic representation of $\GL_{n}(\Acm)$ as in \ref{sect:Eisen}, such that each cuspidal automorphic summand $\Pi_i$ is conjugate self-dual. If $\mu_{i}^{\sf alg}$ is not sufficiently regular, we assume Hyp.\ \ref{hyp a 0} \& \ref{hyp a 1} for $\Pi_i^{\sf alg}$. One has
$$
L^{S}(1,\Pi',{\rm As}^{(-1)^{n}})\ \sim_{E(\Pi')} \ (2\pi i)^{dn} p(\Pi')
$$
which is equivariant under the action of $\emph{\textrm{Aut}}(\C/F^{Gal})$.

\end{thm}



\begin{proof}
On the one hand, by Lem.\ \ref{lem:fact_asai_lfunction} we know that
$$L^S(1,\Pi',{\rm As}^{(-1)^{n}})=\prod_{i=1}^k L^S(1,\Pi_i^{\sf alg},{\rm As}^{(-1)^{n_i}})\ \cdot \prod_{1\leq i<j\leq k} L^S(1,\Pi_i\times\Pi^\vee_j).$$
Since $\Pi_{i}^{\sf alg}$ is unitary conjugate self-dual, cuspidal and cohomological, we may apply Thm.\ \ref{thm:gro-har-lap} and by our extra assumptions on $\Pi_i^{\sf alg}$ moreover Thm.\ \ref{second proposition} to get
 $$L^S(1,\Pi_i^{\sf alg},{\rm As}^{(-1)^{n_i}}) \sim_{E(\Pi_i^{\sf alg})} a(\Pi^{\sf alg}_{i,\infty}) p(\Pi_i^{\sf alg}) \sim_{E(\Pi_i^{\sf alg})} (2\pi i)^{dn_{i}} p(\Pi^{\sf alg}_i)$$
which is equivariant under the action of Aut$(\C/F^{Gal})$.\\\\
On the other hand, by Cor.\ \ref{isobaric sum Whittaker corollary}, we have
 $$p(\Pi') \sim_{E(\Pi')E(\phi)}  \prod_{1\leq i \leq k} p(\Pi_i^{\sf alg}) \prod_{1\leq i < j \leq k} L^S(1,\Pi_i\times\Pi^\vee_j)$$
which is also equivariant under the action of Aut$(\C/F^{Gal})$. Hence,
$$L^S(1,\Pi',{\rm As}^{(-1)^{n}})\sim_{\prod\limits_{i=1}^{k} E(\Pi^{\sf alg}_{i}) E(\Pi') E(\phi)}\prod_{i=1}^k (2\pi i)^{dn_{i}} \cdot p(\Pi')=(2\pi i)^{dn}p(\Pi').$$
We apply the Minimizing-Lemma, cf.\ Lem.\ \ref{minimize}, in order to shrink the base field of the relation to $E(\Pi')$. This shows the claim.

\end{proof}

\subsection{Critical values of Rankin-Selberg $L$-functions}\label{sect:main_gh}

Our second main theorem for special values provides a explicit refinement of Thm.\ \ref{thm:main-app}, revealing the archimedean factor $p(m,\Pi_\infty,\Pi'_\infty)$ -- extending Cor.\ \ref{prop:generalized grob-harr} -- also for non-cuspidal isobaric representations $\Pi'$ as an explicit power of $(2\pi i)$. \\\\ As before, we may choose some appropriate algebraic conjugate self-dual Hecke characters $\chi$ and $\chi'$ as in \S\ref{sect:reps}, such that $\Pi_{\chi,\infty}\cong \Pi_{\infty}$ and $\Pi_{\chi',\infty}\cong \Pi'_{\infty}$. We write $\Pi_{\chi'}=\Pi_{\chi',1}\boxplus...\boxplus\Pi_{\chi',k}$.

\begin{thm}\label{thm:generalized grob-harr}
Let $\Pi$ and $\Pi'$ be cohomological automorphic representations as in Thm.\ \ref{thm:main-app} and assume that their archimedean components $\Pi_\infty$ and $\Pi'_\infty$ are conjugate self-dual.
Let $\frac{1}{2}+m\in \emph{\textrm{Crit}}(\Pi\times\Pi')$ be a critical point. If $m=0$ we assume that there are algebraic conjugate self-dual Hecke characters $\chi$ and $\chi'$ as in \S\ref{sect:reps}, such that $\Pi_\chi$ is cuspidal and such that $(\Pi_\chi,\Pi_{\chi'})$ satisfies Hyp.\ \ref{cond} and moreover, that whenever $\mu$ or $\mu_{\chi',i}^{\sf alg}$ is not sufficiently regular, Hyp.\ \ref{hyp a 0} \& \ref{hyp a 1} hold for $\Pi$ resp.\ $\Pi_{\chi',i}^{\sf alg}$.  Then,
$$
L^{S}(\tfrac 12+m,\Pi \times \Pi') \ \sim_{E(\Pi)E(\Pi')}(2\pi i)^{mdn(n-1)-\frac{1}{2}d(n-1)(n-2)} p(\Pi)\ p(\Pi') \mathcal G(\omega_{\Pi'_f})
$$
which is equivariant under $\textrm{\emph{Aut}}(\C/F^{Gal})$.
\end{thm}
\begin{proof}
Let $\chi$ and $\chi'$ be appropriate algebraic conjugate self-dual Hecke characters as in \S\ref{sect:reps}, such that $\Pi_{\chi,\infty}\cong \Pi_{\infty}$ and $\Pi_{\chi',\infty}\cong \Pi'_{\infty}$. We may arrange that $\Pi_\chi$ is cuspidal and write $\Pi_{\chi'}=\Pi_{\chi',1}\boxplus...\boxplus\Pi_{\chi',k}$. Let $\frac{1}{2}+m\in$Crit$(\Pi\times\Pi')=$Crit$(\Pi_\chi\times\Pi_{\chi'})$. If $m\neq 0$, then our description of the set of critical points, cf.\ \S\ref{sect:critpts}, implies that the highest weights $\mu$ or $\mu_{\chi',i}^{\sf alg}$ are all sufficiently regular. In particular, the automorphic representations $\Pi_\chi$ and $\Pi_{\chi'}$ then satisfy the assumptions of Thm.\ \ref{thm:main-app} and Thm.\ \ref{generalized thm:gro-har-lap}. Hence, we obtain
\begin{equation}\label{f1}
\cfrac{L^{S}(\tfrac 12+m,\Pi_{\chi} \times \Pi_{\chi'})}{L^{S}(1,\Pi_{\chi},{\rm As}^{(-1)^{n}})L^{S}(1,\Pi_{\chi'},{\rm As}^{(-1)^{n-1}})}\ \sim_{E(\Pi_{\chi})E(\Pi_{\chi'})} \ \cfrac{p(m,\Pi_{\chi,\infty},\Pi_{\chi',\infty})}{(2\pi i)^{d(2n-1)}}
\end{equation}
which is equivariant under the action of Aut$(\C/F^{Gal})$. On the other hand, applying Prop.\ \ref{middle step down} and \ref{middle step up} shows that the same quotient satisfies the relation
\begin{equation}\label{f2}
\cfrac{L^{S}(\tfrac 12+m,\Pi_{\chi} \times \Pi_{\chi'})}{L^{S}(1,\Pi_{\chi},{\rm As}^{(-1)^{n}})L^{S}(1,\Pi_{\chi'},{\rm As}^{(-1)^{n-1}})}  \sim_{E(\chi)E(\chi')E(\phi)}(2\pi i)^{(\frac{1}{2}+m)dn(n-1)-dn^2}
\end{equation}
which is equivariant under the action of Aut$(\C/(LL')^{Gal})$. As $\tfrac12+m\neq\tfrac12$ we have $L^{S}(\tfrac 12+m,\Pi_{\chi} \times \Pi_{\chi'})\neq 0$, so we may combine \eqref{f1} and \eqref{f2} and obtain
$$p(m,\Pi_\infty,\Pi'_{\infty})\sim_{E(\Pi)E(\Pi')} (2\pi i)^{mdn(n-1)-\frac{1}{2}d(n-1)(n-2)}$$
by the Minimizing-Lemma, cf.\ Lem.\ \ref{minimize}. Hence, the result follows for $m\neq 0$ from applying Thm.\ \ref{thm:main-app} to $\Pi$ and $\Pi'$. If finally $m=0$, then our assumptions on $\Pi$ and $\Pi_{\chi',i}^{\sf alg}$ imply that one may argue as for the case $m\neq 0$. This completes the proof.
\end{proof}

\subsection{Quotients of critical $L$-values}

As a consequence of Thm.\ \ref{generalized thm:gro-har-lap} and Thm.\ \ref{thm:generalized grob-harr}, we obtain another two rationality-results, both for quotients of critical $L$-values, see Thm.\ \ref{cor:HR} and Thm.\ \ref{thm:intermediate} below. It is one of their advantages that they avoid any reference to bottom-degree Whittaker periods, but express the respective ratio of critical $L$-values purely in terms of powers of $(2\pi i)$.\\\\
Let us point out that the first of these theorems, Thm.\ \ref{cor:HR}, establishes the main result of \cite{harder-raghuram} for general CM-fields $F$, and a general pair of automorphic representations $(\Pi, \Pi')$ of $\GL_n(\A_F)\times\GL_{n-1}(\A_F)$ satisfying Thm.\ \ref{thm:intermediate}, as compared to the case of totally real fields $F^+$ and a pair of cuspidal cohomological representations $(\sigma,\sigma')$ of $\GL_n(\A_{F^+})\times\GL_{n'}(\A_{F^+})$ considered {\it ibidem}. While the second theorem, Thm.\ \ref{thm:intermediate}, will allow us to prove a version of the refined Gan--Gross--Prasad conjecture for unitary groups in \S\ref{sect:GGP} below. It is also closely connected to Deligne's conjecture for motivic $L$-functions, see Rem.\ \ref{rem:Deligne}.

\begin{thm}\label{cor:HR}
Let $\Pi$ and $\Pi'$ be as in Thm.\ \ref{thm:generalized grob-harr} and let $\frac{1}{2}+m,\tfrac12+\ell\in \emph{\textrm{Crit}}(\Pi\times\Pi')$ be two critical points. Whenever $L^S(\tfrac12+\ell, \Pi\times\Pi')$ is non-zero (e.g., if $\ell\neq 0$),
$$\frac{L^S(\tfrac12+m, \Pi\times\Pi')}{L^S(\tfrac12+\ell, \Pi\times\Pi')} \sim_{E(\Pi)E(\Pi')} (2\pi i)^{d(m-\ell)n(n-1)}.$$
and this relation is equivariant under the action of ${\rm Aut}(\C/F^{Gal})$. In particular, if $L^S(\tfrac32+m, \Pi\times\Pi')$ is non-zero (e.g., if $m\neq -1$), the quotient of consecutive critical $L$-values satisfies
$$(2\pi i)^{dn(n-1)}\frac{L^S(\tfrac12+m, \Pi\times\Pi')}{L^S(\tfrac32+m, \Pi\times\Pi')} \in E(\Pi)E(\Pi') .$$
\end{thm}
\begin{proof}
This follows directly from Thm.\ \ref{thm:generalized grob-harr}.
\end{proof}
This theorem also generalizes \cite{janus}, Thm.\ A, where an analogously explicit result has been proved (under different assumptions) for pairs of cuspidal representations $(\pi,\sigma)$. 

\begin{thm}\label{thm:intermediate}
Let $\Pi$ be a cohomological conjugate self-dual cuspidal automorphic representation of $\GL_n(\A_F)$ and let $\Pi'=\Pi_1\boxplus...\boxplus\Pi_k$ be a cohomological isobaric automorphic sum on $\GL_{n-1}(\Acm)$, fully induced from distinct conjugate self-dual cuspidal automorphic representations $\Pi_i$. Assume that the highest weight modules $\EE_\mu$ and  $\EE_{\mu'}$ of $\Pi$ and $\Pi'$ satisfy the piano-hypothesis Hyp.\ \ref{hypo:piano}.
Let $\frac{1}{2}+m\in \emph{\textrm{Crit}}(\Pi\times\Pi')$ be a critical point. If $m=0$ we assume that there are algebraic conjugate self-dual Hecke characters $\chi$ and $\chi'$ as in \S\ref{sect:reps}, such that $\Pi_\chi$ is cuspidal, $(\Pi_\chi,\Pi_{\chi'})$ satisfies Hyp.\ \ref{cond} and moreover, that whenever $\mu$ or $\mu_{\chi',i}^{\sf alg}$ is not sufficiently regular, Hyp.\ \ref{hyp a 0} \& \ref{hyp a 1} hold for $\Pi$ resp.\ $\Pi_{\chi',i}^{\sf alg}$.  Then,
$$\frac{L^S(\tfrac12+m, \Pi\times\Pi')}{L^S(1,\Pi,{\rm As}^{(-1)^n}) \ L^S(1,\Pi',{\rm As}^{(-1)^{n-1}})} \sim_{E(\Pi)E(\Pi')} (2\pi i)^{mdn(n-1)-dn(n+1)/2}.$$
and this relation is equivariant under the action of ${\rm Aut}(\C/F^{Gal})$.
\end{thm}

\begin{proof}
Let $\Pi$ and $\Pi'$ be as stated. By Thm.\ \ref{thm:generalized grob-harr}, see also Rem.\ \ref{remove Gauss sum},
$$
L^{S}(\tfrac 12+m,\Pi \times \Pi') \ \sim_{E(\Pi)E(\Pi')}(2\pi i)^{mdn(n-1)-\frac{1}{2}d(n-1)(n-2)} p(\Pi)\ p(\Pi'). \
$$
By Thm.\ \ref{generalized thm:gro-har-lap}, we have
$L^{S}(1,\Pi,{\rm As}^{(-1)^{n}})\ \sim_{E(\Pi)} \ (2\pi i)^{dn} p(\Pi)$ and $L^{S}(1,\Pi',{\rm As}^{(-1)^{n-1}})\ \sim_{E(\Pi')} \ (2\pi i)^{d(n-1)} p(\Pi')$. This shows the claim.
\end{proof}

\begin{rem}
From the proof we can see that the same strategy works as well for certain non-cuspidal $\Pi$, for example, if $\Pi$ is isobaric sum of Hecke characters.
\end{rem}

\begin{rem}[{\it Relation to Deligne's conjecture}]\label{rem:Deligne}
Due to the absence of our Whittaker periods, it is easiest to interpret Thm.\ \ref{thm:intermediate} from the perspective of Deligne's conjecture on critical values of motivic $L$-functions. Indeed, in Thm.\ \ref{thm:intermediate}, $s_0=\tfrac12+m$ is critical for $L(s,\Pi\times\Pi')$ and $s_0=1$ is critical for $L(s,\Pi,{\rm As}^{(-1)^{n}})$ $L(s,\Pi', {\rm As}^{(-1)^{n-1}})$ in the sense coined by Deligne, cf.\ \cite{deligne79}. Invoking the conjectural dictionary between automorphic representations $\Pi$ and $\Pi'$ and motives, there should hence be irreducible motives $\mathbb M$ and $\mathbb M'$ over $F$ whose attached Deligne periods capture the transcendental part of the respective $L$-value. More precisely, we have:
\begin{eqnarray*}
&&L^S(\tfrac12+m, \Pi\times\Pi')=L^S(m+n-1, \mathbb M\times \mathbb M') =L^S(0, \mathbb M\times \mathbb M'(m+n-1)), \\
&&L^S(1,\Pi,{\rm As}^{(-1)^n})= L^S(1,{\rm As}^{(-1)^n}(\mathbb M))=L^S(0,{\rm As}^{(-1)^n}(\mathbb M)(1)),\\
&&L^S(1,\Pi',{\rm As}^{(-1)^{n-1}})=L^S(1,{\rm As}^{(-1)^{n-1}}(\mathbb M'))=L^S(0,{\rm As}^{(-1)^{n-1}}(\mathbb M')(1)).
\end{eqnarray*}
Moreover, one can show that if $(\Pi,\Pi')$ satisfies the piano-hypothesis, then the Deligne periods are related to each other by the formula
$$c^{+}(\mathbb M\times \mathbb M'(m+n-1)) \sim  (2\pi i)^{mdn(n-1)-dn(n+1)/2}c^{+}( {\rm As}^{(-1)^n}(\mathbb M)(1)) c^{+}( {\rm As}^{(-1)^{n-1}}(\mathbb M')(1))$$
We refer to \S 1 of \cite{harris_adjoint}, when $\tr=\Q$, and to \S 2 of \cite{jie_michael} for general $\tr$. As a consequence, Thm.\ \ref{thm:intermediate} is in perfect fit with Deligne's conjecture, \cite[Conj.\ 2.8]{deligne79}.

One can also compare Thm.\ \ref{generalized thm:gro-har-lap} and Thm.\ \ref{thm:generalized grob-harr} with Deligne's conjecture, though the actual presence of Whittaker periods makes it trickier to interpret our formulas motivically. The difficulty relies in the problem to find a motivic analogue of our Whittaker periods: At least when $\Pi$ and $\Pi'$ descend to unitary groups of all signatures, one can define so-called {\it arithmetic automorphic periods} for these representations (cf.\ \cite{harris97}, \cite{jie-thesis}), which in fact have motivic analogues (cf.\ \S 4 of \cite{jie_michael}). The final bridge between Whittaker periods and arithmetic automorphic periods is then provided by \cite{grob-harr} and \cite{jie-thesis}. We remark that there is an archimedean factor left undecided in the underlying relations. By a strategy, similar to the one presented here, one can show however that this archimedean factor is also equivalent to a power of $2\pi i$. One can then compare the Whittaker periods with the Deligne periods.

\end{rem}

\section{Our main theorem on the refined GGP-conjecture for unitary groups}\label{sect:GGP}
\subsection{The framework}\label{sect:ggp}
Let $\E/\F$ be a field extension of number fields of degree $\dim_\F\E\leq 2$ and $c$ the unique automorphism of $\E$ which has $\F$ as fixed points $\E^{c=\triv}=\F$ (e.g., $\E=F$ and $\F=F^+$ from \S \ref{sect:fields}). Let $\V$ be a finite dimensional vector space over $\E$ and let $\<\cdot,\cdot\>:\mathcal V\times \V\ra\E$ be a non-degenerate, $c$-sesquilinear Hermitian pairing. The connected component of the identity of the group of isometries with respect to $(\V,\<\cdot,\cdot\>)$ is denoted $\G(\V)$ and a reductive algebraic group over $\F$ (e.g., $\V=V_n$ and $\G(\V)=H_n$ from \S \ref{sect:alggrp}). Not to interfere with low-rank cases, we will have to assume tacitly that $\dim_\E(\V)+[\E:\F]\geq 4$ (e.g, that $n\geq 2$ in the notation of \S\ref{sect:alggrp}).\\\\
Let $\W\subset \V$ be a non-degenerate subspace of $\V$ of odd codimension $\dim_\E(\W^\bot)=2r+1$, whose orthogonal complement contains an isotropic subspace $\mathcal X$ of dimension $r\geq 0$ (i.e., $\W$ is $r$-split). (Here, for reasons of precision, we assume that $\G(\W)$ is not split if $\dim_\E(\W)=2$.) 
We define $\mathscr P=\mathscr P_{\mathfrak F}$ to be the parabolic subgroup of $\G(\V)$, which stabilizes a fixed complete flag $\mathfrak F$ of $r+1$ isotropic $\E$-subspaces in $\mathcal X$ and let $\G(\W)$ be defined as above, replacing $\V$ by $\W$. Then there are natural inclusions $\G(\W)\hra \mathscr P_{\mathfrak F}\hra \G(\V)$, where $\G(\W)$ embeds into a Levi subgroup of $\mathscr P$, whence it acts naturally by conjugation on the unipotent radical $\mathscr N=\mathscr N_{\mathfrak F}$ of $\mathscr P$. We set $\mathscr H:= \G(\W)\rtimes \mathscr N$, which is again a natural subgroup of $\G(\V)$. For all the above we refer to \cite{ggp}, \S 2 and \S 12.\\\\
In what follows $\A=\A_\F$. We chose a generic (and hence by definition unitary) character
$$\psi_{\mathfrak F}=\otimes_{v}\psi_{\mathfrak F,v}:\mathscr N_{\mathfrak F}(\F)\backslash \mathscr N_{\mathfrak F}(\A)\ra\C^\times,$$
which is invariant under conjugation by $\G(\W)(\A)$ and define the form
$$\Psi_{\psi_{\mathfrak F}}(\varphi)(g):=\int_{\mathscr N_{\mathfrak F}(\F)\backslash \mathscr N_{\mathfrak F}(\A)}\varphi(n) \ \psi_{\mathfrak F}(ng)^{-1} \ dn,$$
for an automorphic form $\varphi$ of $\G(\V)(\A)$ and the Tamagawa measure $dn$ of $\mathscr N_{\mathfrak F}(\A)$. Since the domain of integration is compact, the integral converges absolutely. Now,  let $\pi_\V$ (resp.\ $\pi_\W$) be a cuspidal automorphic representation of $\G(\V)(\A)$ (resp.\ $\G(\W)(\A)$) and $\varphi\in\pi_\V$ (resp.\ $\varphi'\in\pi_\W$) be a cusp form. Then the {\it global period integral}
\begin{equation}\label{eq:p}
\mathcal P(\varphi,\varphi'):=\int_{\G(\W)(\F)\backslash\G(\W)(\A)} \Psi_{\psi_{\mathfrak F}}(\varphi)(g') \ \varphi'(g') \ dg'
\end{equation}
is absolutely convergent. Again, $dg'$ denotes the Tamagawa measure on $\G(\W)(\A)$.\\\\
Suppose now in addition that $\pi_\V$ and $\pi_\W$ are tempered at all places and let $S$ be any finite set of places containing all archimedean places and the places where $\pi_\V$ or $\pi_\W$ ramify. Then the partial $L$-function $L^S(s,\pi_\V\boxtimes\pi_\W)$ is defined with respect to the local Satake-parameters of $\pi_\V$ and $\pi_\W$ outside $S$ and the representation
$$R=\left\{\begin{array}{ll}
 {\rm St}\otimes{\rm St} & \textrm{if $\E=\F$} \\
 {\rm Ind}_{\widehat{\G(\V)\times\G(\W)}}^{^L(\G(\V)\times\G(\W))}[{\rm St}\otimes{\rm St}] & \textrm{if $[\E:\F]=2$}
\end{array}
\right. $$
of the $L$-group $^L(\G(\V)\times\G(\W))$. Here, ${\rm St}$ denotes the standard representation of the respective factor. We have to assume that this $L$-function allows a meromorphic continuation to whole $s$-plane.\\\\
Then, in the situation at hand, the GGP-conjecture asserts 
\begin{conj}[\cite{ggp}, Conj.\ 24.1]\label{conj:GGP}
Let $\pi_\V$ and $\pi_\W$ be tempered cuspidal automorphic representations of $\G(\V)(\A)$ resp.\ $\G(\W)(\A)$, which appear with multiplicity one in the cuspidal spectrum. Then the following statements are equivalent:
\begin{itemize}
\item[(i)] $L^S(\frac12,\pi_\V\boxtimes\pi_\W)\neq 0$ and $\dim_\C {\rm Hom}_{\mathscr H(\A)}[\pi_\V\otimes\pi_\W,\psi_{\mathfrak F}]=1$
\item[(ii)] $\mathcal P(\varphi,\varphi')\neq 0$ for some cusp forms $\varphi\in\pi_\V$ and $\varphi'\in\pi_\W$.
\end{itemize}
\end{conj}

We remark that, strictly speaking, this is an interpretation of the GGP-conjecture, because it assumes the (expected) holomorphy and non-vanishing for $s>0$ of the (still partly mysterious) local $L$-function at the ramified places. Moreover, in view of the focus of this paper, we restricted our attention to sesquilinear forms of sign 1, while the original GGP-conjecture allows sign -1 as well. On the other hand, however, \cite{ggp} deals only with quasisplit groups, a restriction, which we avoided.

\subsection{Refinements of the GGP-conjecture}\label{sect:ref_ggp}
In the last couple of years the GGP-Conjecture has undergone a series of increasingly general refinements, which shade a significant amount of new light on the original conjecture of GGP. As it will be of great importance for our major application to know them precisely, we have to recall them shortly. We define $L$-functions $L^S(s,\pi_\V,{\rm Ad})$ (resp.\ $L^S(s,\pi_\W,{\rm Ad})$) of $\pi_\V$ (resp.\ $\pi_\W$) with respect to the Satake parameters and the adjoint representation $R={\rm Ad}$ of the $L$-group $^L\G(\V)$ (resp.\ $^L\G(\W)$). Again, we shall suppose that these $L$-functions are meromorphically continuable to all $s\in\C$ and moreover, that they do not vanish at $s=1$ (Note that in this generality these $L$-functions don't come under the purview of \cite{shahidi_certainL}, Thm.\ 5.1).\\\\
Recall now that $dg'$ denotes the Tamagawa measure on $\G(\W)$. We choose, once and for all, local Haar measures $dg'_v$ at all places $v$ of $\F$, such that the following holds
\begin{enumerate}
\item $dg'=\prod_v dg'_v$
\item $vol_{dg'_v}(\mathscr O_v)\in\Q$ for all open subsets $\mathscr O_v$ of $\G(\W)(\F_v)$, if $v$ is non-archimedean
\item $vol_{dg'_v}(\mathscr K_v)=1$ for a hyperspecial maximal compact subgroup $\mathscr K_v$ of $\G(\W)(\F_v)$, for $\G(\W)$ unramified at $v$.
\end{enumerate}
Next, pin down a factorization $\pi_\V\cong\otimes'_v\pi_{\V,v}$  
which is compatible with the factorization of global and local inner products, i.e., for the usual $L^2$-product
$$\<\varphi_1,\varphi_2\>^{\V}_\A = \int_{\G(\V)(F)\backslash\G(\V)(\A)} \varphi_1(g) \ \overline{\varphi_2(g)} \ dg,$$
$dg$ denoting the Tamagawa measure, and the given inner products $\<\cdot,\cdot\>^\V_v$ on the Hilbert spaces underlying $\pi_{\V,v}$, we have
$$\<\varphi_1,\varphi_2\>^{\V}_\A = \prod_v \<\varphi_{1,v},\varphi_{2,v}\>^{\V}_v$$
for decomposable data $\varphi_i=\otimes'_v \varphi_{i,v}$, $i=1,2$. Likewise, we fix a factorization $\pi_\W\cong\otimes'_v\pi_{\W,v}$.\\\\
Let now $v$ be any non-archimedean place of $\F$. For any pair of (smooth) vectors $\varphi_v, \phi_v\in\pi_{\V,v}$, the integral
$$\int_{\mathscr N_{\mathfrak F}(\F_v)} \<\pi_{\V,v}(n_v)\varphi_v,\phi_v\>^\V_v \ \psi_{\mathfrak F,v}(n_v)^{-1} \ dn_v,$$
with $dn_v$ being the self-dual measure, stabilizes at some compact open subgroup $\mathscr N_0\subseteq \mathscr N_{\mathfrak F}(\F_v)$, i.e., for all compact open subgroups $\mathscr N_1\supseteq \mathscr N_0$, integration of $\<\pi_{\V,v}(n_v)\varphi_v,\phi_v\>^\V_v \ \psi_{\mathfrak F,v}(n_v)^{-1}$ over $\mathscr N_1$ and $\mathscr N_0$ gives rise to the same value, denoted
$$\int^{\sf st}_{\mathscr N_{\mathfrak F}(\F_v)} \<\pi_{\V,v}(n_v)\varphi_v,\phi_v\>^\V_v \ \psi_{\mathfrak F,v}(n_v)^{-1} \ dn_v.$$
See \cite{lapidmao}, \S 2.1, in particular Prop.\ 2.3. For our tempered cuspidal automorphic representations $\pi_\V\cong\otimes'_v\pi_{\V,v}$ and $\pi_\W\cong\otimes'_v\pi_{\W,v}$ and decomposable cuspidal automorphic forms $\varphi=\otimes'_v \varphi_{v}$ and $\varphi'=\otimes'_v \varphi'_{v}$ we may define
\begin{equation}\label{eq:alpha}
\alpha_v(\varphi_v,\varphi'_v):=\int_{\G(\W)(\F_v)} \left(\int^{\sf st}_{\mathscr N_{\mathfrak F}(\F_v)} \<\pi_{\V,v}(g'_v n_v)\varphi_v,\varphi_v\>^\V_v \ \psi_{\mathfrak F,v}(n_v)^{-1} \ dn_v\right) \ \overline{\<\pi_{\W,v}(g'_v)\varphi'_v,\varphi'_v\>}^\W_v \ dg'_v.
\end{equation}
By one of the main results in \cite{liu}, Thm.\ 2.1, $\alpha_v(\varphi_v,\varphi'_v)$ is absolutely convergent and $\alpha_v(\varphi_v,\varphi'_v)\geq 0$.\\\\
If $v$ is archimedean, let $\mathscr C_v$ (resp.\ $\mathscr C'_v$) be a maximal compact subgroup of $\G(\V)(\F_v)$ (resp.\ $\G(\W)(\F_v)$) and let $\varphi_v$ (resp.\ $\varphi'_v$) be a $\mathscr C_v$-finite (resp.\ $\mathscr C'_v$-finite) function in $\pi_{\V,v}$ (resp.\ $\pi_{\W,v}$). For such functions we define $\alpha_v(\varphi_v,\varphi'_v)$ as the Fourier transform of the tempered distribution given by the absolutely convergent integral (cf.\ \cite{liu}, Cor.\ 3.13 and \cite{sun}, Thm.\ 1.2)
$$\a_{\varphi_v,\varphi'_v}(n_v) :=\int_{\mathscr N_{\mathfrak F}(\F_v)_{-\infty}\times\G(\W)(\F_v)} \<\pi_{\V,v}(n_v n'_v g'_v)\varphi_v,\varphi_v\>^\V_v \ \overline{\<\pi_{\W,v}(g'_v)\varphi'_v,\varphi'_v\>}^\W_v \ dn'_v dg'_v,$$
(where $n_v\in \mathscr N_{\mathfrak F}(\F_v)$ and $\mathscr N_{\mathfrak F}(\F_v)_{-\infty}$ denotes the subset of matrices in $\mathscr N_{\mathfrak F}(\F_v)$, which are $0$ at those off-diagonal matrix-entries, on which $\psi_{\mathfrak F,v}$ is defined, cf.\ \cite{liu}, p.\ 155) evaluated at the generic character $\psi_{\mathfrak F,v}$,
\begin{equation}\label{eq:alpha_inf}
\alpha_v(\varphi_v,\varphi'_v):= \widehat{\a_{\varphi_v,\varphi'_v}}(\psi_{\mathfrak F,v}).
\end{equation}
By \cite{liu}, Thm.\ 2.1, $\alpha_v(\varphi_v,\varphi'_v)\geq 0$. If $\pi_{\V,v}$ is in the discrete series, then $\alpha_v(\varphi_v,\varphi'_v)$ is known to be absolutely convergent, see \cite{liu}, Prop.\ 3.15.\\\\
Set
$$\Delta_{\G(\V)}:= \left\{\begin{array}{ll}
 \prod_{i=1}^n \zeta_\F(2i) & \textrm{if $\E=\F$ and $\dim_\E \V=2n+1$} \\
  \prod_{i=1}^{n-1} \zeta_\F(2i)\cdot L(n,\chi_{\V,f}) & \textrm{if $\E=\F$ and $\dim_\E\V=2n$}\\
  \prod_{i=1}^n L(i,\epsilon_f^i) & \textrm{if $[\E:\F]=2$ (and $\dim_\E\V=n$)}
\end{array}
\right.$$
where $\chi_{\V}$ (resp.\ $\epsilon$) denotes the quadratic Hecke character $\F^\times\backslash\A^\times\ra\C^\times$ associated with the discriminant of $\<\cdot,\cdot\>_\V$ (resp.\ with the quadratic extension $\E:\F$ by class field theory) in the second (resp.\ in the last) line.\\\\
As a final ingredient, we invoke the theory of global Arthur packets for the square-integrable automorphic spectrum of $\G(\V)(\A)$ and $\G(\W)(\A)$. It is expected that $\pi_\V$ should be associated with a tempered elliptic Arthur parameter
$$\Psi(\pi_\V):\mathcal L_\F \ra ^L\G(\V),$$
uniquely determined by $\pi_\V$. We define $\mathcal S_{\pi_\V}:={\rm Cent}_{\widehat{\G(\V)}}({\rm Im}\Psi(\pi_\V))$ to be the centralizer of the image of $\Psi(\pi_\V)$ in the Langlands dual group $\widehat{\G(\V)}$. Analogously, one obtains $\mathcal S_{\pi_\W}:={\rm Cent}_{\widehat{\G(\W)}}({\rm Im}\Psi(\pi_\W))$. Both are a elementary 2-abelian groups.\\\\
Liu's refinement of the GGP-conjecture now provides a comparison of two adelic pairings, the key-ingredient of this comparison being that one of them is defined {\it ad hoc} globally (by \eqref{eq:p}) while the other is only defined {\it ex post} globally (by forming the product over all places $v$ of the integrals \eqref{eq:alpha}). Here is Liu's conjecture 

\begin{conj}[\cite{liu}, Conj.\ 2.5]\label{conj:Liu}
Let $\pi_\V\cong\otimes'_v\pi_{\V,v}$ (resp.\ $\pi_\W\cong\otimes'_v\pi_{\W,v}$) be a tempered cuspidal automorphic representation of $\G(\V)(\A)$ (resp.\ $\G(\W)(\A)$) coming together with a fixed tensor product factorization, which is compatible with the factorization of global and local inner products, and appearing with multiplicity one in the cuspidal spectrum. Let $S$ be any finite set of places of $\F$, containing the archimedean ones and such that $\pi_\V$, $\pi_\W$ and $\psi_{\mathfrak F}$ are unramified outside $S$. Then for all decomposable $\mathscr C_\infty$-finite (resp.\ $\mathscr C'_\infty$-finite) smooth functions $\varphi=\otimes'_v\varphi_{v}\in\pi_\V$ resp.\ $\varphi'=\otimes'_v\varphi'_{v}\in\pi_\W$,
\begin{itemize}
\item the Fourier transform $\alpha_v(\varphi_v,\varphi'_v)$ is absolutely convergent for all archimedean places $v$,
\item $\alpha_v\not\equiv 0$ if and only if $\dim_\C{\rm Hom}_{\mathscr H(\F_v)}[\pi_{\V,v}\otimes\pi_{\W,v},\psi_{\mathfrak F, v}]=1$ for all $v$
\end{itemize}
and one obtains the identity
$$|\mathcal P(\varphi,\varphi')|^2 = \frac{1}{|\mathcal S_{\pi_\V} |\cdot |\mathcal S_{\pi_\W}|} \ \frac{\Delta_{\G(\V)} \ L^S(\tfrac12, \pi_\V\boxtimes\pi_\W)}{L^S(1,\pi_\V,{\rm Ad}) \ L^S(1,\pi_\W,{\rm Ad})} \ \prod_{v\in S} \alpha_v(\varphi_v,\varphi'_v).$$
\end{conj}
Again, our Conj.\ \ref{conj:Liu} amounts to a ``sanded'' version of Liu's original conjecture \cite{liu}, Conj.\ 2.5: On the one hand, we believe it is more convenient to simply assume that our cusp forms are tempered and appear with multiplicity one. This emulates Liu's assumption of being {\it almost locally generic}, but also has the advantage that it is (conjecturally) even less restrictive than his original genericity-supposition and avoids moreover all difficulties arising from questions of convergence of $\alpha_v(\varphi_v,\varphi'_v)$ at non-archimedean places. On the other hand, we have to assume the well-expected local properties of our $L$-functions at $v\in S$ of being holomorphic and non-zero for $s>0$.

Specifying on the data entering Conj.\ \ref{conj:Liu}, one retrieves the older conjectures of Ichino-Ikeda, \cite{ichino_ikeda}, Conj. 2.1 and N.\ Harris, \cite{neil_harris}, Conj.\ 1.3:

\begin{conj}[Ichino-Ikeda]\label{conj:II}
This is Conj.\ \ref{conj:Liu} with $r=0$ and $\E=\F$.
\end{conj}

\begin{conj}[N. Harris]\label{conj:NH}
This is Conj.\ \ref{conj:Liu} with $r=0$ and $[\E:\F]=2$.
\end{conj}

\subsection{An application of Thm.\ \ref{thm:intermediate} - an algebraic version of the refined GGP-conjecture}
It is the goal of this section to prove an algebraic version of Liu's refined GGP-conjecture. More precisely, recall the $c$-hermitian spaces $\V=V_n/F$ and attached unitary groups $\G(\V)=H_n=U(V_n)/F^+$ from \S \ref{sect:alggrp}. By an {\it $E(\pi)$-rational function} $\varphi=\otimes'_v\varphi_{v}\in\pi$ we mean a decomposable function whose $\pi_f$-component lies in the fixed  $E(\pi)$-structure on $\pi_f$, chosen in \S \ref{sect:twists}, while its attached matrix coefficients at the archimedean places define an element of the affine algebra $E(\pi)(H_n)$ of the algebraic group $H_n$. Likewise for $\pi'$. We may now prove

\begin{thm}\label{thm:mainmain}
Let $\pi\cong\otimes'_v\pi_{v}$ (resp.\ $\pi'\cong\otimes'_v\pi'_{v}$) be a cohomological tempered cuspidal automorphic representation of $H_n(\A_{F^+})=U(V_n)(\A_{F^+})$ (resp.\ $H_{n-1}(\A_{F^+})=U(V_{n-1})(\A_{F^+})$), coming together with a fixed tensor product factorization, which is compatible with the factorization of global and local inner products. Let $S$ be any finite set of places of $F^+$, containing the archimedean ones and such that $\pi$ and $\pi'$ are unramified outside $S$. Assume moreover that the quadratic base change $BC(\pi)=\Pi$ is a cohomological cuspidal automorphic representation $\Pi$ of $\GL_n(\A_F)$ as in \S \ref{sect:pi} and that the quadratic base change $BC(\pi')=\Pi'$ is a cohomological isobaric automorphic representation $\GL_{n-1}(\A_F)$ as in \S \ref{sect:Eisen}.

\begin{enumerate}
\item If $H_{n,\infty}$ and $H_{n-1,\infty}$ are compact and $\Pi$ and $\Pi'$ satisfy the conditions of Thm.\ \ref{thm:intermediate}, then for all smooth $E(\pi)$-rational (resp.\ $E(\pi')$-rational) functions $\varphi=\otimes'_v\varphi_{v}\in\pi$ (resp.\ $\varphi'=\otimes'_v\varphi'_{v}\in\pi'$),
\begin{equation}\label{eq:main}
|\mathcal P(\varphi,\varphi')|^2 \sim_{E(\pi) E(\pi')} \frac{\Delta_{H_n} \ L^S(\tfrac12, \pi\boxtimes\pi')}{L^S(1,\pi,{\rm Ad}) \ L^S(1,\pi',{\rm Ad})} \ \prod_{v\in S} \alpha_v(\varphi_v,\varphi'_v)
\end{equation}
where $E(\pi)$ or $E(\pi')$ are the number fields defined in \S\ref{sect:twists}. \\
\item If $H_{n,\infty}$ or $H_{n-1,\infty}$ are non-compact, but $\EE_\mu$ and $\EE_{\mu'}$ satisfy the piano-hypothesis, cf.\ Hypothesis \ref{hypo:piano}, then the same conclusion holds trivially for all smooth $C_\infty$- (resp.\ $C'_\infty$--)finite decomposable functions in $\pi$ (resp.\ $\pi'$).
\end{enumerate}
\end{thm}

\begin{rem}
The factor $q\in E(\pi) E(\pi')$ used in relating the left- and the right-hand-side of \eqref{eq:main} is independent of the cusp forms $\varphi$ and $\varphi'$. However, as it is obvious by the definition of our relation ``$\sim_{E(\pi) E(\pi')} $'', see Def.\ \ref{def:rel}, our theorem cannot detect whether or not one side in \eqref{eq:main} is non-zero, but rather compensates this defect: If one side of \eqref{eq:main} vanishes, then the theorem holds by brute force, multiplying the respective other side by $0\in E(\pi) E(\pi')$. {\it The non-trivial assertion of our theorem is hence in fact about the case when both sides of the relation \eqref{eq:main} do not vanish: Then they are linked by a non-zero number $q$ in the concrete number field $E(\pi) E(\pi')$, $q$ being furthermore independent of $\varphi$ and $\varphi'$.}
\end{rem}

Before we prove Thm.\ \ref{thm:mainmain} we state two further important remarks:

\begin{rem}[{\it Well-definedness}]\label{rem:welldef}
Before we put our main theorem into relation with the recent literature on Conj.\ \ref{conj:Liu} (and Conj.\ \ref{conj:NH}), let us first remark on the various objects in Thm.\ \ref{thm:mainmain}, in particular the quantities in \eqref{eq:main}, being well-defined. Firstly, the results in \cite{labesse-book}, Cor.\ 5.3 and \cite{morel}, Prop.\ 8.5.3, as jointly refined by Shin, \cite{shin}, Thm.\ 1.1, show that quadratic base change $BC$ is well-defined and exists for all unitary groups $H_n$ and $H_{n-1}$ and representations $\pi$ and $\pi'$ as above (at least if $F=\mathcal K F^+$, $\mathcal K$ an imaginary quadratic field; the general case being a consequence of \cite{KMSW}, Thm.\ 1.7.1 and their Rem.\ 1.7.2 right below). Moreover, by the description of its image, it makes sense to specify the properties of the base change lifts $BC(\pi)$ and $BC(\pi')$ as we did in the statement of Thm.\ \ref{thm:mainmain}, namely to assume that $BC(\pi)$ is (cohomological) cuspidal (as in \S\ref{sect:pi}) and that $BC(\pi')$ is a (cohomological) isobaric sum of cuspidal automorphic representations (as made precise in \S \ref{sect:Eisen}).\\
Secondly, this implies that $L^S(s,\pi\boxtimes\pi')=\prod_{1\leq i\leq k} L^S(s,\Pi\times\Pi_i)$ is holomorphic at $s=\tfrac12$ as well as that the product
$$L^S(s,\pi,{\rm Ad}) \cdot L^S(s,\pi',{\rm Ad})=L^S(1,\Pi,{\rm As}^{(-1)^{n}}) \cdot L^S(1,\Pi',{\rm As}^{(-1)^{n-1}})$$
is holomorphic and non-vanishing at $s=1$, see Cor.\ \ref{cor:asai-nonvan}. In particular, the quotient of $L$-values in our algebraic relation \eqref{eq:main} makes sense without any assumptions.\\
Thirdly, we recall that the absolute convergence of $\alpha_v(\varphi_v,\varphi'_v)$ at archimedean $v$ -- as demanded by Conj.\ \ref{conj:Liu} -- follows from the fact that a (by assumption) tempered and cohomological representation of a unitary group must be in the discrete series, cf.\ \S\ref{sect:pi}.\\
This demonstrates that all objects and quantities in Thm.\ \ref{thm:mainmain} exist and are well-defined. On a final remark, let us point out that both $\pi$ and $\pi'$ enjoy multiplicity one in the square-integrable automorphic spectrum (combine \cite{KMSW}, \S 0.3.3, Thm.\ 1.7.1, Rem.\ 1.7.2 and Thm.\ 5.0.5).
\end{rem}

\begin{rem}[{\it A comparison of our theorem with results on the refined GGP-conjecture in the literature}]
As our algebraicity-result is coarser in its very statement, than Conj.\ \ref{conj:Liu} resp.\ Conj.\ \ref{conj:NH}, we feel that a remark is in order to put our result in relation with the results of recent literature. 
Zhang (\cite{zhang}, Thm.\ 1.2.(2)) has established the equality
$$|\mathcal P(\varphi,\varphi')|^2 = \frac{c_{\pi_\infty,\pi'_\infty}}{4} \frac{\Delta_{H_n} \ L^S(\tfrac12, \Pi\times\Pi')}{L^S(1,\Pi,{\rm As}^{(-1)^{n}}) \ L^S(1,\Pi',{\rm As}^{(-1)^{n-1}})} \ \prod_{v\in S} \alpha_v(\varphi_v,\varphi'_v),$$
whenever $\G(\V)\times\G(\W)$ is compact at every archimedean place $v$ of $F^+$. Here, $c_{\pi_\infty,\pi'_\infty}$ is a certain constant, only depending on the archimedean components of $\pi$ and $\pi'$.\\\\
Zhang's theorem is built on a list of conditions on $\pi_\V=\pi$ and $\pi_\W=\pi'$ (called {\bf RH(I)} and {\bf RH(II)}, p.\ 544). In a series of preprints Beuzart-Plessis has been able to significantly relax Zhang's conditions and in fact sharpen his result, assuming that the cuspidal automorphic representations $\pi$ and $\pi'$ are supercuspidal at one non-archimedean place. See \cite{beuzard-plessis16}, Thm.\ 1.0.4 and \cite{beuzard-plessis18}, Thm.\ 5. The analogous assumption of supercuspidality also appears in the work of Chaudouard-Zydor \cite[Thm.\ 1.1.6.2]{chau_zydor} and Xue \cite[Thm.\ 1.1]{xue}. If $\Pi\otimes\Pi$ is cuspidal, then the refined GGP-conjecture has been shown very recently in \cite{BLZZ}, Thm.\ 1.9.\\\\ It is important to notice that our result, Thm.\ \ref{thm:mainmain}, avoids any assumption of supercuspidality of $\pi\otimes\pi'$ at any place, nor do we have to assume that both lifts $\Pi$ and $\Pi'$ are cuspidal. In this regard, our result on the refined GGP-conjecture, Thm.\ \ref{thm:mainmain}, may be viewed as a complementary theorem to the above mentioned results, applying to a different (broader) class of cuspidal representations $\pi$ and $\pi'$. As we have learned very recently, in \cite{BCZ}, which quotes our Thm.\ \ref{thm:mainmain} above, the cuspidality assumption on $\Pi\otimes\Pi$ will finally be removed.
\end{rem}

\subsection{Proof of Thm.\ \ref{thm:mainmain}} \label{final section}

Recall $\EE_\mu$ and $\EE_{\mu'}$, the coefficient modules with respect to which $\Pi$, resp.\ $\Pi'$ are of non-trivial cohomology. For simplicity, put for each $v\in S_\infty$, $\lambda_v:=(\mu_{\iota_v,1}, ...,\mu_{\iota_v,n})$ (resp.\ $\lambda'_v:=(\mu'_{\iota_v,1}, ...,\mu'_{\iota_v,n-1})$ ) and let $\cF_\lambda$ (resp.\ $\cF_{\lambda'}$) be the irreducible algebraic representation of $H_{n,\infty}$ (resp.\ $H_{n-1,\infty}$) given by the highest weight $\lambda:=(\lambda_v)_{v\in S_\infty}$ (resp.\ $\lambda':=(\lambda'_v)_{v\in S_\infty}$) as in \S\ref{sect:finitereps}. Then $\cF_\lambda$ (resp.\ $\cF_{\lambda'}$) is the highest weight module with respect to which $\pi_\infty$ (resp.\ $\pi'_\infty$) is cohomological, cf.\ \cite{labesse-book}, Cor. 5.3.\\

We assume at first that  $H_{n,\infty}$ and $H_{n-1,\infty}$ are compact.
\begin{lem}
If $\EE_\mu$ and $\EE_{\mu'}$ do not satisfy the piano-condition, Hyp.\ \ref{hypo:piano}, then $\alpha_v(\varphi_v,\varphi'_v)=0$ for all archimedean places $v$ and functions $\varphi_v\in\pi_v$ and $\varphi'_v\in\pi'_v$.
\end{lem}
\begin{proof}
This an easy consequence of the branching law to which the piano hypothesis is equivalent, described, in \cite{goodman-wall}, Thm.\ 8.1.1: If $\EE_\mu$ and $\EE_{\mu'}$ do not satisfy the piano-condition, then the branching law says that Hom$_{H_{n-1}(F^+_v)}[\cF_{\lambda_{v}}\otimes \cF_{\lambda'_{v}},\C]=0$ for all $v\in S_\infty$. Compactness of $H_{n,\infty}$ and $H_{n-1,\infty}$ implies that $\pi_v\cong \cF^{\sf v}_{\lambda_{v}}$ and $\pi'_v\cong \cF^{\sf v}_{\lambda'_{v}}$ , so by dualizing also Hom$_{H_{n-1}(F^+_v)}[\pi_v\otimes\pi'_v,\C]=0$. As $\alpha_v\in$ Hom$_{H_{n-1}(F^+_v)}[\pi_v\otimes\pi'_v,\C]$, this shows the claim.
\end{proof}
Therefore, if $H_{n,\infty}$ and $H_{n-1,\infty}$ are compact, but $\EE_\mu$ and $\EE_{\mu'}$ do not satisfy the piano-condition, then our main theorem, Thm.\ \ref{thm:mainmain}, trivially follows by multiplying the left hand side of relation \eqref{eq:main} with $q=0$.\\\\
Hence, let us now consider the non-trivial case, when $\EE_\mu$ and $\EE_{\mu'}$ do satisfy the piano-condition. Let $\mathcal A_{cusp}(H_{n},\cF_\lambda)$ be the space of automorphic (and hence, by compactness of $H_{n,\infty}$ automatically) cuspidal functions, which transform by $\cF^{\sf v}_\lambda$ on the right. Again by compactness of $H_{n,\infty}$, the restriction of functions $\phi\mapsto\phi|_{H_n(\A_f)}$ defines a natural isomorphism
$$R_n: \mathcal A_{cusp}(H_{n},\cF_\lambda) \ira H^0(S_{H_n},\cF_\lambda),$$
the right hand side being defined in \S\ref{sect:rat}. Obviously, the analogous construction works for $H_{n-1}$, defining an isomorphism $R_{n-1}$. Then, it is proved in \cite{harris_adjoint} that one obtains the following three algebraicity results

\begin{prop}[\cite{harris_adjoint}, Cor.\ 2.5.4]\label{prop:comp1}
Let $\varphi\in\pi$ and $\varphi'\in\pi'$ be chosen such that they map via $R_n$ (resp.\ $R_{n-1}$) into the natural $E(\pi)$- (resp.\ $E(\pi')$-)structure of $H^0(S_{H_n},\cF_\lambda)$ (resp.\ $H^0(S_{H_{n-1}},\cF_{\lambda'})$), defined in \S\ref{sect:rat}. Then
$$|\mathcal P(\varphi,\varphi')|^2\in E(\pi)\cdot E(\pi').$$
\end{prop}

\begin{prop}\label{prop:comp2}
If $v\in S/S_\infty$ and $\varphi_v\in\pi_v$ and $\varphi'_v\in\pi'_v$ are chosen such that the lie in the natural $E(\pi)$- (resp.\ $E(\pi')$-) structure of $\pi_v$ (resp.\ $\pi'_v$), induced by the factorization $\pi_f\ira\otimes'_{v\notin S_\infty} \pi_v$ (resp.\ $\pi'_f\ira\otimes'_{v\notin S_\infty} \pi'_v$ ), fixed in Thm.\ \ref{thm:mainmain}. Then
$$\alpha_v(\varphi_v,\varphi'_v)\in E(\pi)\cdot E(\pi').$$
\end{prop}
\begin{proof}
This is \cite{harris_adjoint}, Lem.\ 4.1.5.1 together with the fact that for all $v\in S/S\infty$
$$\frac{L(\tfrac 12,\Pi_v \times \Pi'_v)}{ L(1,\Pi_v,{\rm As}^{(-1)^{n}}) \ L(1,\Pi'_v,{\rm As}^{(-1)^{n-1}})}\in \Q(\Pi)\Q(\Pi')\subseteq \Q(\pi_f)\Q(\pi'_f)\subset E(\pi) E(\pi'),$$ see \cite{raghuram-imrn} Prop.\ 3.17 and \cite{grob_harris_lapid} \S 6.4.
\end{proof}

\begin{prop}[\cite{harris_adjoint}, Cor.\ 4.1.4.3]\label{prop:comp3}
For all $\varphi_\infty=\otimes_{v\in S_\infty}\varphi_v\in\pi_\infty=\otimes_{v\in S_\infty}\pi_v$ and $\varphi'_\infty=\otimes_{v\in S_\infty}\varphi'_v\in\pi'_\infty=\otimes_{v\in S_\infty}\pi'_v$, whose attached matrix coefficients define an element of the affine algebra $E(\pi)(H_n)$ of the algebraic group $H_n$ (resp.\ $E(\pi')(H_{n-1})$ of $H_{n-1}$),
$$\alpha_\infty(\varphi_\infty,\varphi'_\infty)=\prod_{v\in S_\infty}\alpha_v(\varphi_v,\varphi'_v) \in E(\pi)\cdot E(\pi').$$
\end{prop}

\begin{cons}
Thm.\ \ref{thm:mainmain} holds if $H_{n,\infty}$ and $H_{n-1,\infty}$ are compact.
\end{cons}
\begin{proof}
Recall that $\Delta_{H_n}=\prod_{j=1}^n L(i,\varepsilon_f^j)$. By (\ref{Dedekind}) and (\ref{quadratic}) we know that if $j\geq 1$ is even then $L(j,\varepsilon_f^j)\sim_{F^{Gal}} (2\pi i)^{dj}$, and if $j\geq 1$ is odd then $L(j,\varepsilon_f^j)\sim_{F^{Gal}} (2\pi i)^{dj}$. Hence $\Delta_{H_n} \sim_{F^{Gal}} (2\pi i)^{dn(n+1)/2}$.\\
Invoking the three propositions, Prop.\ \ref{prop:comp1}, Prop.\ \ref{prop:comp2} and Prop.\ \ref{prop:comp3}, Thm.\ \ref{thm:mainmain} finally follows from Thm.\ \ref{thm:intermediate}. See also Rem.\ \ref{rem:welldef}.
\end{proof}


Now if $H_{n,\infty}$ or $H_{n-1,\infty}$ is non-compact, but $\EE_\mu$ and $\EE_{\mu'}$ satisfy the piano-condition, then we know by the branching law, \cite{goodman-wall}, Thm.\ 8.1.1, that the tempered representation $\pi_{\infty}\otimes \pi'_{\infty}$ is distinguished for the pair of compact unitary groups. But by the results in \cite{plessis}, there is at most one pair of unitary group such that $\pi_{\infty}\otimes \pi'_{\infty}$ is distinguished. In particular, the representation $\pi_{\infty}\otimes \pi'_{\infty}$ can not be distinguished for the pair $(H_{n-1,\infty},H_{n,\infty})$. Hence, $\alpha_v(\varphi_v,\varphi'_v)=0$ for all $v\in S_\infty$ and all $\varphi_v\in\pi_v$, $\varphi'_v\in\pi'_v$ and Thm.\ \ref{thm:mainmain} is trivially true.

\bigskip

\vskip 10pt
\footnotesize
{\sc Harald Grobner: Fakult\"at f\"ur Mathematik, University of Vienna, Oskar--Morgenstern--Platz 1, A-1090 Vienna, Austria.}
\\ {\it E-mail address:} {\tt harald.grobner@univie.ac.at}

\vskip 10pt

\footnotesize
{\sc Jie Lin: Fakult\"at f\"ur Mathematik, University of Duisburg-Essen, Mathematikcarr\'ee, Thea-Leymann-Stra\ss e 9, D-45127 Essen, Germany.}
\\ {\it E-mail address:} {\tt linjie@ihes.fr}

\end{document}